\documentclass[11pt]{article}
    \usepackage{url} 
    \usepackage{verbatim}
    \usepackage[titletoc]{appendix}    	
    \usepackage{caption3}
    \usepackage{subfigure}
    \usepackage{float}
\usepackage[inkscapelatex=false]{svg}
         \usepackage{bm}
	\usepackage{amsmath}
	\usepackage{amsthm}
	\usepackage{amsfonts}
	\usepackage{multirow}
    \usepackage{nicefrac}
    \usepackage{longtable}
    \usepackage{color}
    \usepackage{graphicx, amssymb,graphics}
  \usepackage{fancyhdr}
  \usepackage{algorithm}
  \usepackage{algorithmic}
\usepackage{cite}
\allowdisplaybreaks[4]
\usepackage{hyperref}

\usepackage{indentfirst} 
\newtheorem{theorem}{Theorem}[section]
\newtheorem{lemma}[theorem]{Lemma}
\newtheorem{remark}[theorem]{Remark}

\newcommand{\xqedhere}[2]{%
\rlap{\hbox to#1{\hfil\llap{\ensuremath{#2}}}}}

\makeatletter
\@addtoreset{equation}{section}
\makeatother

\newcommand{\R}{{\if mm {\rm I}\mkern -3mu{\rm R}\else \leavevmode
\hbox{I}\kern -.17em\hbox{R} \fi}}

\theoremstyle{plain}

\newtheorem{prop}{Proposition}[section]

\newcommand{\bv}{{\mathbf v}}

\def\bx{\bm{x}}
\def\bX{\bm{X}}

\newcommand{\Grad}[1]{\nabla #1}

		\title{ Adaptive time-stepping and maximum-principle preserving Lagrangian schemes for gradient flows }
\date{\today}
	
\begin{document}

\author{
		Qianqian Liu\thanks{School of Mathematical Sciences; Fudan University, Shanghai, China 200433 ({\tt qianqianliu21@m.fudan.edu.cn})} 
				\and
			Wenbin Chen\thanks{
		Shanghai Key Laboratory for Contemporary Applied Mathematics, School of Mathematical Sciences; Fudan University, Shanghai, China 200433 ({\tt wbchen@fudan.edu.cn})} 
		\and
	Jie Shen\thanks{Eastern Institute of Technology, Ningbo, Zhejiang, China 315200 ({\tt jshen@eitech.edu.cn})} 
	\and 
		Qing Cheng\thanks{Key Laboratory of Intelligent Computing and Application (Tongji University), Ministry of Education, Department of Mathematics; Tongji University, Shanghai, China 200092 ({\tt  qingcheng@tongji.edu.cn})} 
}
	\maketitle
	\begin{abstract}{ 			
	We develop in this paper an adaptive time-stepping approach for gradient flows with distinct treatments for conservative and non-conservative dynamics. 
	For the non-conservative gradient flows  in Lagrangian coordinates, we propose a modified formulation augmented by auxiliary terms to guarantee positivity of the determinant, and  prove that the corresponding adaptive second-order Backward Difference Formulas (BDF2) scheme preserves energy stability and the maximum principle under the time-step ratio constraint  $0<r_n\le r_{\max}\le\frac{3}{2}$. 
	On the other hand,  for the conservative Wasserstein gradient flows in Lagrangian coordinates,  we propose an  adaptive BDF2 scheme  which is shown to be  energy dissipative, and  positivity preserving under the time-step ratio constraint  $0<r_n\le r_{\max}\le\frac{3+\sqrt{17}}{2}$ in 1D and $0<r_n\le r_{\max}\le \frac{5}{4}$ in 2D, respectively. 
	We also present ample numerical simulations in 1D and 2D  to validate the efficiency and accuracy of the proposed schemes.

	
		}
	\end{abstract}

	\section{Introduction}

	We consider in this paper adaptive time-stepping Lagrangian approaches for gradient flows with  conservative and non-conservative dynamics. Specifically, the non-conservative models require the enforcement of determinant positivity through auxiliary operators, while the conservative systems need  mechanisms for mass conservation.
	
	For the non-conservative system characterized by energy functional $E(\rho)$ and positive mobility $\mathcal{M}(\rho) > 0$, the governing equation takes the form:
	\begin{equation}\label{eq:non}
		\partial_t\rho = -\mathcal{M}(\rho)\frac{\delta E}{\delta \rho},
	\end{equation}
	where the total mass of $\rho$ is usually not conserved, i.e., $\frac{\mathrm{d}}{\mathrm{d}t}\int_{\Omega}\rho(x,t)\mathrm{d}x \neq 0$. Notable examples include the Allen-Cahn equation \cite{allen1979microscopic}, Burgers-Huxley and Burgers-Fisher equations \cite{wazwaz2009partial}. 
	Through the non-conservative transport framework \cite{cheng2020new}
	\begin{align}\label{eq:non flow}
		\rho_t + \bm{v}\cdot\nabla\rho = 0,
	\end{align}
	the original model \eqref{eq:non} can be reformulated in Eulerian coordinates as:
	\begin{align}\label{eq:ac eu}
		\bm{v}\cdot\nabla\rho = \mathcal{M}(\rho)\frac{\delta E}{\delta \rho}.
	\end{align}
	This Eulerian formulation will subsequently be transformed into an equivalent Lagrangian representation.
	
	On the other hand, a conservative system is governed by
	\begin{equation}\label{eq:wd}
		\partial_t\rho = \nabla\cdot\left(\mathcal{M}(\rho)\nabla\frac{\delta E}{\delta\rho}\right),
	\end{equation}
	which preserves mass, i.e., $\frac{\mathrm{d}}{\mathrm{d}t}\int_{\Omega}\rho(x,t)\mathrm{d}x = 0$. This class of gradient flows  includes well-known models such as the Cahn-Hilliard equation \cite{cahn1958free}, Porous-Medium equation (PME) \cite{vazquez2007porous}, 
    Poisson-Nernst-Planck system \cite{eisenberg2010energy}, 
    and Keller-Segel equation \cite{keller1970initiation}. The conservative model \eqref{eq:wd} can be cast as a  continuity equation in the Eulerian framework \cite{liu2019energetic,liu2020variational,liu2020lagrangian}:
	\begin{align}\label{eq:con flow}
		\partial_t\rho + \nabla\cdot(\rho\bm{v}) = 0,
	\end{align}
	with 
	\begin{align}
		\rho\bm{v} = -\mathcal{M}(\rho)\nabla\frac{\delta E}{\delta \rho}.
	\end{align}
	This conservative formulation will likewise be transformed into its Lagrangian counterpart in subsequent analysis.

A typical example of a non-conservative phase-field model is the Allen-Cahn equation. Extensive research has been dedicated to developing numerical schemes in the Eulerian framework that preserve its key physical properties, especially the energy dissipation law and the maximum bound principle (MBP) \cite{feng2003numerical,shen2010numerical,du2021maximum,du2019maximum,ju2022generalized,ju2022stabilized,li2020arbitrarily,akrivis2022error,akrivis2019energy,hou2023linear,hou2023implicit}.
In contrast to Eulerian methods, Lagrangian approaches offer distinct advantages in interface tracking. The variational Lagrangian framework developed in \cite{liu2020variational} allows  for efficient computation of equilibrium states in Allen-Cahn type models while maintaining sharp interfaces and addressing free boundary problems, while the work  \cite{cheng2020new} introduced a Lagrangian formulation via a non-conservative transport equation, and  rigorously demonstrated the preservation of energy dissipation and MBP. 
	
On the other hand, if $\mathcal{M}(\rho)=\rho$, the conservative model \eqref{eq:wd} reduces to a Wasserstein gradient flow.
Recent advancements in direct numerical methods for Wasserstein gradient flows emphasize the preservation of three fundamental properties: mass conservation, positivity preservation, and energy dissipation \cite{shen2021unconditionally, liu2020lagrangian, li2016maximum, cheng2013positivity}. These methods have been successfully applied to various physical systems, including the porous media equation \cite{zhang2009numerical, duan2019pme, jang2015high}, the Fokker-Planck equation \cite{liu2016entropy, duan2021structure}, the Poisson-Nernst-Planck system \cite{shen2021unconditionally, liu2021positivity, huang2021bound}, and the Keller-Segel model \cite{liu2018positivity, shen2020unconditionally, wang2022fully}.
On the other hand, numerical methods based on the celebrated Jordan-Kinderlehrer-Otto (JKO) scheme \cite{jordan1998variational} have been extensively studied in the literature \cite{carrillo2015finite, carrillo2022primal, carrillo2018lagrangian, li2020fisher, benamou2016augmented, liu2023dynamic, cances2020variational}. A recent study by \cite{cheng2024flow} introduces an innovative flow dynamic approach that reformulates the constrained minimization problem into an unconstrained framework. This transformation facilitates a first-order Lagrangian scheme while preserving the essential properties of the original problem.
	
Since dynamics of gradient flows often behave very differently at different times.
Significant progress has been made in developing adaptive strategies tailored to enhance computational efficiency across a range of gradient flows \cite{calvo1990zero,zhang2013adaptive,qiao2011adaptive,chen2019second,huang2020parallel,hou2023implicit,cao2024exponential}.
	A key advancement is presented in \cite{liao2020second,liao2020energy}, where a novel kernel recombination technique enables the design of nonuniform BDF2 schemes for the Allen-Cahn model in Eulerian coordinates, successfully preserving the MBP property under mild time-step ratio constraints. Within the Wasserstein gradient flow framework, \cite{matthes2019variational} introduces a variational BDF2 method with proven $\frac{1}{2}$-order convergence. This approach is further extended in \cite{gallouet2024geodesic}, which applies geometric extrapolation techniques in Wasserstein space to construct BDF2 numerical methods, with numerical experiments confirming second-order accuracy.
	However, to the best of the authors' knowledge, there is no second-order structure-preserving  method with variable time steps for gradient flows in Lagrangian coordinates. 
	
In this work, we employ distinct flow dynamic approaches introduced in \cite{cheng2020new,cheng2024flow} to develop second-order Lagrangian numerical schemes with adaptive temporal discretization for both non-conservative and conservative systems. Specifically, novel numerical methods for the Allen-Cahn equation and Wasserstein gradient flows are formulated using the non-conservative continuity equation \eqref{eq:non flow} and the conservative continuity equation \eqref{eq:con flow}, respectively. These schemes are shown to preserve the intrinsic properties of the original systems, while the constraints on maximum time-step ratios are rigorously analyzed across various application scenarios.
Compared to conventional Eulerian approaches, the proposed Lagrangian framework offers significant advantages in resolving sharp-interface phenomena and addressing problems involving singularities and free boundaries. The main contributions of this study include:

\begin{itemize}
	\item[(i)] For non conservative models such as the Allen-Cahn equation, we construct a modified BDF2 scheme with variable time steps in Lagrangian coordinates via a regularized flow dynamics framework. This novel strategy introduces a new regularization operator to ensure positive-definite Jacobian determinants. The scheme is rigorously proven to preserve the energy dissipation law and the MBP property under the condition $0 < r_n \leq r_{\max} \leq \frac{3}{2}$.
	
	\item[(ii)] For conservative models written as Wasserstein gradient flows, we develop an adaptive BDF2 scheme in Lagrangian coordinates, which preserves the system’s original properties when the time-step ratio satisfies $0 < r_n \leq r_{\max} \leq \frac{3+\sqrt{17}}{2}$ in 1D and $0 < r_n \leq r_{\max} \leq \frac{5}{4}$ in 2D. 
	
	\item[(iii)] An effective adaptive time-stepping algorithm is implemented for numerical validation. The proposed schemes demonstrate enhanced capabilities in capturing sharp interfaces in the Allen-Cahn system, simulating trajectory evolution, and maintaining determinant positivity through rigorous regularization mechanisms. \end{itemize}

The remainder of the paper is organized as follows. In Section \ref{sec:ac}, we introduce the non-conservative model and present the modified BDF2 scheme with variable time steps for the Allen-Cahn equation, which is both energy-dissipative and MBP-preserving. Section \ref{sec:semi} details and analyzes BDF2 schemes with adaptive time steps for Wasserstein gradient flows, establishing corresponding energy dissipation laws. Numerical experiments in Section \ref{sec:num} validate the theoretical results, utilizing the strategy outlined in Section \ref{sec:strategy}.

 \section{Non-conservative models}\label{sec:ac}

 We shall first consider the adaptive time-stepping method based on the flow dynamic approach for non-conservative models in the section.
 \subsection{Flow map} 
 Given an initial position or a reference configuration ${\bm X}$, and a velocity field ${\bm v}$, recall the flow map ${\bm x}({\bm X},t)$:
 \begin{align}
 &	\frac{\mathrm{d}{\bm x}({\bm X},t)}{\mathrm{d} t}={\bm v}({\bm x}({\bm X},t),t),\label{flow}\\
 &	{\bm x}({\bm X},0)={\bm X},\label{flow1}
 \end{align}
 where ${\bm x}$ represents the Eulerian coordinates and ${\bm X}$ denotes the Lagrangian coordinates, $\frac{\partial {\bm x}}{\partial {\bm X}}$ represents the deformation associated with the flow map.  We assume that ${\bm v}$ is the velocity such that 
\begin{equation}\label{transport}
\rho_t+({\bm v}\cdot\Grad_{\bx})\rho=0.
\end{equation}

\begin{figure}[htbp!]
\centering
\includegraphics[width=0.45\textwidth,clip==]{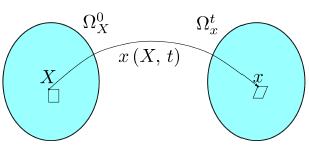}
\caption{A schematic illustration of a flow map $\bm x(\bm X,t)$ at a fixed time $t$: $\bm x(\bm X,t)$ maps $\Omega_0^{\bm X}$ to $\Omega_t^{\bm x}$. $\bm X$ is the Lagrangian coordinate while $\bm x$ is  the Eulerian coordinate, and  $F(\bm X,t)=\frac{\partial \bm x(\bm X,t)}{\partial \bm X}$ represents the deformation associated with the flow map. }\label{flow_map}
\end{figure}

\subsection{Maximum bounds preserving}
Then, the above transport equation and  the flow map  defined in \eqref{flow}-\eqref{flow1}  determine the following  kinematic relationship between Eulerian and Lagrangian coordinates, see also \cite{cheng2020new,cheng2024flow}
\begin{eqnarray}
&&\frac{\mathrm{d}}{\mathrm{d}t}\rho(\bx(\bX,t),t)=\rho_t+({\bm v}\cdot\Grad_{\bx})\rho=0, \label{map:1}
\end{eqnarray}
which leads to 
\begin{equation}\label{essential}
 \rho(\bx(\bX,t),t)=\rho(\bx,0)=\rho_0(\bx) \quad\forall t, 
 \end{equation}
 where $\rho_0(\bx)$ is the  initial condition in the Lagrangian coordinates. Since $\bx(\bX,0)=\bX$, we have $\rho_0(\bx)=\rho_0(\bX)$.
 
 Then we obtain from \eqref{essential}
 \begin{equation}
  \rho(\bx(\bX,t),t)=\rho\circ \bx (\bX) = \rho_0(\bX).
 \end{equation}
 
 Once we have the flow map $\bx(\bX,t)$, we set $\phi(\bX;t)=\bx(\bX,t)$ for each $t$. Then, we derive from \eqref{map:1} that  the solution of \eqref{transport} is given by
\begin{eqnarray}\label{comp}
\rho(\bx,t)=\rho_0 \circ \phi^{-1}(\bx,t)= \rho_0(\phi^{-1}(\bx,t)),
\end{eqnarray}
where  $\phi^{-1}$ is the inverse function of $\phi$, see \cite{cheng2020new,cheng2024flow}.

 Then $\rho(\bx(\bX,t),t)=\rho(\bX,0)$ for all $t$. This indicates that the method preserves the maximum principle, i.e., if the initial value satisfies $0<a\le\rho(\bX,0)\le b$, then for any time $t$, we have  $a\le\rho(\bx,t)\le b$. 
It is also well-known that the Allen-Cahn  equation \eqref{eq:ac ori} with variable mobility  possesses the maximum principle-preserving property \cite{evans1992phase}, which implies that if the initial value satisfies $|\rho(\bx,0)|\le 1$ for all $\bx\in\Omega$, then the solution satisfies $|\rho(\bx,t)|\le 1$ for all $(\bx,t)\in\Omega\times(0,T]$.

 \subsection{$L^2$ gradient flow in Lagrangian coordinate}
We shall take  the Allen-Cahn equation with Dirichlet boundary condition as an example of non-conservative model in this section. The Allen-Cahn equation which is a $L^2$ gradient flow,  originally introduced by Allen and Cahn in \cite{allen1979microscopic} to describe the motion of anti-phase boundaries in crystalline solids, takes  the following form:
\begin{equation}\label{eq:ac ori}
\partial_t\rho=-\mathcal{M}(\rho)(-\epsilon^2\Delta\rho+F'(\rho)),\qquad (x,t)\in\Omega\times(0,T],
\end{equation} 
	 where $\rho$ represents the concentration of one of the two metallic components of the alloy, $\epsilon>0$  is a small parameter reflecting the width of the transition regions, and $T$ is the final time, $\mathcal{M}(\rho)> 0$ is a general mobility function, which can be  $\mathcal{M}(\rho)\equiv1$ or $\mathcal{M}(\rho)=1-\rho^2$, see \cite{hou2023linear}. 
	$F(\rho)=\frac{1}{4}(\rho^2-1)^2$ is the Ginzburg-Landau double-well potential $F(\rho)=\frac{1}{4}(\rho^2-1)^2$ \cite{shen2010numerical}. Equation \eqref{eq:ac ori} can be regarded  as a $L^2$ gradient flow, in the sense that
\begin{equation}\label{diss:law:0}
\frac{\mathrm{d}}{\mathrm{d}t}E(\rho)=-\int_{\Omega}\mathcal{M}(\rho)|\rho_t|^2\mathrm{d}x,
\end{equation}
 in which the energy  is defined by $E(\rho)=\int_{\Omega}\frac{\epsilon^2}{2}|\nabla\rho|^2+F(\rho)\ \mathrm{d}x$. By introducing  the flow map \eqref{transport}, we can reformulate the energy dissipative law \eqref{diss:law:0} into
\begin{align}\label{diss:law:1}
\frac{\mathrm{d}}{\mathrm{d}t}E(\rho)=-\int_{\Omega}\mathcal{M}(\rho)|({\bm v}\cdot\Grad_{\bx})\rho|^2\mathrm{d}x,
\end{align}
where we used the equality $\rho_t=-({\bm v}\cdot\Grad_{\bx})\rho$. By using the least action principle, and Newton's force balance law \cite{cheng2020new,liu2020variational}, we can derive the following trajectory equation in Eulerian coordinates:
 \begin{align}\label{traj}
 {\bm v}\cdot\nabla\rho=\mathcal{M}(\rho)(-\epsilon^2\Delta\rho+F'(\rho)).
 \end{align}

 We rewrite the trajectory equation \eqref{traj}  into  Lagrangian coordinates
 \begin{align}
 x_t(X,t)\rho_0^{\prime}(X)\left(\frac{\partial x}{\partial X}\right)^{-1}{\frac{1}{\mathcal{M}(\rho_0(X))}}=-\epsilon^2\partial_X\left(\rho_0^{\prime}(X)\left(\frac{\partial x}{\partial X}\right)^{-1}\right)\left(\frac{\partial x}{\partial X}\right)^{-1}+F'(\rho_0(X)),\label{ac:tra1}
 \end{align}
 with the following boundary and initial conditions:
 \begin{align}
 x|_{\partial\Omega}=X|_{\partial\Omega},\qquad x(X,0)=X,\ X\in\Omega.
 \end{align}

Notice  that the determinant $\frac{\partial x}{\partial X}(X,t)>0$ should be positive, $\forall X\in\Omega $, see \cite{cheng2020new}.  In the following adaptive time-stepping Lagrangian schemes, we introduce an extra logarithmic term $\eta_0\partial_X\left(\frac{\mathrm{d}}{\mathrm{d}t}\log(\frac{\partial x}{\partial X})\right)$ with $0<\eta_0 \ll 0$ to preserve the positivity of the determinant $\frac{\partial x}{\partial X}(X,t)$.  Then a modified trajectory equation in Lagrangian coordinates for the Allen-Cahn equation with variable mobility is proposed as follows:
 	\begin{equation}\label{ac:regularied tra1}
 	\begin{aligned}
 	& x_t(X,t)(\rho_0^{\prime}(X))^2\left(\frac{\partial x}{\partial X}\right)^{-1}{\frac{1}{\mathcal{M}(\rho_0(X))}}-\eta_0\partial_X\left(\frac{\mathrm{d}}{\mathrm{d}t}\log\left(\frac{\partial x}{\partial X}\right)\right)\\
 	=&-\frac{\epsilon^2}{2}\partial_X\left(\rho_0^{\prime}(X)\left(\frac{\partial x}{\partial X}\right)^{-1}\right)^2+\partial_XF(\rho_0(X)).
 	\end{aligned}
 	\end{equation}
 	
Based on the modified trajectory equation \eqref{ac:regularied tra1}, we can derive the following modified energy dissipative law for the Allen-Cahn equation in Lagrangian coordinate.
 
\begin{lemma}\label{lem:ori energy}
	For the   Allen-Cahn equation \eqref{eq:ac ori} with a general mobility $\mathcal{M}(\rho)>0$, the trajectory equation \eqref{ac:regularied tra1} is energy dissipative in the sense that
	\begin{align}\label{eq: ac 1}
		\frac{\mathrm{d}E_{ac}}{\mathrm{d}t}=-\int_{\Omega_0^{\bm X}}|x_t(X,t)\rho_0^{\prime}(X)(\frac{\partial x}{\partial X})^{-1}|^2\frac{1}{\mathcal{M}(\rho_0(X))}\frac{\partial x}{\partial X}\mathrm{d}X-\int_{\Omega_0^{\bm X}} \eta_0|\partial_Xx_t|^2\left(\frac{\partial x}{\partial X}\right)^{-1}\mathrm{d}X,
	\end{align}
	where 
	$$E_{ac}=\int_{\Omega_0^{\bm X}}\left(\frac{\epsilon^2}{2}|\rho_0^{\prime}(X)(\frac{\partial x}{\partial X})^{-1}|^2+F(\rho_0(X))\right)\frac{\partial x}{\partial X}\mathrm{d}X.$$ 
\end{lemma}
\begin{proof}
Multiplying both sides of \eqref{ac:regularied tra1} by $-x_t$ and using integration by parts leads to 
\begin{equation*}
	\begin{aligned}
		&-\int_{\Omega_0^{\bm X}}\left|x_t\rho_0^{\prime}(X)(\frac{\partial x}{\partial X})^{-1}\right|^2\frac{\partial_{X}x}{\mathcal{M}(\rho_0(X))}\mathrm{d}X+\eta_0\int_{\Omega_0^{\bm X}}\partial_{X}\left(\frac{\mathrm{d}}{\mathrm{d}t}\log\left(\frac{\partial x}{\partial X}\right)\right)x_t\mathrm{d}X\\
		=&\frac{\epsilon^2}{2}\int_{\Omega_0^{\bm X}}\partial_X\left(\rho_0^{\prime}(X)\left(\frac{\partial x}{\partial X}\right)^{-1}\right)^2x_t-\partial_XF(\rho_0(X))x_t\mathrm{d}X\\
		=&-\frac{\epsilon^2}{2}\int_{\Omega_0^{\bm X}}\left(\rho_0^{\prime}(X)\left(\frac{\partial x}{\partial X}\right)^{-1}\right)^2\partial_Xx_t+F(\rho_0(X))\partial_Xx_t\mathrm{d}X\\
		=&\frac{\mathrm{d}}{\mathrm{d}t}\int_{\Omega_0^{\bm X}}\left(\frac{\epsilon^2}{2}|\rho_0^{\prime}(X)\left(\frac{\partial x}{\partial X}\right)^{-1}|^2+F(\rho_0(X))\right)\frac{\partial x}{\partial X}\mathrm{d}X.
	\end{aligned}
\end{equation*}
The second term on the left hand side of above equality can be rewritten by
\begin{equation*}
	\begin{aligned}
	\int_{\Omega_0^{\bm X}}\partial_{X}\left(\frac{\mathrm{d}}{\mathrm{d}t}\log\left(\frac{\partial x}{\partial X}\right)\right)x_t\mathrm{d}X=-\int_{\Omega_0^{\bm X}}\frac{\mathrm{d}}{\mathrm{d}t}\log\left(\frac{\partial x}{\partial X}\right)\partial_{X}x_t\mathrm{d}X=\int_{\Omega_0^{\bm X}} |\partial_Xx_t|^2\left(\frac{\partial x}{\partial X}\right)^{-1}\mathrm{d}X.
	\end{aligned}
\end{equation*}
Then a combination of above estimations yields the energy dissipation law \eqref{eq: ac 1}.
\end{proof}

 \subsection{Adaptive time-stepping Lagrangian scheme}
 Given time-step $\tau_{n}:=t^{n}-t^{n-1}>0$, and define the  time-step ratio $r_{n+1}:=\frac{\tau_{n+1}}{\tau_{n}}>0$, $n\ge 1$,  assume that the time-step ratio $\{r_n\}_n$ is uniformly bounded with an upper bound $r_{\max}$ such that $r_n\le r_{\max}$, $\forall n$. For any given final time $T>0$, $T=\sum_{n=1}^N\tau_{n}$.  Denote ${ x}^n$  the numerical approximation to ${ x}(\cdot, t^n)$.

 The primary challenge lies in approximating the non-constant mobility term  $\rho_0^{\prime}(X)\left(\frac{\partial x}{\partial X}\right)^{-1}\frac{1}{\mathcal{M}(\rho_0(X))}$ in \eqref{ac:regularied tra1} for variable time-stepping methods  that can also preserve energy stability in the discrete level. 
 By introducing  $\ell_n:=\frac{\partial x^n}{\partial X}\mathcal{M}(\rho_0(X))$, we propose  the following  second-order adaptive time-stepping  scheme for \eqref{ac:regularied tra1}:
  $\forall n\ge1$, given ${ x}^{n-1},\ { x}^n$ and $\rho({X},0)$, solve $({ x}^{n+1},\rho^{n+1})$  from
  \begin{equation}
 \begin{aligned}
 &\frac{(2r_{n+1}+1)(\rho_0^{\prime}(X))^2}{2\tau_{n+1}(r_{n+1}+1)}(\ell_{n+1}^{-1}+\ell_n^{-1})(x^{n+1}-x^n) -\eta\tau_{n+1}\partial_X\left(\log\frac{\partial x^{n+1}}{\partial X}-\log\frac{\partial x^{n}}{\partial X}\right)\\
 &-\frac{r_{n+1}^2(\rho_0^{\prime}(X))^2}{\tau_{n+1}(r_{n+1}+1)}((1+\frac{1}{2r_{n+1}})\ell_{n}^{-\frac{1}{2}}-\frac{1}{2r_{n+1}}\ell_{n+1}^{-\frac{1}{2}})(\frac{1}{2}\ell_{n-1}^{-\frac{1}{2}}+\frac{1}{2}\ell_{n}^{-\frac{1}{2}})(x^{n}-x^{n-1})\\
 =&-\frac{\epsilon^2}{2}\partial_X\left(\rho_0^{\prime}(X)\left(\frac{\partial x^{n+1}}{\partial X}\right)^{-1}\right)^2+\partial_XF(\rho_0(X)).\label{scheme:bdf2 modified variable}
 \end{aligned}
 \end{equation}
 
 \begin{remark}
Notice that $\frac{1}{2\tau_{n+1}}(\ell_{n+1}^{-1}+\ell_n^{-1})(x^{n+1}-x^n)$ and $\frac{1}{\tau_n}((1+\frac{1}{2r_{n+1}})\ell_{n}^{-\frac{1}{2}}-\frac{1}{2r_{n+1}}\ell_{n+1}^{-\frac{1}{2}})(\frac{1}{2}\ell_{n-1}^{-\frac{1}{2}}+\frac{1}{2}\ell_{n}^{-\frac{1}{2}})(x^{n}-x^{n-1})$ are second-order approximations to $\ell_{n+\frac{1}{2}}^{-1}x_t|_{n+\frac{1}{2}}$ and $\ell^{-1}_{n-\frac{1}{2}}x_t|_{n-\frac{1}{2}}$, respectively.  Then the linear combination $\frac{2r_{n+1}+1}{r_{n+1}+1}\ell^{-1}_{n+\frac{1}{2}}x_t|_{n+\frac{1}{2}}-\frac{r_{n+1}}{r_{n+1}+1}\ell^{-1}_{n-\frac{1}{2}}x_t|_{n-\frac{1}{2}}$ is a second-order approximation to $\ell^{-1}_{n+1}x_t|_{n+1}$, see  \cite{wang2022fully} . 
\end{remark} 
 
 The artificial regularization term $-\eta\tau_{n+1}\partial_X\left(\log\frac{\partial x^{n+1}}{\partial X}-\log\frac{\partial x^{n}}{\partial X}\right)$ with $\eta\ge 0$ is introduced in \eqref{scheme:bdf2 modified variable} to preserve the positivity of $\frac{\partial x^{n+1}}{\partial X}$ at the next time step. 
 	Once the trajectory $x^{n+1}$ is determined from \eqref{scheme:bdf2 modified variable}, the solution to Allen-Cahn equation can be obtained from $\rho(x^{n+1}(X))=\rho_0(X)$. 
 	
 	We  prove below that the scheme  \eqref{scheme:bdf2 modified variable} is energy stable with variable time steps.
 \begin{theorem}
 		For the Allen-Cahn equation \eqref{eq:ac ori} with a general mobility $\mathcal{M}(\rho)>0$, if the maximum time-step ratio satisfies $0<r_{\max}\le\frac{3}{2}$, then the second-order scheme  \eqref{scheme:bdf2 modified variable} is energy dissipative in the sense that
 	\begin{equation}
 	\begin{aligned}
 	&E^{n+1}+\frac{r_{\max}}{2\tau_{n+1}(r_{\max}+1)}\left((\rho_0^{\prime}(X))^2(\ell_{n+1}^{-1}+\ell_{n}^{-1})(x^{n+1}-x^{n}),x^{n+1}-x^{n}\right)\\
 	&+\eta\tau_{n+1}\left(\log\frac{\partial x^{n+1}}{\partial X}-\log\frac{\partial x^{n}}{\partial X},\frac{\partial x^{n+1}}{\partial X}-\frac{\partial x^{n}}{\partial X}\right)\\
 	\le &E^{n}+\frac{r_{\max}}{2\tau_{n}(r_{\max}+1)}\left((\rho_0^{\prime}(X))^2(\ell_{n-1}^{-1}+\ell_{n}^{-1})(x^{n}-x^{n-1}),x^{n}-x^{n-1}\right),
 	\end{aligned}
 	\end{equation}
 	 	where the energy is  defined by $E^n=\int_{\Omega_0^{\bm X}}\left(\frac{\epsilon^2}{2}|\rho_0^{\prime}(X)(\frac{\partial x^n}{\partial X})^{-1}|^2+F(\rho_0(X))\right)\frac{\partial x^n}{\partial X}\mathrm{d}X$.
 \end{theorem}
 \begin{proof}
 		We take inner product of \eqref{scheme:bdf2 modified variable} with respect to $-(x^{n+1}-x^n)$, and then estimate each term of the following equality:
\begin{equation}\label{energy,eq1}
	 	\begin{aligned}
	&\sum_{i=1}^3L_i:=-\frac{2r_{n+1}+1}{2\tau_{n+1}(r_{n+1}+1)}\left((\rho_0^{\prime}(X))^2(\ell_{n+1}^{-1}+\ell_n^{-1})(x^{n+1}-x^n),x^{n+1}-x^n\right)\\
	&+\left(\eta\tau_{n+1}\partial_X\left(\log\frac{\partial x^{n+1}}{\partial X}-\log\frac{\partial x^{n}}{\partial X}\right),x^{n+1}-x^n\right)+\frac{r_{n+1}^2}{\tau_{n+1}(r_{n+1}+1)}\\
	&\times\left((\rho_0^{\prime}(X))^2((\frac{2r_{n+1}+1}{2r_{n+1}})\ell_{n}^{-\frac{1}{2}}-\frac{1}{2r_{n+1}}\ell_{n+1}^{-\frac{1}{2}})(\frac{1}{2}\ell_{n-1}^{-\frac{1}{2}}+\frac{1}{2}\ell_{n}^{-\frac{1}{2}})(x^{n}-x^{n-1}),x^{n+1}-x^{n}\right)\\
	=&\frac{\epsilon^2}{2}\left(\partial_{\bm X}\Big|\rho_0^{\prime}(X)(\frac{\partial x^{n+1}}{\partial X})^{-1}\Big|^2,x^{n+1}-x^{n}\right)-\left(\partial_XF(\rho_0(X)),x^{n+1}-x^{n}\right):=\sum_{i=1}^2R_i.
	\end{aligned}
\end{equation}
 	By repeatedly applying the Cauchy-Schwarz inequality to estimate $L_3$, it can be shown that 
 	 	\begin{align*}
 	&\left((\rho_0^{\prime}(X))^2((1+\frac{1}{2r_{n+1}})\ell_{n}^{-\frac{1}{2}}-\frac{1}{2r_{n+1}}\ell_{n+1}^{-\frac{1}{2}})(\frac{1}{2}\ell_{n-1}^{-\frac{1}{2}}+\frac{1}{2}\ell_{n}^{-\frac{1}{2}})(x^{n}-x^{n-1}),x^{n+1}-x^{n}\right)\\
 	\le& \frac{1}{4}\left((\rho_0^{\prime}(X))^2((1+\frac{1}{2r_{n+1}})\ell_{n}^{-\frac{1}{2}}-\frac{1}{2r_{n+1}}\ell_{n+1}^{-\frac{1}{2}})^2(x^{n+1}-x^{n}),x^{n+1}-x^{n}\right)\\
 	&+\left((\rho_0^{\prime}(X))^2(\frac{1}{2}\ell_{n-1}^{-\frac{1}{2}}+\frac{1}{2}\ell_{n}^{-\frac{1}{2}})^2(x^{n}-x^{n-1}),x^{n}-x^{n-1}\right)\\
 	\le &\frac{1}{4}\left((\rho_0^{\prime}(X))^2(\frac{(2r_{n+1}+1)^2+2r_{n+1}+1}{4r_{n+1}^2}\ell_{n}^{-1}+\frac{2r_{n+1}+2}{4r_{n+1}^2}\ell_{n+1}^{-1})(x^{n+1}-x^{n}),x^{n+1}-x^{n}\right)\\
 	&+\frac{1}{2}\left((\rho_0^{\prime}(X))^2(\ell_{n-1}^{-1}+\ell_{n}^{-1})(x^{n}-x^{n-1}),x^{n}-x^{n-1}\right)\\
 	= &\frac{1}{4}\frac{(2r_{n+1}+1)^2+2r_{n+1}+1}{4r_{n+1}^2}\left((\rho_0^{\prime}(X))^2(\ell_{n}^{-1}+\ell_{n+1}^{-1})(x^{n+1}-x^{n}),x^{n+1}-x^{n}\right)\\
 	&-\frac{1}{4}\frac{(2r_{n+1}+1)^2-1}{4r_{n+1}^2}\left((\rho_0^{\prime}(X))^2\ell_{n+1}^{-1}(x^{n+1}-x^{n}),x^{n+1}-x^{n}\right)\\
 	&+\frac{1}{2}\left((\rho_0^{\prime}(X))^2(\ell_{n-1}^{-1}+\ell_{n}^{-1})(x^{n}-x^{n-1}),x^{n}-x^{n-1}\right).
 	\end{align*}
  SInce $(2r_{n+1}+1)^2-1>0$ $\forall r_{n+1}>0$, we have
	 	\begin{equation}\label{energy,eql1}
	\begin{aligned}
	L_1+L_3
	\le&-\frac{2r_{n+1}+1}{2\tau_{n+1}(r_{n+1}+1)}\left((\rho_0^{\prime}(X))^2(\ell_{n+1}^{-1}+\ell_n^{-1})(x^{n+1}-x^n),x^{n+1}-x^n\right)\\
	&+\frac{(2r_{n+1}+1)^2+2r_{n+1}+1}{16\tau_{n+1}(r_{n+1}+1)}\left((\rho_0^{\prime}(X))^2(\ell_{n}^{-1}+\ell_{n+1}^{-1})(x^{n+1}-x^{n}),x^{n+1}-x^{n}\right)\\
	&+\frac{r_{n+1}^2}{2\tau_{n+1}(r_{n+1}+1)}\left((\rho_0^{\prime}(X))^2(\ell_{n-1}^{-1}+\ell_{n}^{-1})(x^{n}-x^{n-1}),x^{n}-x^{n-1}\right)\\
	=&-\frac{(2r_{n+1}+1)(3-r_{n+1})}{8\tau_{n+1}(r_{n+1}+1)}\left((\rho_0^{\prime}(X))^2(\ell_{n+1}^{-1}+\ell_n^{-1})(x^{n+1}-x^n),x^{n+1}-x^n\right)\\
	&+\frac{r_{n+1}^2}{2\tau_{n+1}(r_{n+1}+1)}\left((\rho_0^{\prime}(X))^2(\ell_{n-1}^{-1}+\ell_{n}^{-1})(x^{n}-x^{n-1}),x^{n}-x^{n-1}\right).
	\end{aligned}
	\end{equation}
 		Using the integration by parts and the mean value theorem 
 	yields the following estimate for $L_2$: 
 	\begin{equation}
 	\begin{aligned}
 	L_2&=-\eta\tau_{n+1}\left(\log\frac{\partial x^{n+1}}{\partial X}-\log\frac{\partial x^{n}}{\partial X},\frac{\partial x^{n+1}}{\partial X}-\frac{\partial x^{n}}{\partial X}\right)\\
 	&=-\eta\tau_{n+1}\left(\left(\frac{\partial \xi}{\partial X}\right)^{-1},\left(\frac{\partial x^{n+1}}{\partial X}-\frac{\partial x^{n}}{\partial X}\right)^2\right)\le 0,
 	\end{aligned}
 	\end{equation} 
 	where $\xi$ lies between $x^n$ and $x^{n+1}$.
 	The first term on the right hand side of \eqref{energy,eq1} can be estimated by applying the integration by parts and the convexity of $\frac{1}{y}$ with respect to $y$ with $y>0$:
 	\begin{equation}
 	\begin{aligned}
 	R_1=&-\frac{\epsilon^2}{2}\left(\Big|\rho_0^{\prime}(X)(\frac{\partial x^{n+1}}{\partial X})^{-1}\Big|^2,\partial_{\bm X}(x^{n+1}-x^{n})\right)\\
 	\ge&\ \frac{\epsilon^2}{2}\int_{\Omega_0^{\bm X}}(\rho_0^{\prime}(X))^2(\frac{\partial x^{n+1}}{\partial X})^{-1}\mathrm{d} X-\frac{\epsilon^2}{2}\int_{\Omega_0^{\bm X}}(\rho_0^{\prime}(X))^2(\frac{\partial x^{n}}{\partial X})^{-1}\mathrm{d} X.
 	\end{aligned}
 	\end{equation}
 	Similarly, the last term on the right hand side of \eqref{energy,eq1} can be rewritten as follows
 	\begin{equation}\label{energy,eqr2}
 	\begin{aligned}
 	R_2
 	=-\int_{\Omega_0^{\bm X}}\partial_{\bm X}F(\rho_0(X))(x^{n+1}-x^{n})\mathrm{d} X
 	=\int_{\Omega_0^{\bm X}}F(\rho_0(X))\partial_{\bm X}(x^{n+1}-x^{n})\mathrm{d} X.
 	\end{aligned}
 	\end{equation}
 	Substituting \eqref{energy,eql1}-\eqref{energy,eqr2} into \eqref{energy,eq1} shows that
 	\begin{align*}
 	&E^{n+1}+\frac{(2r_{n+1}+1)(3-r_{n+1})}{8\tau_{n+1}(r_{n+1}+1)}\left((\rho_0^{\prime}(X)(\ell_{n+1}^{-1}+\ell_n^{-1})(x^{n+1}-x^n),\rho_0^{\prime}(X)(x^{n+1}-x^n)\right)\\
 	&+\eta\tau_{n+1}\left(\log\frac{\partial x^{n+1}}{\partial X}-\log\frac{\partial x^{n}}{\partial X},\frac{\partial x^{n+1}}{\partial X}-\frac{\partial x^{n}}{\partial X}\right)\\
 	\le &E^{n}+\frac{r_{n+1}^2}{2\tau_{n+1}(r_{n+1}+1)}\left((\rho_0^{\prime}(X))^2(\ell_{n-1}^{-1}+\ell_{n}^{-1})(x^{n}-x^{n-1}),(x^{n}-x^{n-1})\right)\\
 	\le & E^{n}+\frac{r_{\max}}{2\tau_{n}(r_{\max}+1)}\left((\rho_0^{\prime}(X))^2(\ell_{n-1}^{-1}+\ell_{n}^{-1})(x^{n}-x^{n-1}),(x^{n}-x^{n-1})\right).
 	\end{align*}
 	Then the energy dissipation law will be derived once the following inequality holds:
 	\begin{align*}
 	\frac{(2r_{n+1}+1)(3-r_{n+1})}{8(r_{n+1}+1)}\ge \frac{r_{\max}}{2(r_{\max}+1)},
 	\end{align*}
 	Denote function $h(s)=\frac{(2s+1)(3-s)}{8(s+1)}$ for $0<s<3$, since $h(s)$ is increasing in the range $0<s<-1+\sqrt{2}$, and decreasing in the range $-1+\sqrt{2}<s\le r_{\max}<3$. Then if the maximum time-step ratio satisfies $0<r_{\max}\le\frac{3}{2}$, we have $2r_{\max}^2-r_{\max}-3\le 0$ and
 	\begin{align*}
 	h(s)\ge \min\{h(0),h(r_{\max})\}=\min\left\{\frac{3}{8},	\frac{(2r_{\max}+1)(3-r_{\max})}{8(r_{\max}+1)}\right\}\ge \frac{r_{\max}}{2(r_{\max}+1)},
 	\end{align*}
 	which completes the proof.
 \end{proof}
 \begin{remark}
 	If we set the time-step ratio $r_{n}\equiv 1$, $\forall n$, the scheme \eqref{scheme:bdf2 modified variable}  reduces to a second-order numerical scheme with a fixed time step. In this case, the energy dissipation law will also hold
 	 	\begin{align*}
 	&E^{n+1}+\frac{3}{8\tau}\left((\rho_0^{\prime}(X))^2(\ell_{n+1}^{-1}+\ell_n^{-1})(x^{n+1}-x^n),(x^{n+1}-x^n)\right)\\
 	\le &E^{n}+\frac{1}{4\tau}\left((\rho_0^{\prime}(X))^2(\ell_{n-1}^{-1}+\ell_n^{-1})(x^{n}-x^{n-1}),(x^{n}-x^{n-1})\right).
 	\end{align*}
 \end{remark}
\begin{remark}
	 	The standard BDF2 numerical scheme with variable time steps can also be proposed:
	\begin{equation}
	\begin{aligned}
	&\frac{(1+2r_{n+1})x^{n+1}-(1+r_{n+1})^2x^{n}+r_{n+1}^2x^{n-1}}{\tau_{n+1}(1+r_{n+1})}\rho_0^{\prime}(X)\left(\frac{\partial x^{n+1}_\star}{\partial X}\right)^{-1}\\
	=&-\epsilon^2\partial_X\left(\rho_0^{\prime}(X)\left(\frac{\partial x^{n+1}}{\partial X}\right)^{-1}\right)\left(\frac{\partial x^{n+1}}{\partial X}\right)^{-1}+F'(\rho_0(X)),\label{scheme:ac 2nd}
	\end{aligned}
	\end{equation}
	where $\frac{\partial x_{\star}^{n+1}}{\partial X}$ is a second-order extrapolation term defined by
	\begin{equation*}
		\frac{\partial x_{\star}^{n+1}}{\partial X}=\begin{cases}
			\frac{\partial ((1+r_{n+1})x^{n}-r_{n+1}x^{n-1})}{\partial X},&\ \text{if}\ \frac{\partial x^n}{\partial X}>\frac{\partial x^{n-1}}{\partial X},\\
				\frac{1}{(1+r_{n+1})/\frac{\partial x^n}{\partial X} -r_{n+1}/\frac{\partial x^{n-1}}{\partial X}},&\ \text{if}\ \frac{\partial x^n}{\partial X}<\frac{\partial x^{n-1}}{\partial X}.
		\end{cases}
	\end{equation*} 
	The scheme \eqref{scheme:ac 2nd} is positivity-preserving for $\frac{\partial x^{n+1}}{\partial X}$ and MBP-preserving,  a proof of which follows similarly from previous works \cite{cheng2020new}. This idea has also been applied to the Keller-Segel equation \cite{wang2022fully}, the Cahn-Hilliard equation \cite{liu2023positivity} and the PME \cite{duan2019pme,liu2023envara}. 
	However, as discussed in \cite{wang2022fully,shen2020unconditionally}, the energy stability of \eqref{scheme:ac 2nd} 	is not easy to be proved due to the non-constant term $(\frac{\partial x^{n+1}_\star}{\partial X})^{-1}$. The advantage of the proposed scheme \eqref{scheme:bdf2 modified variable} is that its energy stability can be theoretically established, in contrast to \eqref{scheme:ac 2nd}.
 \end{remark}

\section{Mass conservative models}\label{sec:semi}
In this section, we explore adaptive time-stepping methods for mass-conservative models, using Wasserstein gradient flows as illustrative examples. If the variable $\rho$ satisfies the conservative transport equation \eqref{eq:con flow}, the mass conservation law is inherently fulfilled.

\subsection{Wasserstein gradient flows in Lagrangian coordinate}
The conservative model  \eqref{eq:wd} with $\mathcal{M}(\rho)=\rho$ can be written as Wasserstein gradient flows \cite{benamou2016augmented,carrillo2003kinetic}
\begin{align}\label{eq:wasserstein}
\begin{split}
	\partial_t\rho=\nabla\cdot(\rho\bv),\quad \bm{v}=\nabla\frac{\delta E}{\delta \rho},
	\end{split}
\end{align}
where the  energy  is defined by $E(\rho)=\int_{\Omega}F(\rho)\mathrm{d}{\bm x}
 =\int_{\Omega}U(\rho({\bm x}))+V({\bm x})\rho({\bm x})\mathrm{d}{\bm x}+\frac{1}{2}\int_{\Omega\times\Omega}W({\bm x}-{\bm y})\rho({\bm x})\rho({\bm y})\mathrm{d}{\bm x}\mathrm{d}{\bm y}$.  $F(\rho)$ is the energy density, $U(\rho)$ is an internal energy density, $V({\bm x})$ is a drift potential and $W({\bm x},{\bm y})=W({\bm y},{\bm x})$ is an interaction potential . 
The solution to \eqref{eq:wasserstein} satisfies the nice properties,  positivity-preserving, mass conservative and energy dissipative.  As discussed in \cite{cheng2024flow}, 
a combination of \eqref{eq:con flow} and Wasserstein gradient flows \eqref{eq:wasserstein} leads to the constitutive relation:
$$\rho {\bm v} = - \rho \Grad \frac{\delta E}{\delta \rho},$$  
which can be rewritten into the following Lagrangian form \cite{cheng2024flow}:
 \begin{equation}\label{vis:force:0}
 \begin{split}
 &\rho_0({\bm X})\frac{\mathrm{d} {\bm x}}{\mathrm{d}t}+\rho_0(\bm X)\Grad_{\bm x}\frac{\delta E}{\delta \rho}  =0,\\
 & \rho(\bx,t)=\frac{\rho_0({\bm X})}{\text{det}\frac{\partial \bm x}{\partial {\bm X}}}.
 \end{split}
\end{equation}
We can readily demonstrate that the Lagrangian formulation \eqref{vis:force:0} is equivalent to the original system \eqref{eq:wd} at the continuous level. We propose adaptive time-stepping schemes to solve \eqref{eq:wd} based on \eqref{vis:force:0}, rather than directly utilizing \eqref{eq:wd}.

We note that \eqref{vis:force:0} constitutes a highly nonlinear system. To solve \eqref{vis:force:0} efficiently, we introduce a regularization term to \eqref{vis:force:0}, resulting in the following modified system:
 \begin{equation}\label{vis:force}
 \begin{split}
& \rho_0({\bm X})\frac{\mathrm{d} {\bm x}}{\mathrm{d}t}-\varepsilon \Delta_{\bm X} \frac{\mathrm{d}{\bm x}}{\mathrm{d}t}  +\rho_0(\bm X)\Grad_{\bm x}\frac{\delta E}{\delta \rho}  =0,\\
&\rho(\bx,t)=\frac{\rho_0({\bm X})}{\text{det}\frac{\partial \bm x}{\partial {\bm X}}}.
 \end{split}
\end{equation}
We use the Dirichlet boundary condition ${\bm x}^{n+1}|_{\partial\Omega}={\bm X}|_{\partial\Omega}$, and $\Omega_0^{\bm x}=\Omega_0^{\bm X}$. Notice that  $\varepsilon \Delta_{\bm X} \frac{\mathrm{d}{\bm x}}{\mathrm{d}t}$ is  an artificial viscosity term and $\varepsilon\ge0$ is a small constant. 
Below we shall propose adaptive time-stepping methods to solve the modified system \eqref{vis:force}.

\subsection{Adaptive time-stepping method}

 A second-order semi-discrete scheme with variable time steps for \eqref{vis:force} can be constructed  as follows:  for any $n\ge1$, given ${\bm x}^{n-1},\ {\bm x}^{n}$, $\tau_{n+1}$, $r_{n+1}$ and the initial value $\rho_0({\bm X})$, the  solution $({\bm x}^{n+1},\rho^{n+1})$ can be obtained by solving
\begin{align}
&\rho_0({\bm X})\frac{(1+2r_{n+1}){\bm x}^{n+1}-(1+r_{n+1})^2{\bm x}^{n}+r_{n+1}^2{\bm x}^{n-1}}{\tau_{n+1}(1+r_{n+1})}-\tau_{n+1}\Delta_{\bm X} ({\bm x}^{n+1}-{\bm x}^{n})\nonumber\\
&\qquad+\rho_0({\bm X})\nabla_{\bm x}F'\left(\frac{\rho_0({\bm X})}{\text{det}\frac{\partial {\bm x}^{n+1}}{\partial {\bm X}}}\right)=0,\label{eq:jko_2}\\
&\rho^{n+1}=\frac{\rho_0({\bm X})}{\text{det}\frac{\partial {\bm x}^{n+1}}{\partial {\bm X}}}, \label{eq:jko_11}
\end{align}
with the Dirichlet boundary condition ${\bm x}^{n+1}|_{\partial\Omega}={\bm X}|_{\partial\Omega}$, and $\Omega_0^{\bm x}=\Omega_0^{\bm X}$.  
Since \eqref{eq:jko_2} requires initial values ${\bm x}^0$ and ${\bm x}^1$, we choose ${\bm x}^0={\bm X}$ and use the first-order scheme in \cite{cheng2024flow} to solve ${\bm x}^1$. 

Now   we will prove that the scheme \eqref{eq:jko_2}-\eqref{eq:jko_11} preserves properties of Wasserstein gradient flows.

\begin{lemma}\label{lem:semi wd}
	Assume  initial value $\rho_0({\bm X})>0$, if the energy $E({\bm x})$ is convex with respect to ${\bm x}$, then there exists a unique solution to scheme \eqref{eq:jko_2}-\eqref{eq:jko_11}. Moreover, if the internal energy density is of the form $U(s)=s\log s$, then for any $1\leq n\leq N$, the solution $\rho^{n+1}({\bm x})$ to numerical scheme \eqref{eq:jko_2}-\eqref{eq:jko_11} 
	is positive and mass conservative in the sense that $	\int_{\Omega_t^{\bm x}}\rho^{n+1}({\bm x})\mathrm{d}{\bm x}=\int_{\Omega_0^{\bm X}}\rho_0({\bm X})\mathrm{d}{\bm X}$.
\end{lemma}
\begin{proof}
The solution to scheme \eqref{eq:jko_2} is the minimizer of the following variational problem:
\begin{equation}\label{eq:jko_semi}
	\begin{aligned}
		{\bm x}^{k+1}:=\arg\min_{\bm x}\Big\{&\frac{1}{2\tau_{n+1}}\int_{\Omega_0^{\bm X}}\rho_0({\bm X})\frac{|(1+2r_{n+1}){\bm x}-(1+r_{n+1})^2{\bm x}^{n}+r_{n+1}^2{\bm x}^{n-1}|^2}{(1+r_{n+1})(1+2r_{n+1})}\mathrm{d}{\bm X}\\
		&+E({\bm x})+\frac{\tau_{n+1}}{2}\int_{\Omega_0^{\bm X}}|\nabla_{\bm X}({\bm x}-{\bm x}^k)|^2\mathrm{d}{\bm X}\Big\}.
	\end{aligned}
\end{equation}
Since the above variational problem is convex with respect to ${\bm x}$ due to the convexity of $E({\bm x})$, the  scheme \eqref{eq:jko_2} admits a unique solution. In addition, 
If the internal energy density is of the form $U(s)=s\log s$, then the term $\frac{\rho_0({\bm X})}{\text{det}\frac{\partial {\bm x}^{n+1}}{\partial {\bm X}}}$ should stay in the domain of the logarithmic function.  This implies that $\text{det}\frac{\partial {\bm x}^{n+1}}{\partial {\bm X}}\ge0$. If $\text{det}\frac{\partial {\bm x}^{n+1}}{\partial {\bm X}}=0$, then the variational energy functional in \eqref{eq:jko_semi} at 
${\bm x}^{n+1}$ becomes $+\infty$, which contradicts the minimizer. Then the solution satisfies $\rho^{n+1}>0$. 
	Using the equality \eqref{eq:jko_11} and  $\text{det}\frac{\partial {\bm x}}{\partial {\bm X}}\mathrm{d}{\bm X}=\mathrm{d}{\bm x}$, we have
	\begin{align*}
		\int_{\Omega_t^{\bm x}}\rho^{n+1}({\bm x})\mathrm{d}{\bm x}=\int_{\Omega_t^{\bm x}}\frac{\rho_0({\bm X})}{\text{det}\frac{\partial {\bm x}^{n+1}}{\partial {\bm X}}}\mathrm{d}{\bm x}=\int_{\Omega_0^{\bm X}}\frac{\rho_0({\bm X})}{\text{det}\frac{\partial {\bm x}^{n+1}}{\partial {\bm X}}}\text{det}\frac{\partial {\bm x}^{n+1}}{\partial {\bm X}}\mathrm{d}{\bm X}=\int_{\Omega_0^{\bm X}}\rho_0({\bm X})\mathrm{d}{\bm X}.
	\end{align*}  
	Hence, the scheme \eqref{eq:jko_2}-\eqref{eq:jko_11} is mass conservative.
\end{proof}

The energy dissipation law will hold under a mild condition on the maximum time-step ratio.
	\begin{theorem}
		 Assume the initial value $\rho_0({\bm X})>0$,  and the maximum time-step ratio satisfies $0<r_{\max}\le\frac{3+\sqrt{17}}{2}$. If energy $E({\bm x})$ is convex with respect to ${\bm x}$, then scheme \eqref{eq:jko_2} is energy dissipative in the sense that
\begin{equation}
	\begin{aligned}
		&E({\bm x}^{n+1})+\frac{r_{\max}}{2\tau_{n+1}(1+r_{\max})}\int_{\Omega_0^{\bm X}}\rho_0({\bm X})|{\bm x}^{n+1}-{\bm x}^{n}|^2\mathrm{d}{\bm X}\\
		\le& E({\bm x}^{n})+\frac{r_{\max}}{2\tau_{n}(1+r_{\max})}\int_{\Omega_0^{\bm X}}\rho_0({\bm X})|{\bm x}^{n}-{\bm x}^{n-1}|^2\mathrm{d}{\bm X}.
	\end{aligned}
\end{equation}
	\end{theorem}
	\begin{proof}
		Taking inner product with ${\bm x}^{n+1}-{\bm x}^{n}$ of \eqref{eq:jko_2} leads to
		\begin{align*}
		&\frac{1+2r_{n+1}}{\tau_{n+1}(1+r_{n+1})}\int_{\Omega_0^{\bm X}}\rho_0({\bm X})\left(|{\bm x}^{n+1}-{\bm x}^{n}|^2-\frac{r_{n+1}^2}{1+2r_{n+1}}({\bm x}^{n+1}-{\bm x}^{n})({\bm x}^{n}-{\bm x}^{n-1})\right)\mathrm{d}{\bm X}\nonumber\\
		+&\tau_{n+1}\|\partial_{\bm X}({\bm x}^{n+1}-{\bm x}^{n})\|^2+\left(\frac{\delta E}{\delta {\bm x}}({\bm x}^{n+1}),({\bm x}^{n+1}-{\bm x}^{n})\right)=0,
		\end{align*}
		 an application of the convexity of $E({\bm x})$ with respect to ${\bm x}$ implies that
\begin{equation}\label{eq:energy1}
	\begin{aligned}
	&\frac{1+2r_{n+1}}{\tau_{n+1}(1+r_{n+1})}\int_{\Omega_0^{\bm X}}\rho_0({\bm X})\left(|{\bm x}^{n+1}-{\bm x}^{n}|^2-\frac{r_{n+1}^2}{1+2r_{n+1}}({\bm x}^{n+1}-{\bm x}^{n})({\bm x}^{n}-{\bm x}^{n-1})\right)\mathrm{d}{\bm X}\\
	+&\tau_{n+1}\|\partial_{\bm X}({\bm x}^{n+1}-{\bm x}^{n})\|^2+E({\bm x}^{n+1})\le E({\bm x}^{n}).
	\end{aligned}
\end{equation}
		For the first term on the left hand side of \eqref{eq:energy1}, using the Cauchy-Schwarz inequality shows that
\begin{equation*}
		\begin{aligned}
		&\int_{\Omega_0^{\bm X}}\rho_0({\bm X})\left(|{\bm x}^{n+1}-{\bm x}^{n}|^2-\frac{r_{n+1}^2}{1+2r_{n+1}}({\bm x}^{n+1}-{\bm x}^{n})({\bm x}^{n}-{\bm x}^{n-1})\right)\mathrm{d}{\bm X}\\
		\ge &\int_{\Omega_0^{\bm X}}\rho_0({\bm X})\left(\frac{2+4r_{n+1}-r_{n+1}^2}{2(1+2r_{n+1})}|{\bm x}^{n+1}-{\bm x}^{n}|^2-\frac{r_{n+1}^2}{2(1+2r_{n+1})}|{\bm x}^{n}-{\bm x}^{n-1}|^2\right)\mathrm{d}{\bm X}.
	\end{aligned}
\end{equation*}
		Substituting above estimate into \eqref{eq:energy1} leads to
		\begin{align*}
		&\int_{\Omega_0^{\bm X}}\rho_0({\bm X})\left(\frac{2+4r_{n+1}-r_{n+1}^2}{2\tau_{n+1}(1+r_{n+1})}|{\bm x}^{n+1}-{\bm x}^{n}|^2-\frac{r_{n+1}^2}{2\tau_{n+1}(1+r_{n+1})}|{\bm x}^{n}-{\bm x}^{n-1}|^2\right)\mathrm{d}{\bm X}\nonumber\\
		+&\tau_{n+1}\|\partial_{\bm X}({\bm x}^{n+1}-{\bm x}^{n})\|^2+E({\bm x}^{n+1})\le E({\bm x}^{n}),
		\end{align*}
		then the following inequality holds by using $r_{n+1}=\frac{\tau_{n+1}}{\tau_{n}}$:
		\begin{align*}
		&\frac{2+4r_{n+1}-r_{n+1}^2}{2\tau_{n+1}(1+r_{n+1})}\int_{\Omega_0^{\bm X}}\rho_0({\bm X})|{\bm x}^{n+1}-{\bm x}^{n}|^2\mathrm{d}{\bm X}+\tau_{n+1}\|\partial_{\bm X}({\bm x}^{n+1}-{\bm x}^{n})\|^2+E({\bm x}^{n+1})\\
		\le &E({\bm x}^{n})+\frac{r_{n+1}}{2\tau_{n}(1+r_{n+1})}\int_{\Omega_0^{\bm X}}\rho_0({\bm X})|{\bm x}^{n}-{\bm x}^{n-1}|^2\mathrm{d}{\bm X}\\
		\le& E({\bm x}^{n})+\frac{r_{\max}}{2\tau_{n}(1+r_{\max})}\int_{\Omega_0^{\bm X}}\rho_0({\bm X})|{\bm x}^{n}-{\bm x}^{n-1}|^2\mathrm{d}{\bm X}.
		\end{align*}
		The energy dissipation law will hold under the condition that $	g(r_{n+1}):=\frac{2+4r_{n+1}-r_{n+1}^2}{1+r_{n+1}}\ge \frac{r_{\max}}{1+r_{\max}}$. 
		 It can be easily verified that the function $g(r)=\frac{2+4r-r^2}{1+r}$ is increasing within $r\in(0,-1+\sqrt{3})$ and decreasing with $r\in(-1+\sqrt{3},r_{\max})$, then we have $g(r)\ge\min\{g(0),g(r_{\max})\}$ for $0<r\le r_{\max}$. The energy dissipation law will be derived once the maximum time-step ratio satisfies:
			\begin{align*}
			g(r_{\max})\ge\frac{r_{\max}}{1+r_{\max}},
			\end{align*}
			that is $2+3r_{\max}-r_{\max}^2\ge0$, 
			then the energy dissipation law holds when $0<r_{\max}\le\frac{3+\sqrt{17}}{2}$. 
\end{proof}

\begin{remark}
	In above Theorem, we assume that the energy is convex with respect to ${\bm x}$. In fact,  the PME and Fokker-Planck equation satisfy such assumption in one dimension.
\end{remark}

\begin{remark}
	Similar to the adaptive time-stepping BDF2 scheme proposed for the Cahn-Hilliard equation in \cite{chen2019second}, which has been proved to be energy stable when the maximum time-step ratio satisfies $0<r_{\max}\le\frac{3+\sqrt{17}}{2}$. The restriction can be relaxed by using the discrete convolution kernel in \cite{liao2022mesh} and the scalar auxiliary variable approach in \cite{hou2023implicit}. It would be interesting to investigate whether the scheme \eqref{eq:jko_2} remains energy stable under a milder restriction on the maximum time-step ratio.
\end{remark}

\section{Fully discretization schemes}
In this section, we shall construct  adaptive time-stepping with  standard finite difference  in space for non-conservative models and conservative models. 
\subsection{Non-conservative models}
For simplicity, we use the Allen-Cahn equation \eqref{ac:tra1} as an example. 
 
 Let $\Omega_0^X=[-L,L]$ and $\Omega_0^x=[-L,L]$ represent the computational domain in Lagrangian and Eulerian coordinates, respectively. The spatial grids are defined as $-L=X_0<X_1<\cdots<X_{M_x}=L$, with the grid spacing given by $h=\frac{2L}{M_x}$. We define the midpoints $x_{j+\frac{1}{2}}:=\frac{1}{2}(x_j+x_{j+1})$,  and the discrete differential operators as follows: 

$$(D_hx)_{j+\frac{1}{2}}:=\frac{1}{h}(x_{j+1}-x_j),$$ for $j=0,1,\cdots,M_x-1$, 
 and $$(d_hx)_j:=\frac{1}{h}(x_{j+\frac{1}{2}}-x_{j-\frac{1}{2}}),$$ for $j=1,\cdots,M_x-1$. 
 The discrete inner product is defined by $$(u,v)_h:=\sum_{j=0}^{M_x-1}(uv)_{j+\frac{1}{2}}h,$$ and $$[u,v]_h=\sum_{j=1}^{M_x-1}(uv)_jh+\frac{h}{2}(u_0v_0+u_{M_x}v_{M_x}).$$ If the test function $v$ satisfies the boundary conditions $v_0=v_{M_x}=0$,  the summation by parts formula holds:  $(D_hu,v)_h=-[u,d_hv]_h$. 
 
 The fully discrete, adaptive time stepping scheme for the Allen-Cahn equation is given as follows: for any $n\ge1$, given $x^{n-1}$, $x^n$, $\tau_{n+1}$ and $r_{n+1}$, solve $(x^{n+1},\rho^{n+1})$ from:
   \begin{equation}\label{scheme:ac fully discrete}
 \begin{aligned}
 &\frac{(2r_{n+1}+1)(\rho_0^{\prime}(X_{j+\frac{1}{2}}))^2}{2\tau_{n+1}(r_{n+1}+1)M(\rho_0(X_{j+\frac{1}{2}}))}\left((D_hx^{n+1})_{j+\frac{1}{2}}^{-1}+(D_hx^{n})_{j+\frac{1}{2}}^{-1}\right)(x_{j+\frac{1}{2}}^{n+1}-x_{j+\frac{1}{2}}^{n})\\
 &-\eta\tau_{n+1}D_h\left(\log(d_hx^{n+1}) -\log(d_hx^{n})\right)_{j+\frac{1}{2}}\\
 &-\frac{r_{n+1}^2(\rho_0^{\prime}(X_{j+\frac{1}{2}}))^2}{2\tau_{n+1}(r_{n+1}+1)M(\rho_0(X_{j+\frac{1}{2}}))}\left((1+\frac{1}{2r_{n+1}})(D_hx^n)_{j+\frac{1}{2}}^{-\frac{1}{2}}-\frac{1}{2r_{n+1}}(D_hx^{n+1})_{j+\frac{1}{2}}^{-\frac{1}{2}}\right)\\
 &\times\left((D_hx^{n-1})_{j+\frac{1}{2}}^{-\frac{1}{2}}-(D_hx^n)_{j+\frac{1}{2}}^{-\frac{1}{2}}\right)(x_{j+\frac{1}{2}}^{n}-x_{j+\frac{1}{2}}^{n-1})\\
 =&-\frac{\epsilon^2}{2}D_h(\rho_0^{\prime}(X)(d_hx^{n+1})^{-1})^2_{j+\frac{1}{2}}+D_hF(\rho_0(X))_{j+\frac{1}{2}},
 \end{aligned}
 \end{equation}
  then the solution to the Allen-Cahn equation will be obtained by $\rho^{n+1}_{j+\frac{1}{2}}=\rho_0(X_{j+\frac{1}{2}})$ for $j=0$, $1$, $\cdots$, $M_x-1$. 
 The following initial and boundary conditions will be considered:
 \begin{align}
 	(x_0^0,x_1^0,\cdots,x_{M_x}^0)=(X_0,X_1,\cdots,X_{M_x}),\quad x^{n+1}_0=X_0,\ x^{n+1}_{M_x}=X_{M_x}.
 \end{align}
 The term $\eta\tau_{n+1}D_h\left(\log(d_hx^{n+1}) -\log(d_hx^{n})\right)$ with $\eta\ge0$ is introduced to ensure the positivity of $d_hx^{n+1}$. 

 Similar to the semi-discrete case, taking a discrete inner product of \eqref{scheme:ac fully discrete} with $-(x^{n+1}-x^n)$ will lead to the energy dissipation law in Lagrangian coordinate.
\begin{theorem}
	 	For the Allen-Cahn equation \eqref{eq:ac ori} with $\mathcal{M}(\rho)>0$, if the maximum time-step ratio satisfies $0<r_{\max}\le\frac{3}{2}$, then the fully discrete numerical scheme \eqref{scheme:ac fully discrete} is energy dissipative in the sense that
	\begin{equation}\label{eq: ac energy discrete}
	\begin{aligned}
	&E_h^{n+1}+\frac{r_{\max}}{2\tau_{n+1}(r_{\max}+1)}\left(\frac{(\rho_0^{\prime})^2}{\mathcal{M}(\rho_0)}((D_hx^{n+1})^{-1}+(D_hx^{n})^{-1})(x^{n+1}-x^{n}),(x^{n+1}-x^{n})\right)_h\\
	&+\eta\tau_{n+1} [\log(d_hx^{n+1}) -\log(d_hx^{n}),d_h(x^{n+1}-x^n)]_h\\
	\le &E_h^{n}+\frac{r_{\max}}{2\tau_{n}(r_{\max}+1)}\left(\frac{(\rho_0^{\prime})^2}{\mathcal{M}(\rho_0)}((D_hx^{n-1})^{-1}+(D_hx^{n})^{-1})(x^{n}-x^{n-1}),(x^{n}-x^{n-1})\right)_h,
	\end{aligned}
	\end{equation}
	where the discrete energy is defined by $$E_h^n=\frac{\epsilon^2}{2}\left[|\rho_0^{\prime}(X)(d_hx^n)^{-1}|^2,d_hx^n\right]_h+\left[F(\rho_0(X)),d_hx^n\right]_h.$$ 
\end{theorem}
\begin{proof}
 Taking a discrete inner product of \eqref{scheme:ac fully discrete} with $-(x^{n+1}-x^n)$ leads to
\begin{equation}
	 \begin{aligned}
		&-\frac{2r_{n+1}+1}{2\tau_{n+1}(r_{n+1}+1)}\left(\frac{(\rho_0^{\prime}(X))^2}{\mathcal{M}(\rho_0(X))}\left((D_hx^{n+1})^{-1}+(D_hx^{n})^{-1}\right)(x^{n+1}-x^{n}),x^{n+1}-x^n\right)_h\\
		&+\eta\tau_{n+1}\left(D_h\left(\log(d_hx^{n+1}) -\log(d_hx^{n})\right),x^{n+1}-x^n\right)_h\\
		&+\frac{r_{n+1}^2}{2\tau_{n+1}(r_{n+1}+1)}\Bigg(\frac{(\rho_0^{\prime}(X))^2}{\mathcal{M}(\rho_0(X))}\left((1+\frac{1}{2r_{n+1}})(D_hx^n)^{-\frac{1}{2}}-\frac{1}{2r_{n+1}}(D_hx^{n+1})^{-\frac{1}{2}}\right)\\
		&\times\left((D_hx^{n-1})^{-\frac{1}{2}}-(D_hx^n)^{-\frac{1}{2}}\right)(x^{n}-x^{n-1}),x^{n+1}-x^n\Bigg)_h\\
		=&\frac{\epsilon^2}{2}(D_h(\rho_0^{\prime}(X)(d_hx^{n+1})^{-1})^2,x^{n+1}-x^n)-(D_hF(\rho_0(X)),x^{n+1}-x^n)_h.
	\end{aligned}
\end{equation}
By applying summation by parts and utilizing the convexity of $\frac{1}{y}$ with respect to $y$, we have
 \begin{equation}\label{eq: discrete ac 1}
 	\begin{aligned}
 		&-\frac{2r_{n+1}+1}{2\tau_{n+1}(r_{n+1}+1)}\left(\frac{(\rho_0^{\prime}(X))^2}{\mathcal{M}(\rho_0(X))}\left((D_hx^{n+1})^{-1}+(D_hx^{n})^{-1}\right)(x^{n+1}-x^{n}),x^{n+1}-x^n\right)_h\\
 		&-\eta\tau_{n+1}\left[\log(d_hx^{n+1}) -\log(d_hx^{n}),d_h(x^{n+1}-x^n)\right]_h\\
 		&+\frac{r_{n+1}^2}{2\tau_{n+1}(r_{n+1}+1)}\Bigg(\frac{(\rho_0^{\prime}(X))^2}{\mathcal{M}(\rho_0(X))}\left((1+\frac{1}{2r_{n+1}})(D_hx^n)^{-\frac{1}{2}}-\frac{1}{2r_{n+1}}(D_hx^{n+1})^{-\frac{1}{2}}\right)\\
 		&\times\left((D_hx^{n-1})^{-\frac{1}{2}}-(D_hx^n)^{-\frac{1}{2}}\right)(x^{n}-x^{n-1}),x^{n+1}-x^n\Bigg)_h\\
 		=&-\frac{\epsilon^2}{2}[(\rho_0^{\prime}(X)(d_hx^{n+1})^{-1})^2,d_h(x^{n+1}-x^n)]_h+[F(\rho_0(X)),d_h(x^{n+1}-x^n)]_h\\
 		\ge&\frac{\epsilon^2}{2}[(\rho_0^{\prime}(X)(d_hx^{n+1})^{-1})^2,d_hx^{n+1}]_h-\frac{\epsilon^2}{2}[(\rho_0^{\prime}(X)(d_hx^{n})^{-1})^2,d_hx^{n}]_h\\
 		&+[F(\rho_0(X)),d_h(x^{n+1}-x^n)]_h.
 	\end{aligned}
 \end{equation}
 Denote $D_hx^n\mathcal{M}(\rho_0)=\ell_{h^n}$, similar to 
 the semi-discrete case, repeatedly applying the Cauchy-Schwarz inequality leads to
 \begin{align*}
 	 	&\left((\rho_0^{\prime}(X))^2((1+\frac{1}{2r_{n+1}})\ell_{h^n}^{-\frac{1}{2}}-\frac{1}{2r_{n+1}}\ell_{h^{n+1}}^{-\frac{1}{2}})(\frac{1}{2}\ell_{h^{n-1}}^{-\frac{1}{2}}+\frac{1}{2}\ell_{h^{n}}^{-\frac{1}{2}})(x^{n}-x^{n-1}),x^{n+1}-x^{n}\right)_h\\
 	\le &\frac{1}{4}\left((\rho_0^{\prime}(X))^2(\frac{(2r_{n+1}+1)^2+2r_{n+1}+1}{4r_{n+1}^2}\ell_{h^{n}}^{-1}+\frac{2r_{n+1}+2}{4r_{n+1}^2}\ell_{h^{n+1}}^{-1})(x^{n+1}-x^{n}),x^{n+1}-x^{n}\right)_h\\
 	&+\frac{1}{2}\left((\rho_0^{\prime}(X))^2(\ell_{h^{n-1}}^{-1}+\ell_{h^n}^{-1})(x^{n}-x^{n-1}),x^{n}-x^{n-1}\right)_h\\
 	\le &\frac{1}{4}\frac{(2r_{n+1}+1)^2+2r_{n+1}+1}{4r_{n+1}^2}\left((\rho_0^{\prime}(X))^2(\ell_{h^n}^{-1}+\ell_{h^{n+1}}^{-1})(x^{n+1}-x^{n}),x^{n+1}-x^{n}\right)_h\\
 	&+\frac{1}{2}\left((\rho_0^{\prime}(X))^2(\ell_{h^{n-1}}^{-1}+\ell_{h^n}^{-1})(x^{n}-x^{n-1}),x^{n}-x^{n-1}\right)_h.
 \end{align*}
 Then we obtain the following estimate for the first and third terms on the left hand side of \eqref{eq: discrete ac 1}:
 \begin{align*}
 	&-\frac{2r_{n+1}+1}{2\tau_{n+1}(r_{n+1}+1)}\left(\frac{(\rho_0^{\prime}(X))^2}{\mathcal{M}(\rho_0(X))}\left((D_hx^{n+1})^{-1}+(D_hx^{n})^{-1}\right)(x^{n+1}-x^{n}),x^{n+1}-x^n\right)_h\\
 	&+\frac{r_{n+1}^2}{2\tau_{n+1}(r_{n+1}+1)}\Bigg(\frac{(\rho_0^{\prime}(X))^2}{\mathcal{M}(\rho_0(X))}\left((1+\frac{1}{2r_{n+1}})(D_hx^n)^{-\frac{1}{2}}-\frac{1}{2r_{n+1}}(D_hx^{n+1})^{-\frac{1}{2}}\right)\\
 	&\times\left((D_hx^{n-1})^{-\frac{1}{2}}-(D_hx^n)^{-\frac{1}{2}}\right)(x^{n}-x^{n-1}),x^{n+1}-x^n\Bigg)_h\\
 	\le&-\frac{2r_{n+1}+1}{2\tau_{n+1}(r_{n+1}+1)}\left(\frac{(\rho_0^{\prime}(X))^2}{\mathcal{M}(\rho_0(X))}\left((D_hx^{n+1})^{-1}+(D_hx^{n})^{-1}\right)(x^{n+1}-x^{n}),x^{n+1}-x^n\right)_h\\
 	&+\frac{(2r_{n+1}+1)^2+2r_{n+1}+1}{16\tau_{n+1}(r_{n+1}+1)}\left(\frac{(\rho_0^{\prime}(X))^2}{\mathcal{M}(\rho_0(X))}((D_hx^n)^{-1}+(D_hx^{n+1})^{-1})(x^{n+1}-x^{n}),x^{n+1}-x^{n}\right)_h\\
 	&+\frac{r_{n+1}^2}{2\tau_{n+1}(r_{n+1}+1)}\left(\frac{(\rho_0^{\prime}(X))^2}{\mathcal{M}(\rho_0(X))}((D_hx^{n-1})^{-1}+(D_hx^n)^{-1})(x^{n}-x^{n-1}),x^{n}-x^{n-1}\right)_h\\
 	 =&-\frac{(2r_{n+1}+1)(3-r_{n+1})}{8\tau_{n+1}(r_{n+1}+1)}\left(\frac{(\rho_0^{\prime}(X))^2}{\mathcal{M}(\rho_0(X))}((D_hx^n)^{-1}+(D_hx^{n+1})^{-1})(x^{n+1}-x^{n}),x^{n+1}-x^{n}\right)_h\\
 	&+\frac{r_{n+1}^2}{2\tau_{n+1}(r_{n+1}+1)}\left(\frac{(\rho_0^{\prime}(X))^2}{\mathcal{M}(\rho_0(X))}((D_hx^{n-1})^{-1}+(D_hx^n)^{-1})(x^{n}-x^{n-1}),x^{n}-x^{n-1}\right)_h,
 \end{align*}
 a combination of above inequality with \eqref{eq: discrete ac 1} shows
 \begin{equation*}
 	\begin{aligned}
 		&E_h^{n+1}+\frac{(2r_{n+1}+1)(3-r_{n+1})}{8\tau_{n+1}(r_{n+1}+1)}\left(\frac{(\rho_0^{\prime}(X))^2}{\mathcal{M}(\rho_0(X))}((D_hx^n)^{-1}+(D_hx^{n+1})^{-1})(x^{n+1}-x^{n}),x^{n+1}-x^{n}\right)_h\\
 		&+\eta\tau_{n+1}\left[\log(d_hx^{n+1}) -\log(d_hx^{n}),d_h(x^{n+1}-x^n)\right]_h\\
 		\le &E_h^n+\frac{r_{\max}}{2\tau_{n}(r_{\max}+1)}\left(\frac{(\rho_0^{\prime}(X))^2}{\mathcal{M}(\rho_0(X))}((D_hx^{n-1})^{-1}+(D_hx^n)^{-1})(x^{n}-x^{n-1}),x^{n}-x^{n-1}\right)_h.
 	\end{aligned}
 \end{equation*}
 then the energy dissipation law will hold  once the following inequality holds:
 \begin{align*}
 	h(r_{n+1})=\frac{(2r_{n+1}+1)(3-r_{n+1})}{8(r_{n+1}+1)}\ge \frac{r_{\max}}{2(r_{\max}+1)},
 \end{align*} 
  which implies that if the maximum time-step ratio satisfies $0<r_{\max}\le\frac{3}{2}$, the energy dissipation law \eqref{eq: ac energy discrete} holds.
\end{proof}

\subsection{Conservative models}\label{sec:fully discrete}
We construct below a full discrete numerical scheme for the Wasserstein gradient flows \eqref{eq:wasserstein}. 

Define the admissible set $S_{ad}:=\{{\bm x}: \ x_{j+1}>x_{j}\ \text{for}\ j=0, 1, \cdots, M_x-1, \ \text{and} \ x_0=X_0,\ x_{M_x}=X_{M_x}\}$. 
Given ${\bm x}^{n-1}$, ${\bm x}^{n}$, $\tau_{n+1}$, $r_{n+1}$ and the initial value $\rho_0({\bm X})$, solving $({\bm x}^{n+1},\rho^{n+1})$ from the following finite dimensional minimization problem:
\begin{equation}\label{eq:discre optim11}
		\begin{cases}
			&\displaystyle{\bm x}^{n+1}:=\arg\inf_{{\bm x}\in S_{ad}}\frac{1+2r_{n+1}}{2\tau_{n+1}(1+r_{n+1})}\left(\rho_0({\bm X}),|{\bm x}-\hat{\bm x}^{n}|^2\right)_h+E_h({\bm x})\\
			&\qquad\qquad\qquad\qquad\qquad\displaystyle+\frac{\tau_{n+1}}{2}\left(|D_h( {\bm x}- {\bm x}^{n})|^2,1\right)_h,\\
			&\displaystyle\rho_{j+\frac{1}{2}}^{n+1}:=\frac{\rho_0(X_{j+\frac{1}{2}})}{(D_hx^{n+1})_{j+\frac{1}{2}}},\qquad j=0, 1, \cdots, M_x-1,
		\end{cases}
\end{equation}
	where ${\bm x}=(x_0,x_1,\ldots,x_{M_x-1},x_{M_x})$, 
	${x}_{j+\frac{1}{2}}:=\frac{1}{2}(x_{j+1}+x_{j})$, $\hat{\bm x}^{n}=\frac{(1+r_{n+1})^2}{1+2r_{n+1}}{\bm x}^{n}-\frac{r_{n+1}^2}{1+2r_{n+1}}{\bm x}^{n-1}$, and the discrete energy is defined by $	E_h({\bm x}):=\left(F\left(\frac{\rho_0({\bm X})}{D_h{\bm x}}\right),D_h{\bm x}\right)_h$.
	
Taking the  variational of \eqref{eq:discre optim11} with respect to ${\bm x}$ leads to the weak formulation with ${\bm x}={\bm x}^{n+1}$:
\begin{equation}\label{eq:weak formula}
\begin{aligned}
	&\dfrac{1+2r_{n+1}}{\tau_{n+1}(1+r_{n+1})}\left(\rho_0({\bm X})({\bm x}^{n+1}-\hat{\bm x}^{n}),\delta{\bm x}\right)_h+\tau_{n+1}\left(D_h({\bm x}^{n+1}-{\bm x}^{n}),D_h\delta{\bm x}\right)_h\\
	&\quad+\left(\frac{\delta E_h}{\delta {\bm x}}({\bm x}^{n+1}),\delta{\bm x}\right)_h=0.
\end{aligned} 
\end{equation}
Taking the test function $\delta{\bm x}=(\delta x_0,\delta x_1,\cdots,\delta x_{M_x-1},\delta x_{M_x})$ with $\delta x_i= \delta_{ij}$ (Dirac function) for $i,j=0,1\cdots,M_x-1$, we obtain the following second-order scheme with variable time steps:
\begin{align}
&\dfrac{1+2r_{n+1}}{2\tau_{n+1}(1+r_{n+1})}\rho_0(X_{j+\frac{1}{2}})(x_{j+\frac{1}{2}}^{n+1}-\hat{x}_{j+\frac{1}{2}}^{n})h+\dfrac{1+2r_{n+1}}{2\tau_{n+1}(1+r_{n+1})}\rho_0(X_{j-\frac{1}{2}})(x_{j-\frac{1}{2}}^{n+1}-\hat{x}_{j-\frac{1}{2}}^{n})h\nonumber\\
&-\tau_{n+1}\frac{(x_{j+1}^{n+1}-x_{j+1}^{n})-2(x_j^{n+1}-x_{j}^{n})+x_{j-1}^{n+1}-x_{j-1}^{n}}{h^2}h+\frac{\delta E_h}{\delta x_j}({\bm x}^{n+1})=0,\label{eq:scheme bdf2}\\
&\rho_{j+\frac{1}{2}}^{n+1}:=\frac{\rho_0(X_{j+\frac{1}{2}})}{(D_hx^{n+1})_{j+\frac{1}{2}}},\qquad j=0, 1, \cdots, M_x-1,\label{schem:2-1}
\end{align}
with the initial and boundary conditions
\begin{align}
{\bm x}^0=(X_0,X_1,\cdots,X_{M_x})\quad\text{and}\quad x_{0}^{n+1}=X_0,\quad x_{M_x}^{n+1}=X_{M_x}.\label{schem:1-3}
\end{align}
	We can use the first-order numerical scheme proposed in \cite{cheng2024flow} to calculate ${\bm x}^1$ and $\rho^1$.

\begin{theorem}\label{thm:1d}
	Assume the initial value $\rho_0(X)>0$ for $X\in\Omega_0^X$, the energy density $F(s)$ satisfies $F(s)\ge 0$ for $s\ge 0$, and $\lim\limits_{s\rightarrow 0}F(\frac{1}{s})s=\infty$,  
	then the scheme \eqref{eq:scheme bdf2}-\eqref{schem:2-1} has a unique solution $\bm{x}^{n+1}\in S_{ad}$ when $\frac{\delta^2 E_h}{\delta{\bm x}^2}>0$, 
	the solution is positive $\rho_{j+\frac{1}{2}}^{n+1}>0$ for $j=0, 1, \cdots, M_x-1$, and mass conservative. 
	Moreover, if the time-step ratio satisfies $0<r_{n+1}\le r_{\max}\le\frac{3+\sqrt{17}}{2}$, the energy is dissipative: 
	\begin{equation}
	\begin{split}
	&E_h({\bm x}^{n+1})+\dfrac{r_{\max}}{2\tau_{n+1}(1+r_{\max})}\left(\rho_0({\bm X})({\bm x}^{n+1}-{\bm x}^{n})^2,1\right)_h\\
	\le& E_h({\bm x}^{n})+\dfrac{r_{\max}}{2\tau_{n}(1+r_{\max})}\left(\rho_0({\bm X})({\bm x}^{n}-{\bm x}^{n-1})^2,1\right)_h.
	\end{split}
	\end{equation}
\end{theorem}
\begin{proof}
	The first two properties can be derived using similar technique in \cite{cheng2020new}, the proof of which is omitted here. To prove the energy dissipation law, we set $\delta{\bm x}={\bm x}^{n+1}-{\bm x}^{n}$ in \eqref{eq:weak formula} and obtain
	\begin{align*}
	&\dfrac{1+2r_{n+1}}{\tau_{n+1}(1+r_{n+1})}\left(\rho_0({\bm X})({\bm x}^{n+1}-{\bm x}^{n})^2,1\right)_h-	\dfrac{r_{n+1}^2}{\tau_{n+1}(1+r_{n+1})}\left(\rho_0({\bm X})({\bm x}^{n}-{\bm x}^{n-1}),{\bm x}^{n+1}-{\bm x}^{n}\right)_h\\
	&	+\tau_{n+1}(|D_h({\bm x}^{n+1}-{\bm x}^{n})|^2,1)_h+\left(\frac{\delta E_h}{\delta {\bm x}}({\bm x}^{n+1}),{\bm x}^{n+1}-{\bm x}^{n}\right)_h=0,
	\end{align*}
	using the Cauchy-Schwarz inequality and the convexity of $E({\bm x})$ with respect to ${\bm x}$ shows that
	\begin{align*}
	&\dfrac{2+4r_{n+1}-r_{n+1}^2}{2\tau_{n+1}(1+r_{n+1})}\left(\rho_0({\bm X})({\bm x}^{n+1}-{\bm x}^{n})^2,1\right)_h-		\dfrac{r_{n+1}^2}{2\tau_{n+1}(1+r_{n+1})}\left(\rho_0({\bm X})({\bm x}^{n}-{\bm x}^{n-1})^2,1\right)_h\\
	+&\tau_{n+1}(|\nabla_h({\bm x}^{n+1}-{\bm x}^{n})|^2,1)_h+E_h({\bm x}^{n+1})\le E_h({\bm x}^{n}),
	\end{align*}
	then we have
	\begin{align*}
	&E_h({\bm x}^{n+1})+\dfrac{2+4r_{n+1}-r_{n+1}^2}{2\tau_{n+1}(1+r_{n+1})}\left(\rho_0({\bm X})({\bm x}^{n+1}-{\bm x}^{n})^2,1\right)_h\\
	\le& E_h({\bm x}^{n})+		\dfrac{r_{\max}}{2\tau_{n}(1+r_{\max})}\left(\rho_0({\bm X})({\bm x}^{n}-{\bm x}^{n-1})^2,1\right)_h.
	\end{align*}
	Similar to the semi-discrete case discussed above, the energy dissipation law can be derived under the condition that $	g(r_{\max})\ge\frac{r_{\max}}{1+r_{\max}}$, which holds for  $0<r_{\max}\le\frac{3+\sqrt{17}}{2}$.
\end{proof}

\begin{remark}
	In fact, the energy densities of the PME and Fokker-Planck equation are  $F(s)=\frac{1}{m-1}s^m$, $m>1$ and  $F(s)=s\log s+sV(x)$, respectively. 
	 Both satisfy the assumption that  $F(s)\ge 0$ for $s\ge 0$, and $\lim\limits_{s\rightarrow 0}F(\frac{1}{s})s=\infty$.
\end{remark}

	\subsection{Adaptive time-stepping methods in 2D}\label{sec:two dim}
Now we generalize adaptive time-stepping, Lagrangian methods based on flow dynamic approach  into two-dimensional case. 

Denote ${\bm x}=(x,y)$, ${\bm X}=(X,Y)$, and the Jacobian matrix $\frac{\partial {\bm x}}{\partial {\bm X}}=\frac{\partial (x,y)}{\partial (X,Y)}$. 
Set $\Omega_0^X=[-L_x,L_x]\times[-L_y,L_y]$ with $L_x$, $L_y>0$, and $\Omega_0^x=\Omega_0^X$. 
Given $M_x$, $M_y\in\mathbb{N}$, and define the spatial grid size $h_x=\frac{2L_x}{M_x}$, $h_y=\frac{2L_y}{M_y}$. 
Let $X_{ij}=X_0+jh_x$, $Y_{ij}=Y_0+ih_y$ for $0\le j\le M_x$, $ 0\le i\le M_y$, and $\rho_{ij}^0=\rho(X_{ij},Y_{ij},0)\ge 0$.


Given ${\bm x}^0={\bm X}$ and the initial value $\rho_0({\bm X})$, the first step ${\bm x}^1$ is calculated by the first-order scheme in \cite{cheng2024flow}. Then the fully implicit BDF2 scheme with variable time steps can be proposed: 

{\bf Implicit numerical scheme}. 
$\forall n\ge1$, given ${\bm x}^{n-1}$, ${\bm x}^{n}$, $\tau_{n+1}$, $r_{n+1}$, solving $(x^{n+1},y^{n+1})$ from
\begin{align}
	&\rho_{ij}^0D_2x_{ij}^{n+1}-\varepsilon\tau_{n+1}\Delta_{\bm X}(x^{n+1}_{ij}-x^{n}_{ij}) +\frac{\delta \tilde{E}_{h,2}}{\delta x}({\bm x}_{ij}^{n+1})=0,\label{eq:wasserstein 2d-1}\\
	&\rho_{ij}^0D_2y_{ij}^{n+1}-\varepsilon\tau_{n+1}\Delta_{\bm X}(y^{n+1}_{ij}-y^{n}_{ij}) +\frac{\delta \tilde{E}_{h,2}}{\delta y}({\bm x}_{ij}^{n+1})=0,\label{eq:wasserstein 2d-2}
\end{align}
where $D_2a_{ij}^{n+1}:=\frac{(1+2r_{n+1})a_{ij}^{n+1}-(1+r_{n+1})^2a_{ij}^n+r_{n+1}^2a_{ij}^{n-1}}{\tau_{n+1}(1+r_{n+1})}$ and
$\tilde{E}_{h,2}({\bm x}):=\sum_{i,j}F(\frac{\rho_{ij}^0}{\text{det}\frac{\partial {\bm x}}{\partial {\bm X}}|_{ij}})\text{det}\frac{\partial {\bm x}}{\partial {\bm X}}|_{ij}$. 
The Dirichlet boundary condition is considered: $ {\bm x }^{n+1}|_{\partial\Omega}={\bm X}|_{\partial\Omega}$. 
Then $\rho_{ij}^{n+1}$ is derived by
\begin{align}\label{eq:rho in 2d}
	\rho_{ij}^{n+1}=\frac{\rho_{ij}^0}{\mathcal{F}_{ij}^{n+1}}\quad\text{with}\quad 
	\mathcal{F}_{ij}^{n+1}=\left |\begin{array}{cc}
		\frac{\partial x_{ij}^{n+1}}{\partial X} &\frac{\partial y_{ij}^{n+1}}{\partial X}   \\
		\frac{\partial x_{ij}^{n+1}}{\partial Y} &\frac{\partial y_{ij}^{n+1}}{\partial Y}   \\
	\end{array}\right|=\left |\begin{array}{cc}
		\frac{ x_{i,j+1}^{n+1}-x_{i,j-1}^{n+1}}{2h_x} &\frac{ y_{i,j+1}^{n+1}-y_{i,j-1}^{n+1}}{2h_x}   \\
		\frac{ x_{i+1,j}^{n+1}-x_{i-1,j}^{n+1}}{2h_y} &\frac{y_{i+1,j}^{n+1}-y_{i-1,j}^{n+1}}{2h_y}   \\
	\end{array}\right|.
\end{align}
The scheme \eqref{eq:wasserstein 2d-1}-\eqref{eq:wasserstein 2d-2} should be solved in the admissible set $E_{ad}=\{{\bm x}:\ \text{det}\frac{\partial {\bm x}}{\partial {\bm X}}|_{ij}>0 \ \text{for all}\ i,j\in\mathbb{N}, \ {\bm x}|_{\partial\Omega} ={\bm X}|_{\partial\Omega}\}$, and the solution to which is the minimizer of the following minimization problem: 
\begin{align}\label{min:2d}
	{\bm x}^{n+1}:=\arg\inf_{\bm{x}\in E_{ad}}J_{n+1}({\bm x}),
\end{align}
where $J_{n+1}({\bm x}):=E_{h,2}({\bm x})+\sum_{i,j}\frac{1+2r_{n+1}}{2\tau_{n+1}(1+r_{n+1})}\rho_{ij}^0|{\bm x}_{ij}-\hat{\bm x}_{ij}^{n}|^2h_xh_y+\frac{\varepsilon\tau_{n+1}}{2}\sum_{i,j}(|\frac{{\bm x}_{i,j+1}-{\bm x}_{i,j+1}^{n}-{\bm x }_{i,j}+{\bm x}_{i,j}^{n}}{h_x}|^2+|\frac{{\bm x}_{i+1,j}-{\bm x}_{i+1,j}^{n}-{\bm x}_{i,j}+{\bm x}_{i,j}^{n}}{h_y}|^2)h_xh_y$, 
and $E_{h,2}({\bm x}):=\tilde{E}_{h,2}({\bm x})h_xh_y$. We obtain the following  result for the scheme \eqref{eq:wasserstein 2d-1}-\eqref{eq:wasserstein 2d-2}.
\begin{theorem}
	Assume the energy density $F(s)\ge 0$ for $s\ge 0$, and satisfies $\lim\limits_{s\rightarrow 0}F(\frac{1}{s})s=\infty$. 
	Then there exists a solution ${\bm x}^{n+1}\in E_{ad}$ to the nonlinear numerical scheme \eqref{eq:wasserstein 2d-1}-\eqref{eq:wasserstein 2d-2}, and the following energy dissipation law holds under the condition that $0<r_{n+1}\le r_{\max}\le\frac{5}{4}$:
	\begin{equation}
		\begin{aligned}
			&E_{h,2}({\bm x}^{n+1})+\sum_{i,j}\frac{r_{\max}^3}{\tau_{n+1}(1+r_{\max})(1+2r_{\max})}\rho_{ij}^0|{\bm x}_{ij}^{n+1}-{\bm x}_{ij}^{n}|^2h_xh_y\\
			\le& E_{h,2}({\bm x}^{n})+\sum_{i,j}\frac{r_{\max}^3}{\tau_{n}(1+r_{\max})(1+2r_{\max})}\rho_{ij}^0|{\bm x}_{ij}^{n}-{\bm x}_{ij}^{n-1}|^2h_xh_y.
		\end{aligned}
	\end{equation}
\end{theorem}
\begin{proof}
	The existence of the solution to scheme \eqref{eq:wasserstein 2d-1}-\eqref{eq:wasserstein 2d-2} is equivalent to the existence of the minimizer of $J_{n+1}({\bm x})$ in the admissible set $E_{ad}$. Then we turn to prove the existence of the minimizer of the minimization problem \eqref{min:2d}. If the minimizer lies on the boundary of the admissible set, that is ${\bm x}\in \partial E_{ad}$, we have $J_{n+1}({\bm x})=\infty$, which is a contradiction. Following the proof in \cite{liu2020lagrangian,carrillo2018lagrangian}, the claim of the theorem will be derived once we show that the sub-level set
	$$	\mathcal{S}:=\Big\{{\bm x}\in E_{ad}:\ J_{n+1}({\bm x})\le E_{h,2}({\bm x}^{n})+\sum_{i,j}\frac{r_{n+1}^4}{2\tau_{n+1}(1+r_{n+1})(1+2r_{n+1})}\rho_{ij}^0|{\bm x}_{ij}^{n}-{\bm x}_{ij}^{n-1}|^2h_xh_y:=\gamma\Big\}$$
	is a non-empty compact subset of $\mathbb{R}^2$. Clearly, ${\bm x}^{n}\in \mathcal{S}$, so it is non-empty. Similar to \cite{cheng2024flow, liu2020lagrangian}, the boundedness and closedness of $\mathcal{S}$ can be established. Consequently, $\mathcal{S}$ is a non-empty compact subset of $\mathbb{R}^2$. 
	
	If ${\bm x}^{n+1}\in \mathcal{S}$ is a minimizer of the minimization problem \eqref{min:2d}, we have
	\begin{equation}\label{eq:energy2d}
		\begin{aligned}
			&E_{h,2}({\bm x}^{n+1})+\sum_{i,j}\frac{1+2r_{n+1}}{2\tau_{n+1}(1+r_{n+1})}\rho_{ij}^0|{\bm x}_{ij}^{n+1}-\hat{\bm x}_{ij}^{n}|^2h_xh_y\\
			\le& E_{h,2}({\bm x}^{n})+\sum_{i,j}\frac{r_{n+1}^4}{2\tau_{n+1}(1+r_{n+1})(1+2r_{n+1})}\rho_{ij}^0|{\bm x}_{ij}^{n}-{\bm x}_{ij}^{n-1}|^2h_xh_y.
		\end{aligned}
	\end{equation}
	Using the inequality $|a-b|^2\ge\frac{1}{2}|a|^2-|b|^2$ with $a={\bm x}^{n+1}_{ij}-{\bm x}^{n}_{ij}$ and $b=\frac{r_{n+1}^2}{1+2r_{n+1}}({\bm x}^{n}_{ij}-{\bm x}^{n-1}_{ij})$ shows that $	\frac{1}{2}|{\bm x}_{ij}^{n+1}-{\bm x}_{ij}^{n}|^2-\frac{r_{n+1}^4}{(1+2r_{n+1})^2}|{\bm x}_{ij}^{n}-{\bm x}_{ij}^{n-1}|^2\le |{\bm x}_{ij}^{n+1}-\hat{\bm x}_{ij}^{n}|^2$, 
	substituting which into \eqref{eq:energy2d} leads to
	\begin{align*}
		&E_{h,2}({\bm x}^{n+1})+\sum_{i,j}\frac{1+2r_{n+1}}{4\tau_{n+1}(1+r_{n+1})}\rho_{ij}^0|{\bm x}_{ij}^{n+1}-{\bm x}_{ij}^{n}|^2h_xh_y\nonumber\\
		\le& E_{h,2}({\bm x}^{n})+\sum_{i,j}\frac{r_{n+1}^4}{\tau_{n+1}(1+r_{n+1})(1+2r_{n+1})}\rho_{ij}^0|{\bm x}_{ij}^{n}-{\bm x}_{ij}^{n-1}|^2h_xh_y\nonumber\\
		\le& E_{h,2}({\bm x}^{n})+\sum_{i,j}\frac{r_{\max}^3}{\tau_{n}(1+r_{\max})(1+2r_{\max})}\rho_{ij}^0|{\bm x}_{ij}^{n}-{\bm x}_{ij}^{n-1}|^2h_xh_y.\nonumber
	\end{align*}
	The energy dissipation law will hold under the condition that the maximum time-step ratio satisfies $		\frac{1+2r_{n+1}}{4(1+r_{n+1})}>\frac{1}{4}\ge \frac{r_{\max}^3}{(1+r_{\max})(1+2r_{\max})}$. 
	If the time-step ratio is chosen such that $0<r_{n+1}\le r_{\max}\le\frac{5}{4}$, where $\frac{5}{4}$ is computed numerically, then the proof is complete.
\end{proof}



	{\bf Explicit numerical scheme}. 
	Given ${\bm x}^0={\bm X}$ and the initial value $\rho_0({\bm X})$, the first step  $({\bm x}^1,\rho^1)$ is solved by the first-order scheme proposed in \cite{cheng2024flow}. Given $(x^{n},y^{n})$, $\forall n\ge1$, $(x^{n+1},y^{n+1})$ can be solved by the linear scheme: 
	\begin{align}
		&\rho_{ij}^0D_2x_{ij}^{n+1}-\varepsilon\tau_{n+1}\Delta_{\bm X}(x^{n+1}_{ij}-x^{n}_{ij})+\frac{\delta \tilde{E}_{h,2}}{\delta x}((1+r_{n+1}){\bm x}_{ij}^{n}-r_{n+1}{\bm x}_{ij}^{n-1})=0,\label{scheme:2d explicit1}\\
		&\rho_{ij}^0D_2y_{ij}^{n+1}-\varepsilon\tau_{n+1}\Delta_{\bm X}(y^{n+1}_{ij}-y^{n}_{ij})+\frac{\delta \tilde{E}_{h,2}}{\delta y}((1+r_{n+1}){\bm x}_{ij}^{n}-r_{n+1}{\bm x}_{ij}^{n-1})=0,\label{scheme:2d explicit2}
	\end{align}
	with the Dirichlet boundary condition $ {\bm x }^{n+1}|_{\partial\Omega}={\bm X}|_{\partial\Omega}$, and the last term $\frac{\delta \tilde{E}_{h,2}}{\delta{\bm x}}$ is an second-order extrapolation term.  Then the solution $\rho^{n+1}$ will be derived by \eqref{eq:rho in 2d}.

\section{Adaptive time-stepping strategy}\label{sec:strategy}
In this section, we describe our adaptive time-step strategy in 
Algorithm~\ref{algo,2} which determines time steps based on   the changes of trajectory and energy  \cite{huang2020parallel,hou2023implicit}. 
To be specific, the parameter $\gamma$ in \eqref{strategy1} reflects the sensitivity of the time step with respect to changes of trajectory, and $\beta$ in \eqref{strategy2} represents the sensitivity of the time step with respect to changes of energy.

\begin{algorithm}[!htbp]
	\caption{Adaptive time-stepping algorithm related to energy and trajectory changes}
	\begin{algorithmic}[1]\label{algo,2}
		\STATE Initialize $\tau_{n}$, ${\bm x}^{n-1}$, ${\bm x}^{n}$, $E^{n-1}$, $E^{n}$, $\tau_{\min}$, $\tau_{\max}$, $r_{\text{user}}$
		\STATE Set tolerances $\text{tol}$
		\STATE Compute the next time-step by the following strategy with $\delta_\tau a^{n}:=\frac{a^{n}-a^{n-1}}{\tau_{n}}$:
		\begin{align}
		&\tau_{n+1}=\min\left(\max\left(\tau_{\min},\frac{\tau_{\max}}{\sqrt{1+\gamma \|\delta_\tau{\bm x}^{n}\|^2}}\right),r_{\text{user}}\tau_{n}\right),\label{strategy1}\\
			\text{or}\quad&\tau_{n+1}=\min\left(\max\left(\tau_{\min},\frac{\tau_{\max}}{\sqrt{1+\beta |\delta_\tau E_h^{n}|^2}}\right),r_{\text{user}}\tau_{n}\right)\label{strategy2}
		\end{align}
		\STATE Compute ${\bm x}^{n+1}_{\text{2nd}}$ and $\rho^{n+1}_\text{2nd}$ by solving scheme \eqref{eq:scheme bdf2}-\eqref{schem:1-3} with time-step size $\tau_{n+1}$ 
		\STATE Compute 
		the Jacobian determinant $\mathrm{det}\frac{\partial {\bm x}^{n+1}}{\partial{\bm X}}$
		\IF{
			$\mathrm{det}\frac{\partial {\bm x}^{n+1}}{\partial{\bm X}}>0$}
		\STATE Accept step: ${\bm x}^{n+1} \gets {\bm x}^{n+1}_{\text{2nd}}$, $\rho^{n+1} \gets \rho^{n+1}_{\text{2nd}}$
		\ELSE
		\STATE Reject step and decrease time-step: $\tau_{n+1}=\frac{1}{2}\tau_{n+1}$ 
		and return to step 4
		\ENDIF
		\STATE \textbf{return} ${\bm x}^{n+1}$, $\rho^{n+1}$, $E^{n+1}$ and $\tau_{n+1}$
	\end{algorithmic}
\end{algorithm}

\section{Numerical experiments}\label{sec:num}
We present in this section some numerical experiments  to validate the accuracy and stability of the proposed numerical schemes.


\subsection{Non-conservative models in one dimension}
We first present numerical experiments for non-conservative models in one-dimension. 

 \subsubsection{Allen-Cahn equation}
  We will take the initial value $\rho_0(X)=1-X^2$, $X\in[-1,1]$, $\epsilon=0.01$, choose $\mathcal{M}(\rho)\equiv 1$ and $\mathcal{M}(\rho)=1-\rho^2$ to solve numerical solutions using scheme \eqref{scheme:ac fully discrete}.
  
{\bf convergence test.} Set $M_x=16\times 2^{i-1}$ and $N=625\times 2^{i-1}$ for $i=1$, 2, 3,  $4$, the reference numerical solution is computed with $i=7$. For the fixed time-step case, we set the time step $\tau=\frac{T}{N}$ and spatial size $\delta X=\frac{2}{M_x}$, and compute the convergence order by $\text{Order}(i)=\ln\left(\frac{\text{error}(i)}{\text{error}(i-1)}\right)/\ln\left(\frac{M_x(i)}{M_x(i-1)}\right)$. For the variable time-step case, the time step is chosen by $\tau_n=\frac{\sigma_nT}{\sum_{k=1}^{N}\sigma_k}$ with uniformly distributed random values $\sigma_n\in(0,1)$, $\forall n$, the convergence order is calculated by $	\text{Order}(i)=\ln\left(\frac{\text{error}(i)}{\text{error}(i-1)}\right)/\ln\left(\frac{\tau(i)}{\tau(i-1)}\right)$, where $\tau(i)$ represents the maximum time step. Numerical results displayed in Figure~\ref{fig:ac order} show that the convergence order of \eqref{scheme:ac fully discrete} is second order.
 \begin{figure}[!htb]
 	\centering
 	\subfigure[Fixed time step]{	
 	\includegraphics[width=0.4\textwidth]{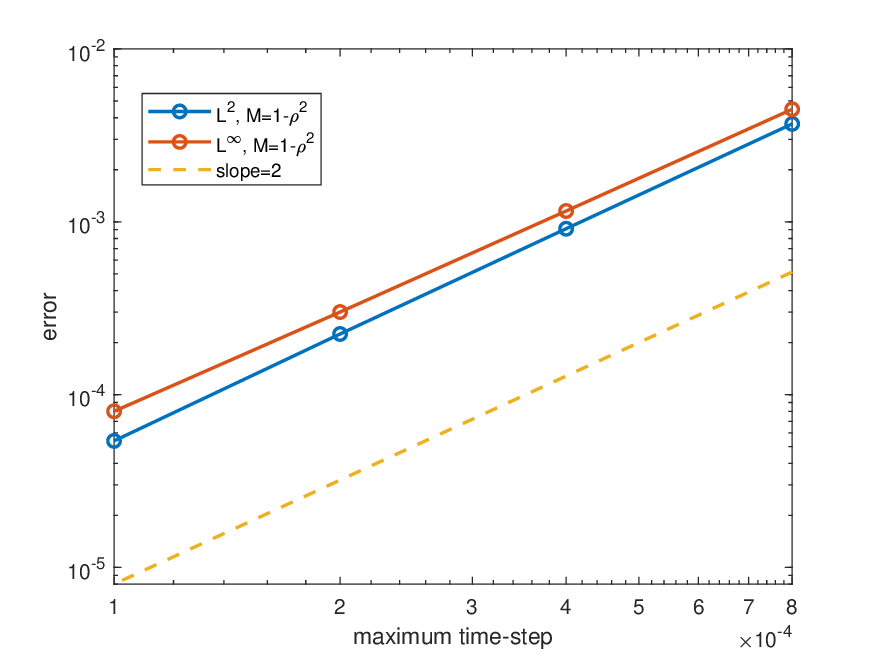}
 	}
	\subfigure[Variable time steps]{
	\includegraphics[width=0.4\textwidth]{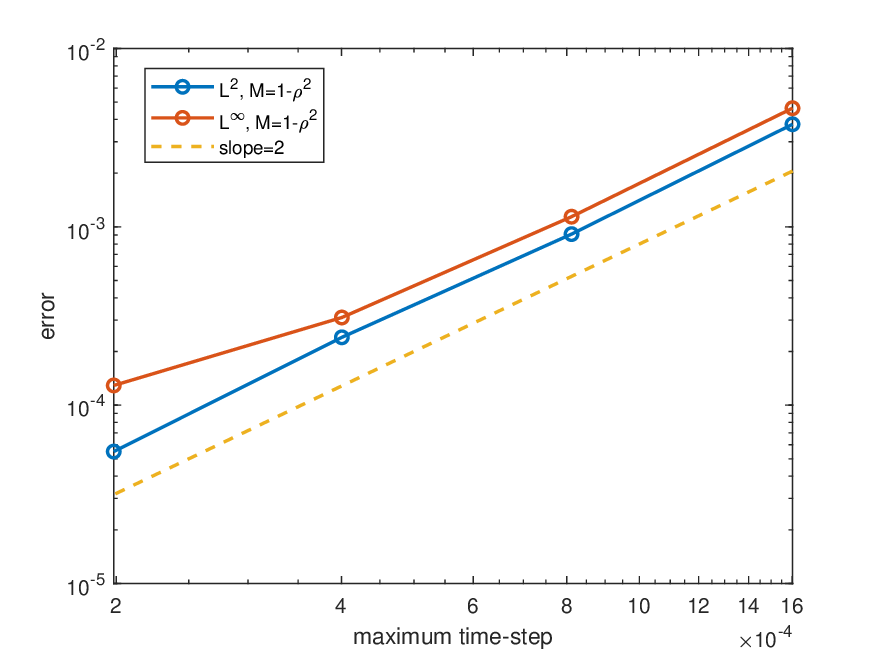}
	}
 	\caption{Convergence order for \eqref{scheme:ac fully discrete} with  $\rho_0(X)=1-X^2$, $X\in[-1,1]$, $\epsilon=0.01$, $\eta=0$, $T=0.5$. }\label{fig:ac order}
\end{figure}

{\bf Interface capture.} 
We first choose the time steps by \( \tau_n = \frac{\sigma_n T}{\sum_{k=1}^{N} \sigma_k} \), with uniformly distributed random values \( \sigma_n \in (0,1) \), \(\forall n \). The numerical results are shown in diagrams (a) to (f) in Figure~\ref{fig:ac2}. The interface width and particle positions are captured, and the discrete energy is also dissipative. Furthermore, we use strategy \eqref{strategy2} to solve the scheme \eqref{scheme:ac fully discrete} with \( \epsilon = 0.01 \) and \( \mathcal{M}(\rho) \equiv 1 \). The energy plot is displayed in in diagrams (g) in Figure~\ref{fig:ac2}. The time steps (blue line) and the corresponding time-step ratios (red line) are plotted in diagrams (h) in Figure~\ref{fig:ac2}. It can be observed that the time step increases as the energy decreases slowly, which improves the efficiency.
 	\begin{figure}[!htb]
 	\centering
 	\subfigure[Interface, $M\equiv 1$]{
 		\includegraphics[width=0.22\textwidth]{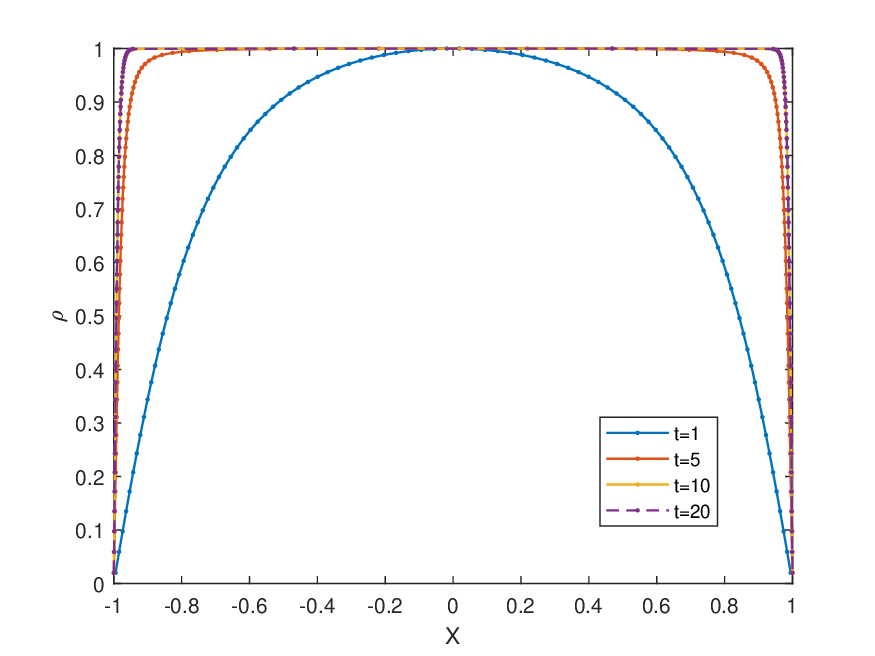}
 	}
 	\subfigure[Interface, $M=1-\rho^2$]{
 		\includegraphics[width=0.22\textwidth]{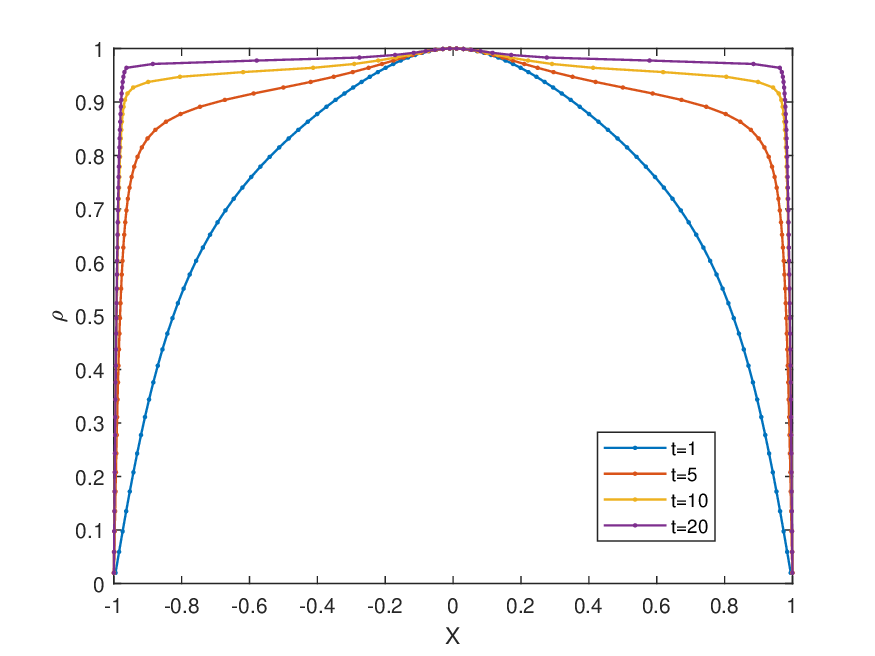}
 	}
 	\subfigure[Position, $M\equiv 1$]{
 		\includegraphics[width=0.22\textwidth]{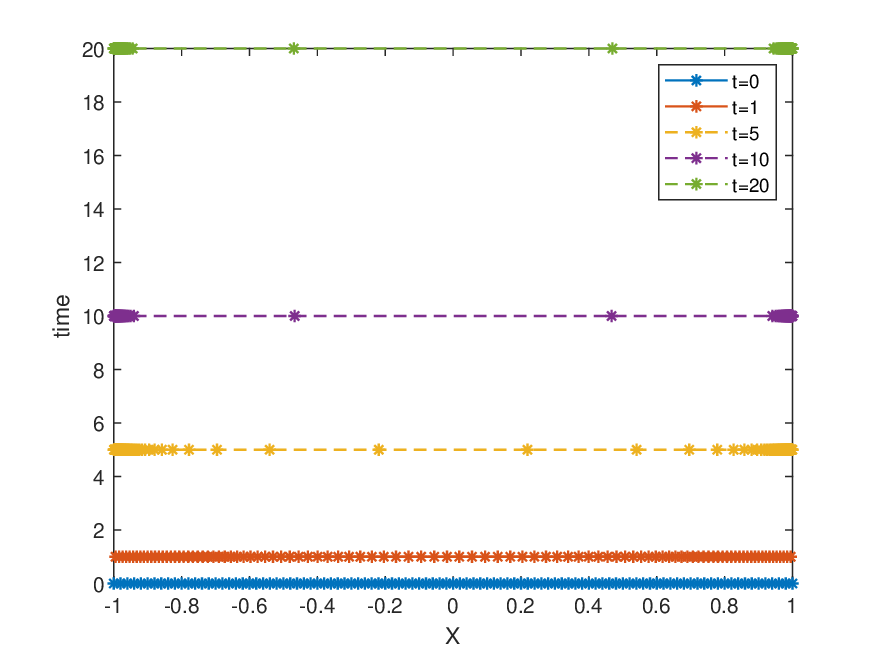}
 	}
 	\subfigure[Position, $M=1-\rho^2$]{
 		\includegraphics[width=0.22\textwidth]{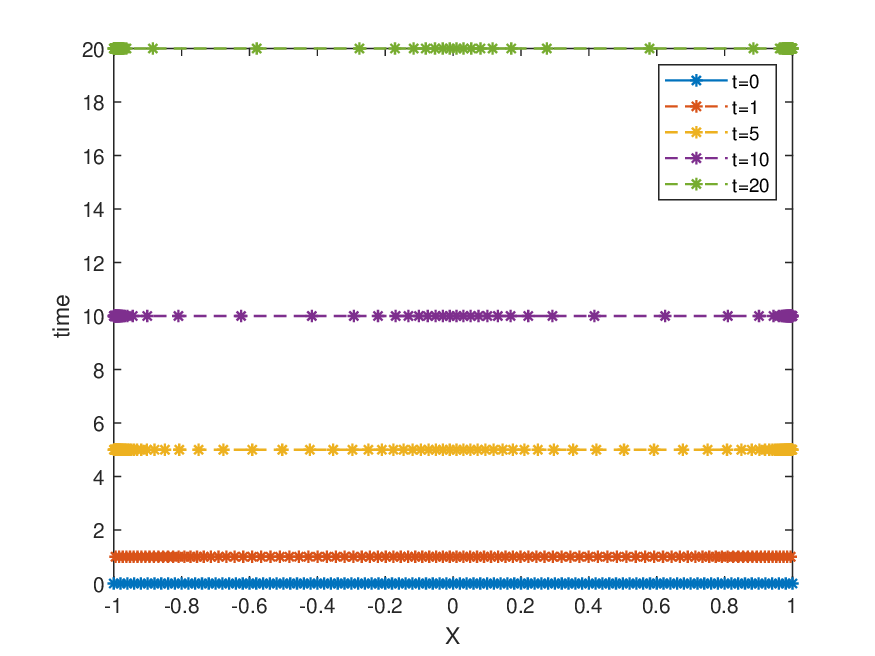}
 	}
  	\subfigure[Interface at $T=20$]{
 	\includegraphics[width=0.22\textwidth]{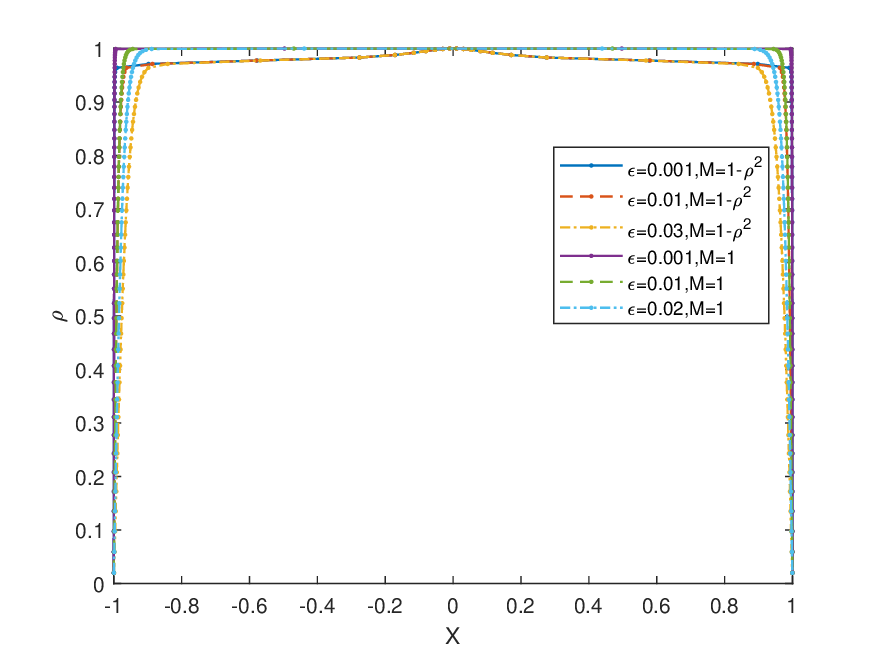}
 }
 	\subfigure[Energy]{
 		\includegraphics[width=0.22\textwidth]{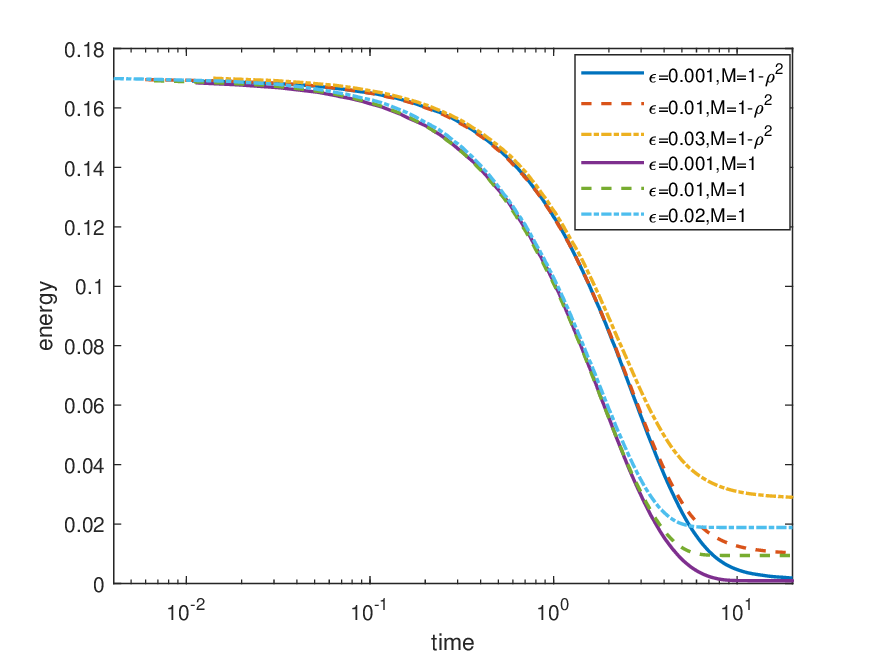}
 	}
  	\subfigure[Energy, strategy \eqref{strategy2}]{
 	\includegraphics[width=0.22\textwidth]{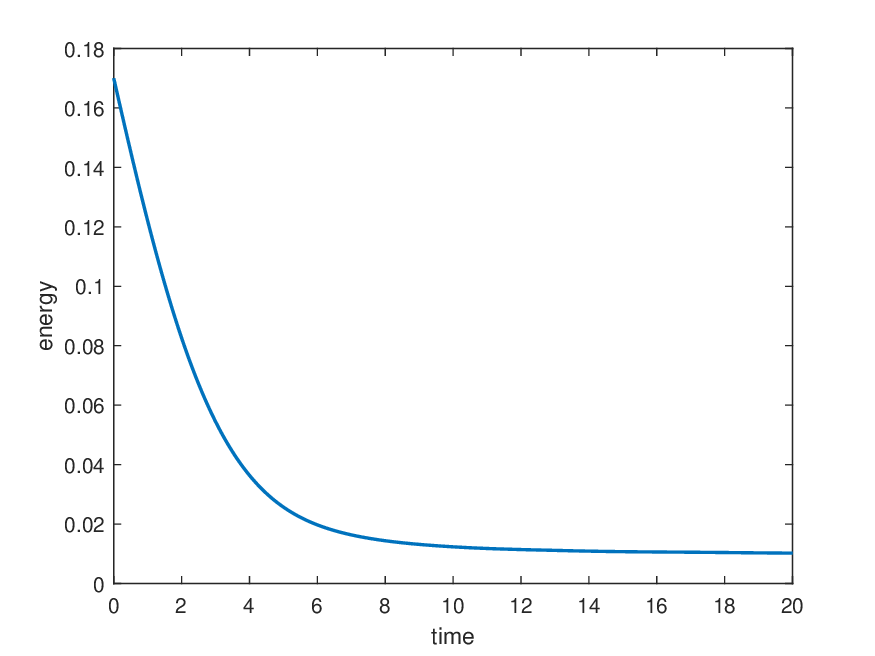}
 }
 \subfigure[Time step with \eqref{strategy2}]{
 	\includegraphics[width=0.22\textwidth]{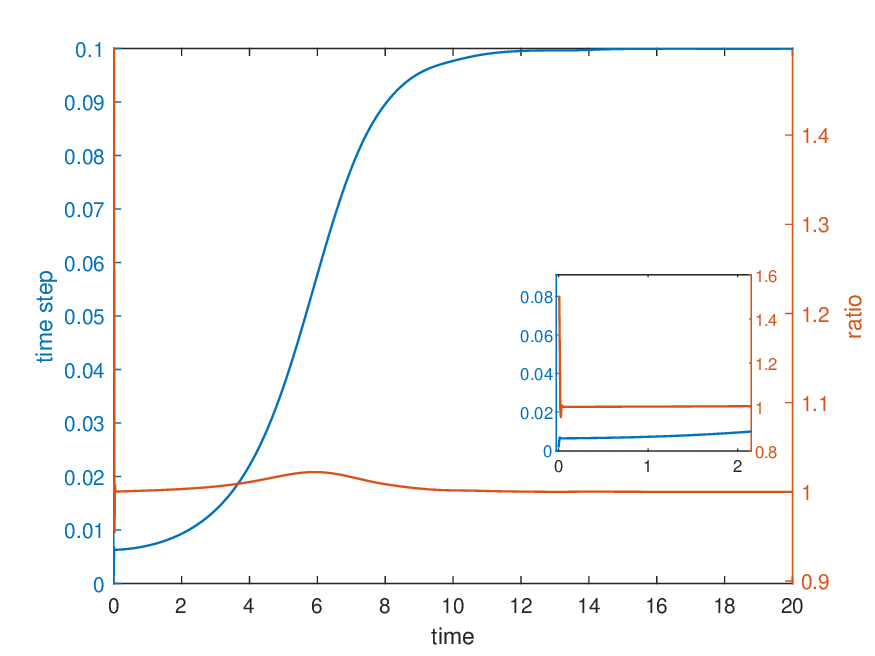}
 }
 	\caption{Numerical solution solved by scheme \eqref{scheme:ac fully discrete} at $T=20$ with $\epsilon=0.01$, $M_x=100$, $N=2000$.  (a-f): $\tau_n=\frac{\sigma_nT}{\sum_{k=1}^{N}\sigma_k}$. (g-h): strategy \eqref{strategy2} with $\beta=1e5$, $r_{\text{user}}=1.5$,  $\tau_{\max}=0.1$ and $\tau_{\min}=1e-3$.}\label{fig:ac2}
 \end{figure}

\subsection{Conservative models in one dimension}
Now, we present numerical experiments for conservative models in one-dimension. 

\subsubsection{Porous medium equation}
The PME is a Wasserstein gradient flow with energy $E(\rho)=\int_{\Omega}\frac{1}{m-1}\rho^m\mathrm{d}x$. 

{\bf Convergence test.} 
Consider the following smooth initial value: 
\begin{align}
	\rho_0(x)=\cos\left(\frac{\pi x}{2}\right),\qquad x\in[-1,1],
\end{align}
with the Dirichlet boundary condition $x|_{\partial\Omega}=X|_{\partial\Omega}$. To calculate the convergence order, we choose the reference solution computed under a much fine mesh with $M_x=10000$, $N=20000$. 

The convergence order of the scheme \eqref{eq:scheme bdf2}-\eqref{schem:1-3} with both fixed and variable time steps  are displayed in Table~\ref{convergence fix} and Table~\ref{convergence fix_variable}, both of which are clearly second order. It can be observed that even when the maximum time-step ratio exceeds the theoretical value \( r_{\max} \), the numerical results remain stable and accurate. This suggests that the limitation on the maximum time-step ratio  in the theoretical analysis is most likely pessimistic.

\begin{table}[!htb]
	\centering
	\caption{Convergence order for the PME with $m=2$ at $T=0.5$ under fixed time-step.}
		\begin{tabular}{cccccccccc}
			\hline
			$M_x$ &$\delta t$&$L_h^{2}$ error ($x$) & order & $L^{\infty}$ error ($x$) & order &\ $L_h^{2}$ error ($u$) & order\\ \hline
			100 &1/200& 8.7715e-05	& & 1.2351e-04& & 3.8165e-05&\\
			200 &1/400& 2.1936e-05&1.9995 &3.0882e-05&1.9998&9.5315e-06 &2.0015\\
			400 &1/800& 5.4852e-06&1.9997& 7.7083e-06 &2.0023&	2.3792e-06& 2.0023\\
			800 &1/1600& 1.3715e-06&1.9998&  1.9271e-06&2.0000&5.9187e-07  &2.0071\\
			\hline
		\end{tabular}\label{convergence fix}
\end{table}
\begin{table}[!htb]
	\centering
	\caption{Convergence order for the PME with $m=2$ at $T=0.5$ under variable time steps.}
		\resizebox{\linewidth}{!}{
	\begin{tabular}{cccccccccc}
		\hline
		$M_x$ &max time-step&max time-ratio &$L_h^{2}$ error ($x$) & order & $L^{\infty}$ error ($x$) & order &\ $L_h^{2}$ error ($u$) & order\\ \hline
		100 &0.0051 & 73.7835&8.8319e-05	& & 1.2351e-04& & 3.7623e-05&\\
		200 &0.0025 & 971.0994&2.2082e-05&1.9179 & 3.0882e-05&1.9179 &9.4009e-06 &1.9188\\
		400 &0.0013 & 369.4026&5.5279e-06&2.0927 &  7.7084e-06 &2.0970&2.3422e-06& 2.0999\\
		800 &6.2751e-04 &4.1802e+03 &1.3816e-06&1.9586&1.9271e-06&1.9582&5.8307e-07  &1.9642\\
		\hline
	\end{tabular}\label{convergence fix_variable}}
\end{table}

{\bf Waiting time.} 
The propagation speed at the boundary for PME can be calculated by
\begin{align}
\partial_tx=-\frac{m}{m-1}\frac{\partial_X(\rho(X,0))^{m-1}}{(\partial_Xx)^{m}},
\end{align}
see \cite{duan2019pme,duan2021structure,liu2023envara}. The numerical waiting time can be calculated as the first instance such that $\partial_tx\neq0$, see the related works \cite{duan2019pme,cheng2024flow}. We set the  initial value to be
\begin{align}\label{initial:wt}
	\rho_0(x)=\left(\frac{m-1}{m}\left((1-\theta)\sin^2(x)+\theta\sin^4(x)\right)\right)^{1/(m-1)},\qquad x\in[-\pi,0],
\end{align}
where $\theta\in[0,0.25]$.  The waiting time for \eqref{initial:wt} is given in \cite{aronson1983initially} by $t_{w,e}:=\frac{1}{2(m+1)(1-\theta)}$. 

 We apply the Algorithm \ref{algo,2} with different time-adaptive strategies, \eqref{strategy1} and \eqref{strategy2}, in the following numerical experiments with $r_{\text{user}}=1.4$, $\tau_{\min}=1e-6$ and $\tau_{\max}=1e-2$ in \eqref{strategy1}, and $r_{\text{user}}=1.4$, $\tau_{\min}=1e-6$ and $\tau_{\max}=5\times1e-3$ in \eqref{strategy2}. We take $\gamma=0.1,\ 1,\ 10,\ 100$ and $\beta=0.1,\ 1,\ 10,\ 100$, and display  the results  in Figure~\ref{fig:adaptive}. It can be noticed that the time-step increases after the waiting time, and the efficiency is improved compared to the scheme with a fixed time-step.
\begin{figure}[!htb]
	\centering
	\subfigure[Time-step with \eqref{strategy1}]{
		\includegraphics[width=0.22\textwidth]{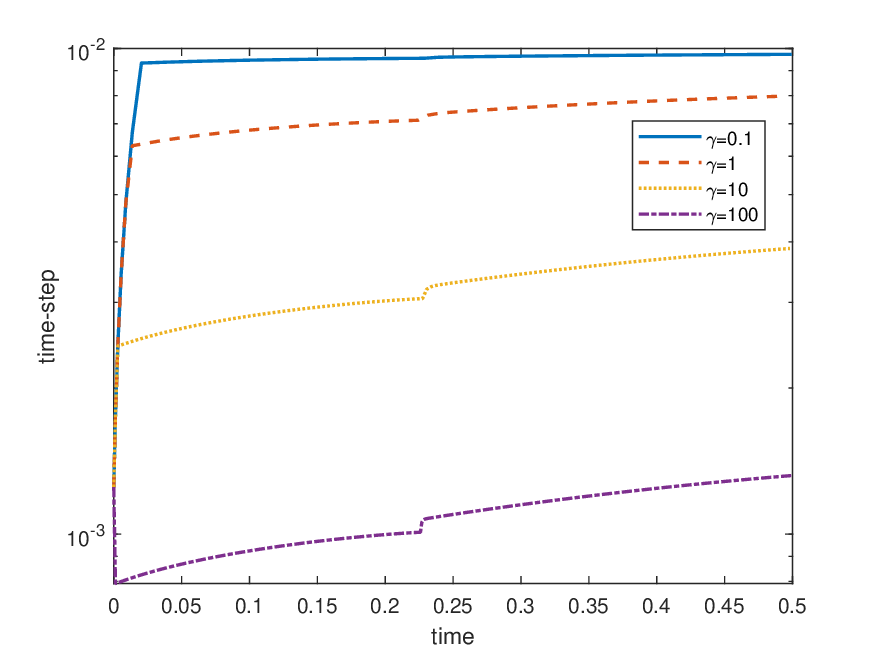}
	}
	\subfigure[Time-step ratio, \eqref{strategy1}]{
		\includegraphics[width=0.22\textwidth]{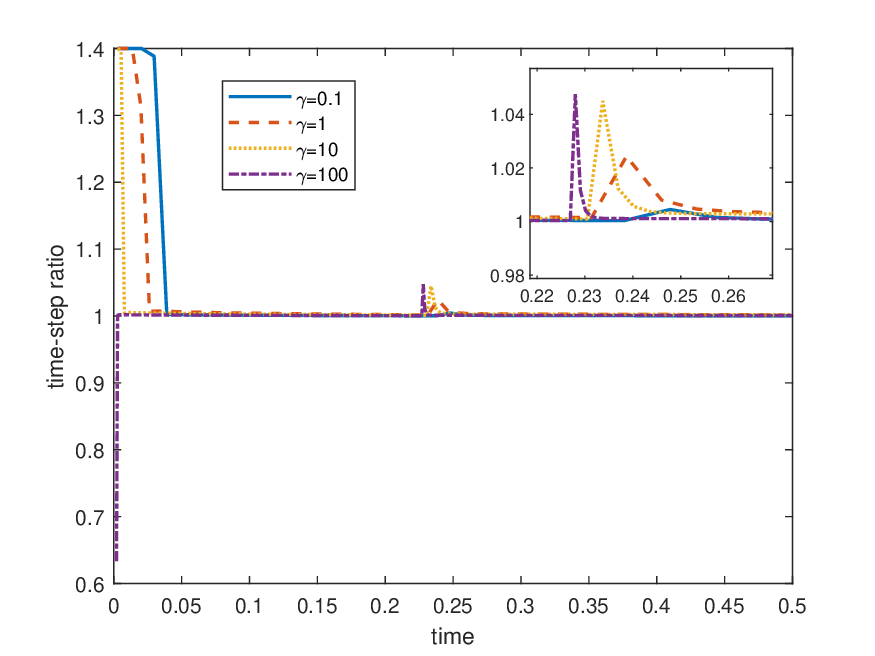}
	}
	\subfigure[Time-step  with \eqref{strategy2}]{
	\includegraphics[width=0.22\textwidth]{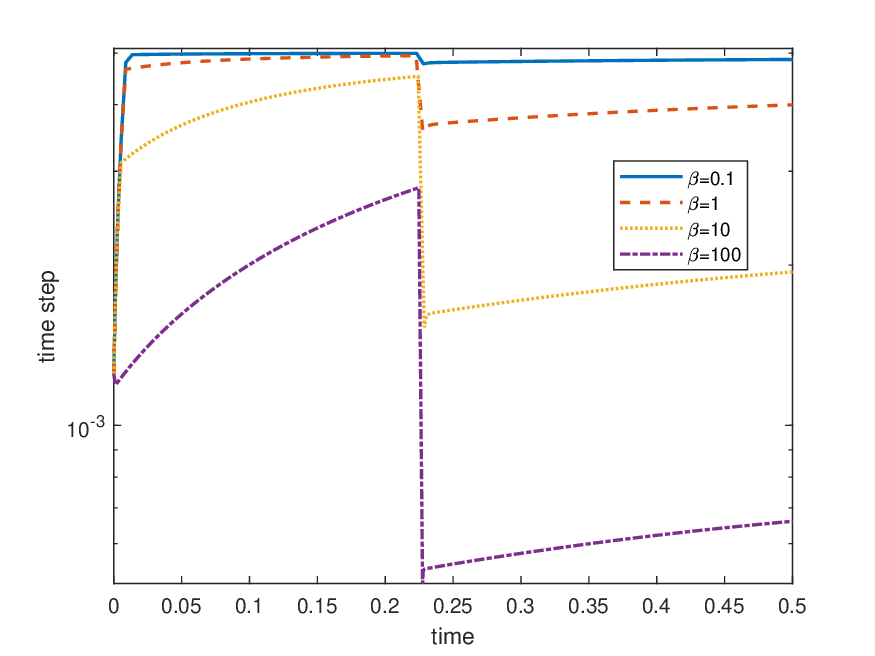}
}
\subfigure[Time-step ratio,  \eqref{strategy2}]{
	\includegraphics[width=0.22\textwidth]{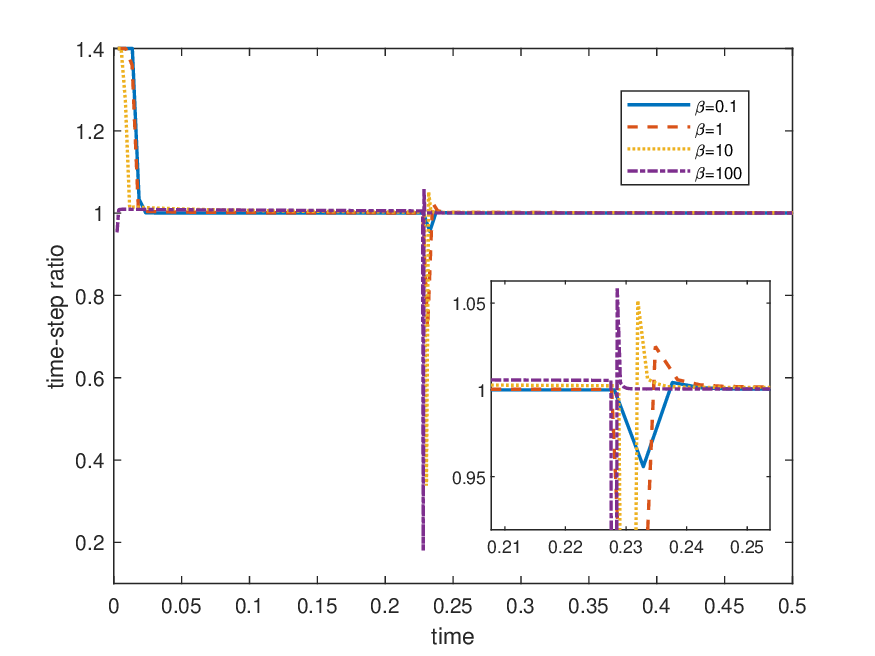}
}
	\caption{Numerical results solved by scheme \eqref{eq:scheme bdf2}-\eqref{schem:1-3} using the Algorithm \ref{algo,2}. Initial value \eqref{initial:wt} with $m=2$, $\theta=0.25$, $M_x=800$.}\label{fig:adaptive}
\end{figure}

Next, we use a fixed time-step to carry out numerical experiments with \eqref{initial:wt} and $\theta=0.25$, the results are displayed in Figure~\ref{fig:waitingtime}. The strategy \eqref{strategy1} is applied with $\gamma=10$, $\tau_{\max}=5\times1e-3$.  In the case where $m=2$, the free boundaries remain static during $0<t\le0.22$, and they begin to move at a finite speed after $t=0.22$. The waiting time is approximately $t=0.19$ when $m=2.5$. 
\begin{figure}[!htb]
	\centering
	\subfigure[$m=2$]{
		\includegraphics[width=0.31\textwidth]{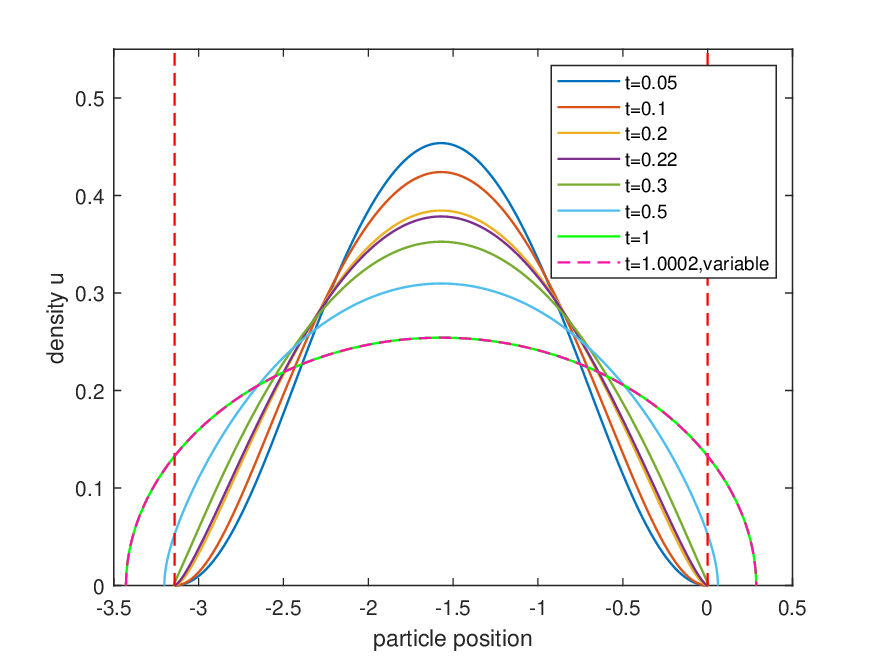}
	}
	\subfigure[$m=2.5$]{
	\includegraphics[width=0.31\textwidth]{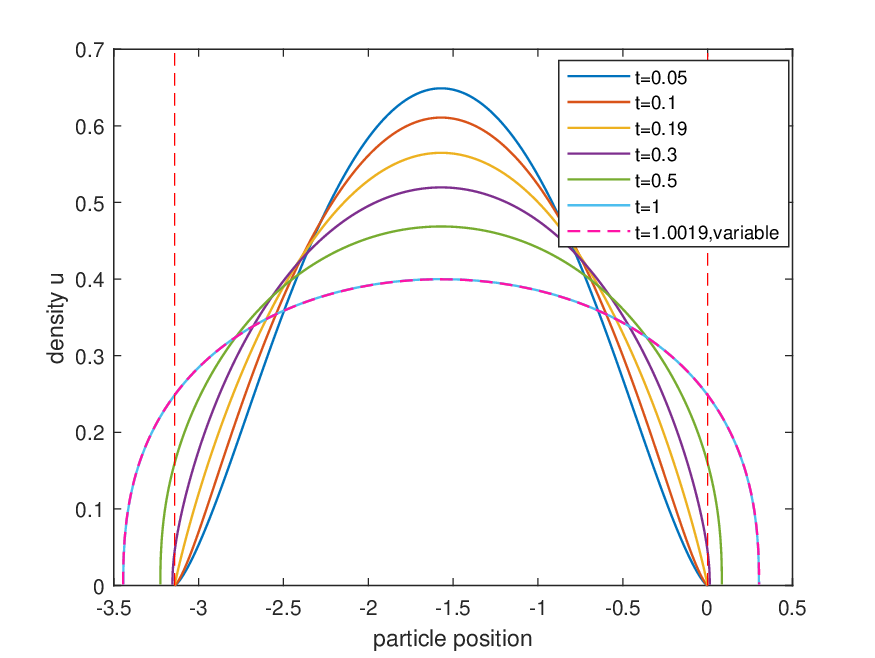}
}
\subfigure[energy]{
	\includegraphics[width=0.31\textwidth]{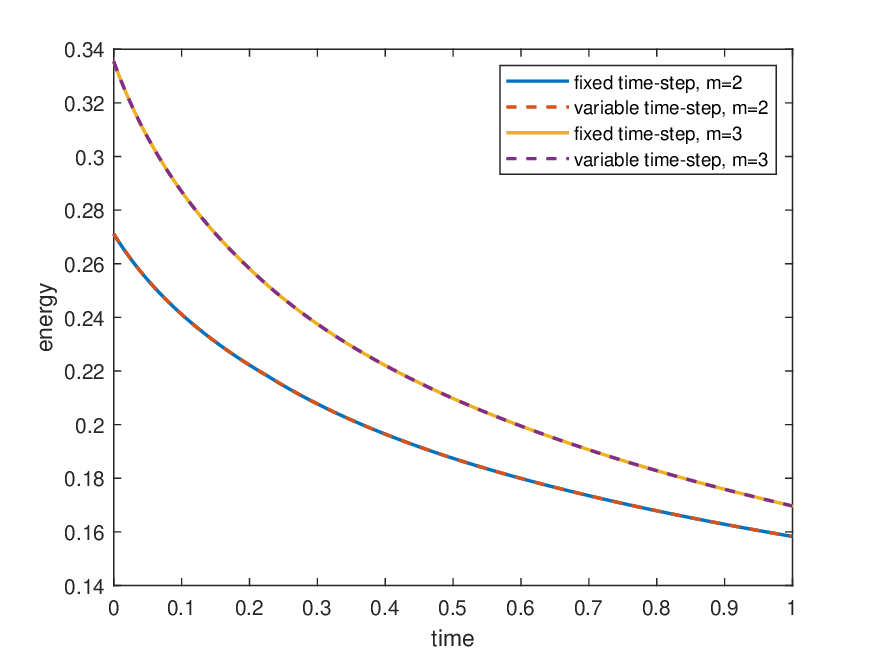}
}
	\caption{Numerical results for PME solved by \eqref{eq:scheme bdf2}-\eqref{schem:1-3} with \eqref{initial:wt}, $\theta=0.25$, $M_x=800$, $\delta t=\frac{1}{800}$.}\label{fig:waitingtime}
\end{figure}
	\subsubsection{Keller-Segel model}
	We simulate the Keller-Segel model with $E(u)=\int_{\Omega} u\ln u\mathrm{d}x+\frac{1}{2\pi}\int_{\Omega\times\Omega}\ln|x-y| u(x)u(y)\mathrm{d}x\mathrm{d}y$, and choose the following initial value:
	\begin{align}
	u_0(x)=\frac{C}{\sqrt{2\pi}}\exp^{-\frac{x^2}{2}}+10^{-8},\quad x\in[-15,15],
	\end{align}
	with $C=5\pi$. 
	We take $M_x=800$ to compute numerical solution using scheme \eqref{scheme:ac fully discrete} with time-step strategy \eqref{strategy2}, numerical results are shown in Figure~\ref{fig:ks}. 
	It can be observed that the proposed scheme is both energy stable and mass conservative. The time steps (blue line) and the corresponding time-step ratios (red line) are plotted in diagram (c) in Figure~\ref{fig:ks}. Specifically, the numerical solution exhibits a blow-up in finite time when the initial mass is large, reflecting the blow-up phenomenon.
	\begin{figure}[!htb]
	\centering
	\subfigure[Numerical solution]{
	\includegraphics[width=0.3\textwidth]{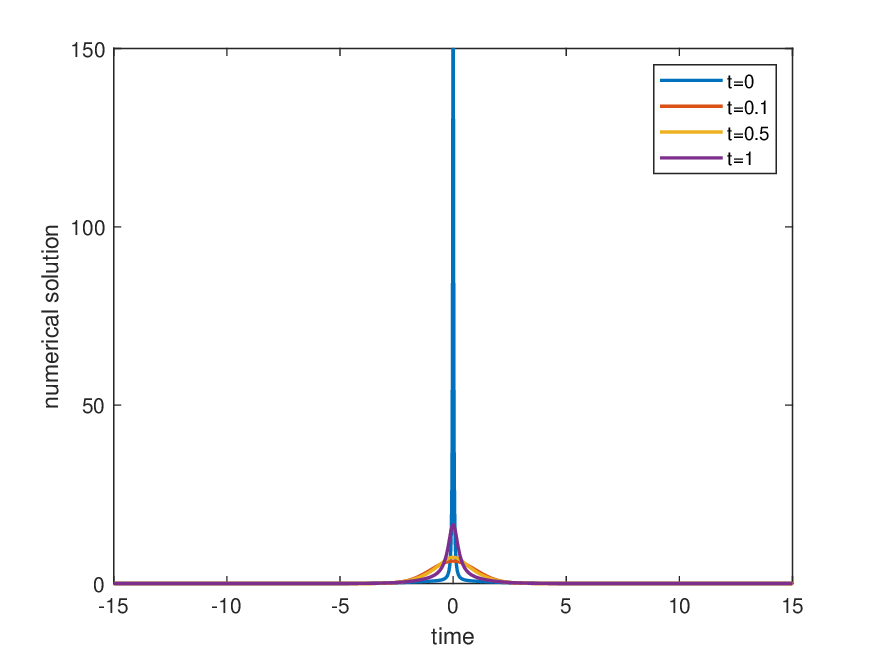}
	}
	\subfigure[Mass and energy]{
	\includegraphics[width=0.3\textwidth]{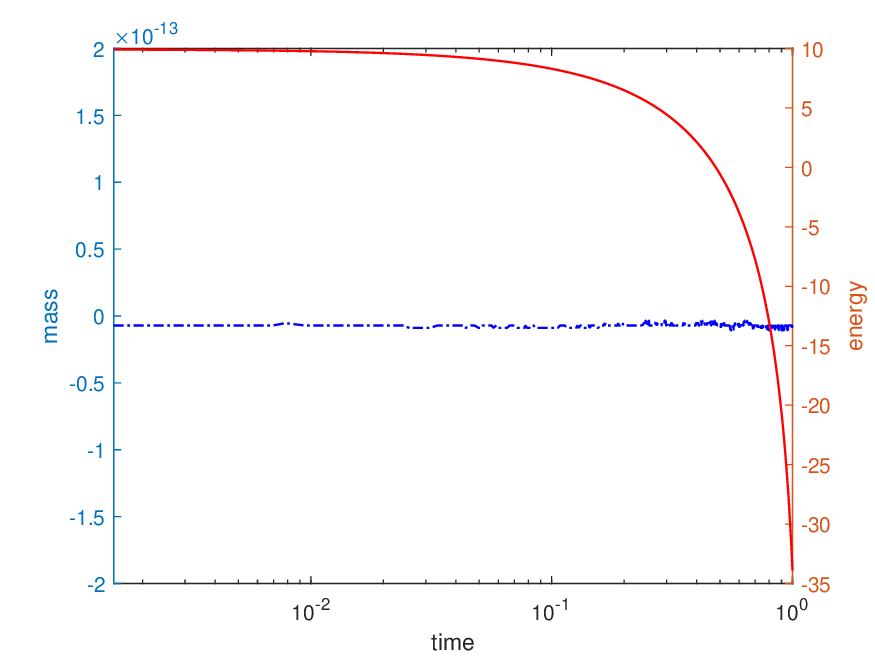}
	}
		\subfigure[Time step]{
		\includegraphics[width=0.3\textwidth]{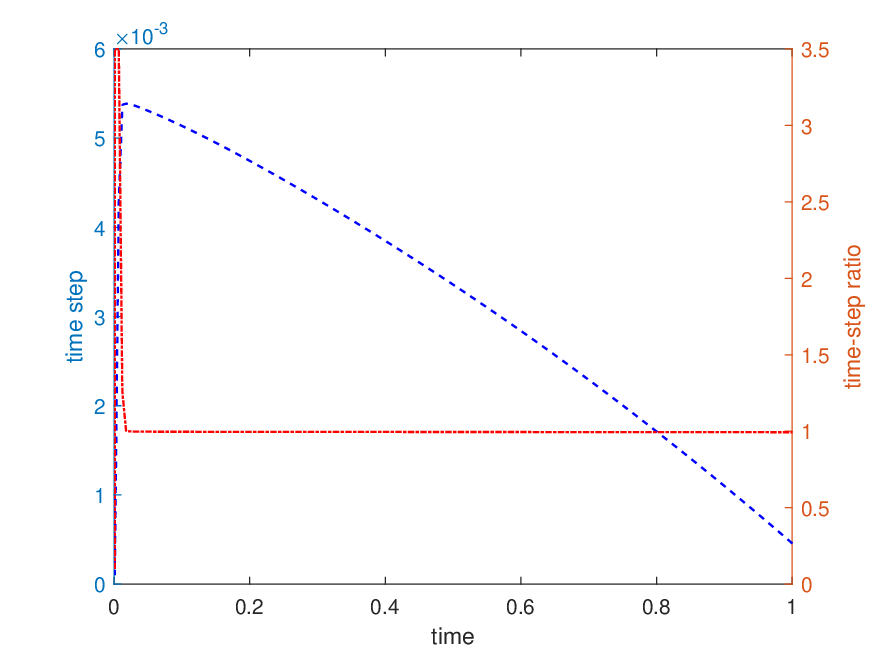}
	}
	\caption{Numerical results for Keller-Segel equation solved by  \eqref{eq:scheme bdf2}-\eqref{schem:1-3} using strategy \eqref{strategy2}, $M_x=800$,  $\beta=1e-2$, $\tau_{\min}=1e-4$, $\tau_{\max}=1e-2$, $r_{\text{user}}=3.5$.}\label{fig:ks}
	\end{figure}
	
\subsection{ Conservative models in two dimension}\label{sec:num in 2d}

Below we shall use the explicit scheme \eqref{scheme:2d explicit1}-\eqref{scheme:2d explicit2}  to simulate various  conservative models.
%

\subsubsection{Porous medium equation}

\textbf{Barenblatt Solution.} We consider the initial value \( u(x, y, 0) \) given by the Barenblatt solution:
\begin{align}\label{barenblatt}
	u(x,y,t) = \max\left( 0.1 - \frac{\kappa(m-1)}{4m} \frac{x^2 + y^2}{(t+1)^{\kappa}}, 0 \right)^{\frac{1}{m-1}},
\end{align}
where \( \kappa = \frac{1}{m} \). 
For the numerical simulations, we set \( M_x = M_y = 64 \). The simulations are performed using the scheme \eqref{scheme:2d explicit1}-\eqref{scheme:2d explicit2}, with both fixed time steps and variable time steps. The regularization term \( \varepsilon \tau_{n+1}^2 \Delta_{\bm{X}} {\bm{x}}^{n+1} \) with \( \varepsilon = 0.5 \) is used. For the variable time-step case, we apply the strategy \eqref{strategy2} with \( \tau_{\min} = 1e-4 \), \( \tau_{\max} = 1e-2 \), \( r_{\text{user}} = 1.25 \), and \( \beta = 1e-2 \). 

We plot  in Figure~\ref{fig:pme2d} and Figure~\ref{fig:pme2d,m5}  the numerical solutions  for $m=2$ at \( T = 2 \) and $m=5$ at \( T = 4 \).  
The time steps (blue line) and the corresponding time-step ratios (red line) are plotted in diagram (f) in Figure~\ref{fig:pme2d}. 
It can be observed that the scheme \eqref{scheme:2d explicit1}-\eqref{scheme:2d explicit2} effectively captures the trajectory movements and free boundaries.

\begin{figure}[!htb]
	\centering
	\subfigure[$t=2$, $\delta t=0.01$]{
		\includegraphics[width=0.31\textwidth]{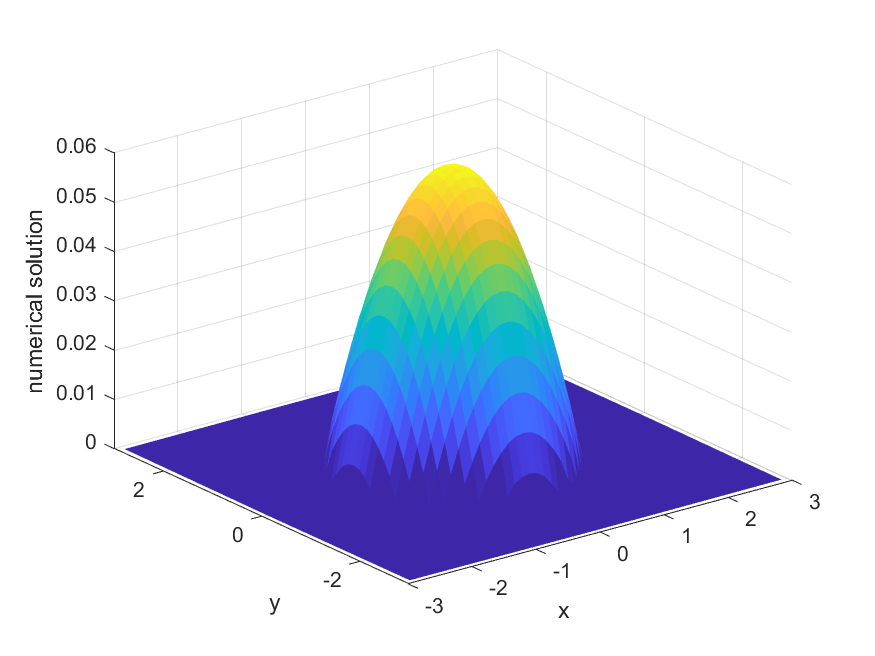}
	}
	\subfigure[ $t=2$, variable steps]{
		\includegraphics[width=0.31\textwidth]{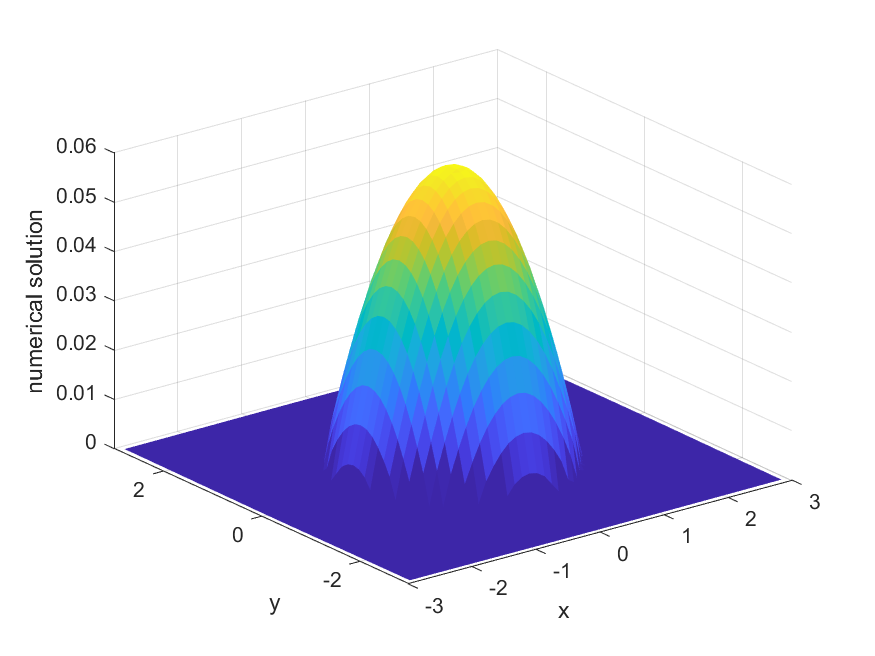}
	}
	\subfigure[Exact solution at $t=2$]{
		\includegraphics[width=0.31\textwidth]{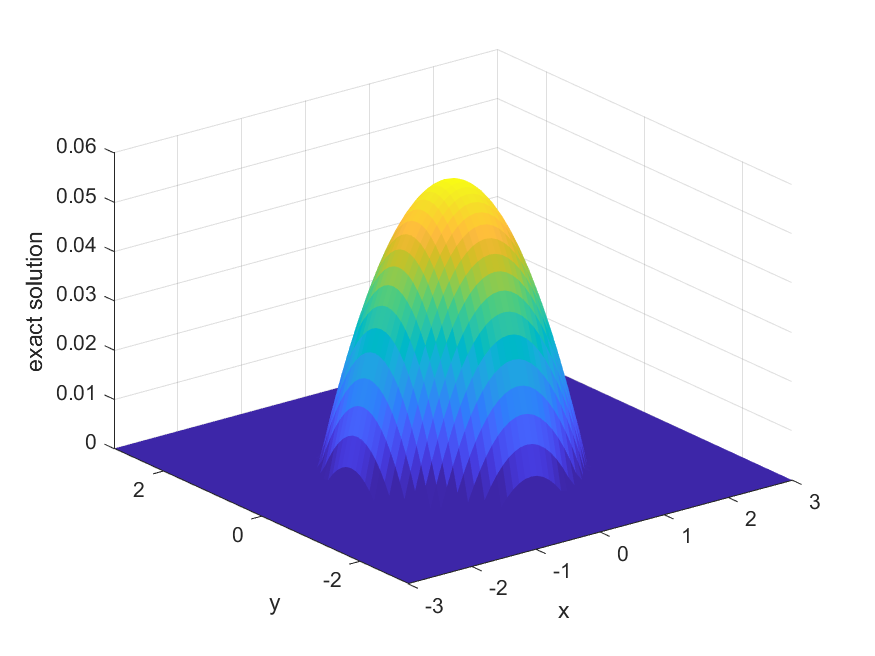}
	}
	\subfigure[$t=2$, $\delta t=0.01$]{
		\includegraphics[width=0.31\textwidth]{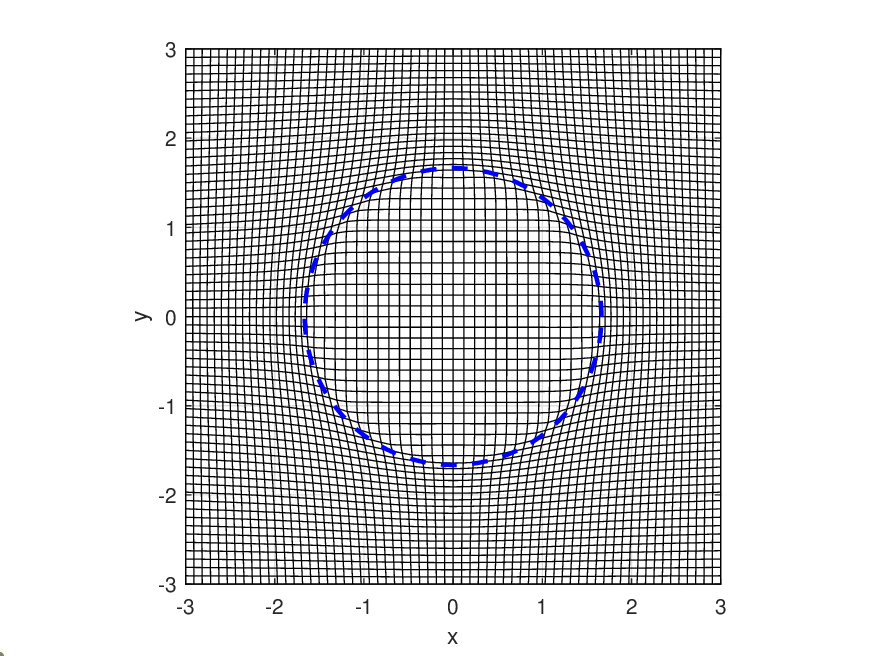}
	}
	\subfigure[$t=2$, variable steps]{
		\includegraphics[width=0.31\textwidth]{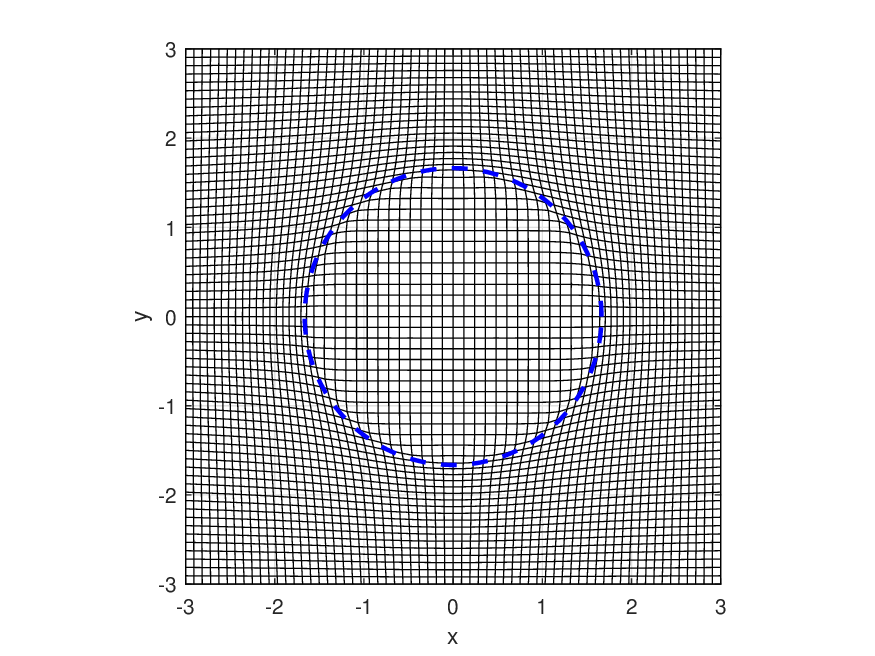}
	}
	\subfigure[Time step]{
		\includegraphics[width=0.31\textwidth]{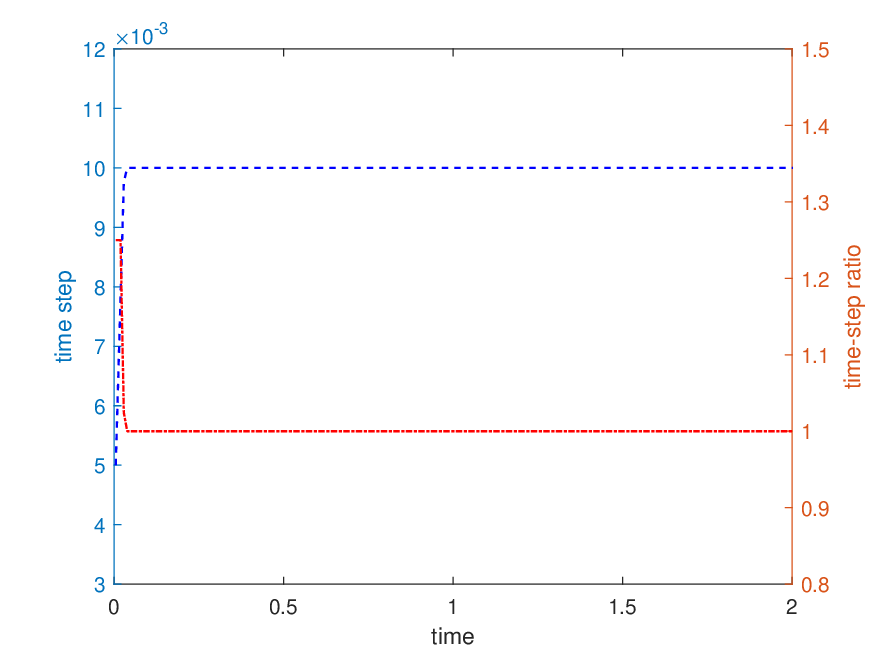}
	}
	\caption{Numerical solution for PME solved by \eqref{scheme:2d explicit1}-\eqref{scheme:2d explicit2} with \eqref{barenblatt} using strategy \eqref{strategy2}, $m=2$, $M_x=M_y=64$, $T=2$, $\varepsilon=0.5$, $\tau_{\min}=1e-4$, $\tau_{\max}=1e-2$, $r_{\text{user}}=1.25$,	$\beta=1e-2$. The blue dashed circle is the exact interface for \eqref{barenblatt} at $T=2$. }\label{fig:pme2d}
\end{figure}

\begin{figure}[!htb]
	\centering
	\subfigure[$t=4$]{
		\includegraphics[width=0.31\textwidth]{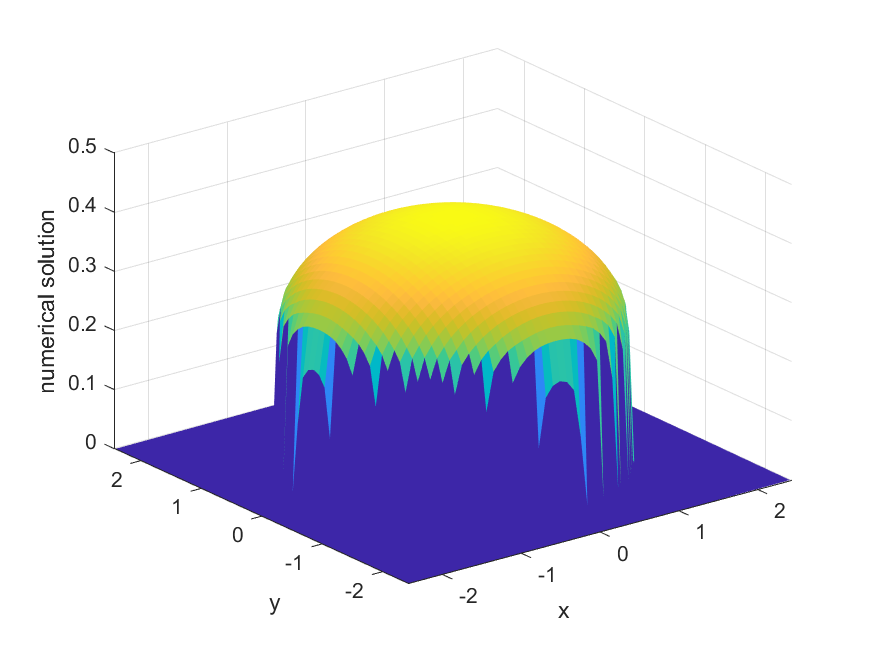}
	}
	\subfigure[Trajectory at $t=4$]{
		\includegraphics[width=0.31\textwidth]{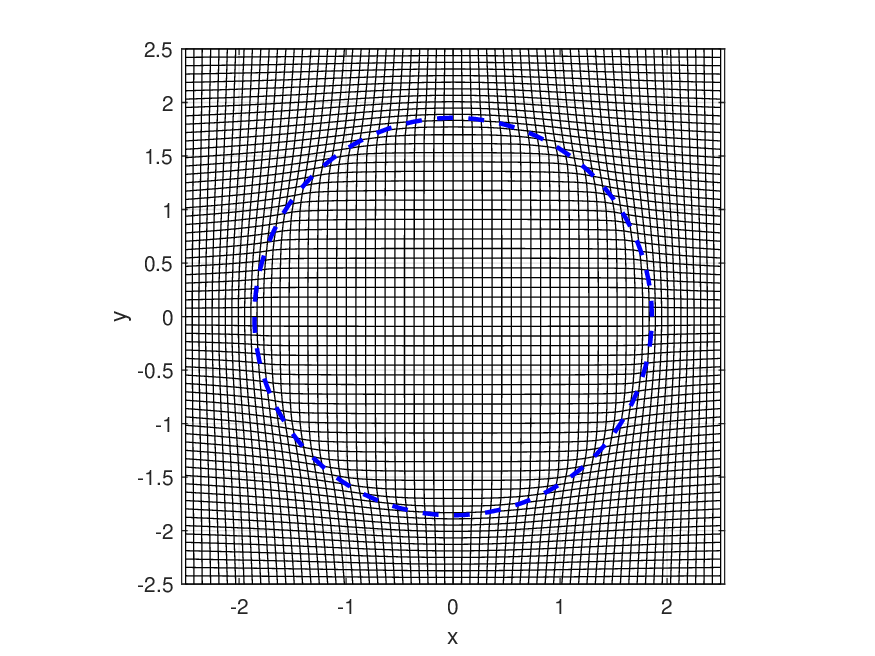}
	}
	\subfigure[ Energy plot]{
		\includegraphics[width=0.31\textwidth]{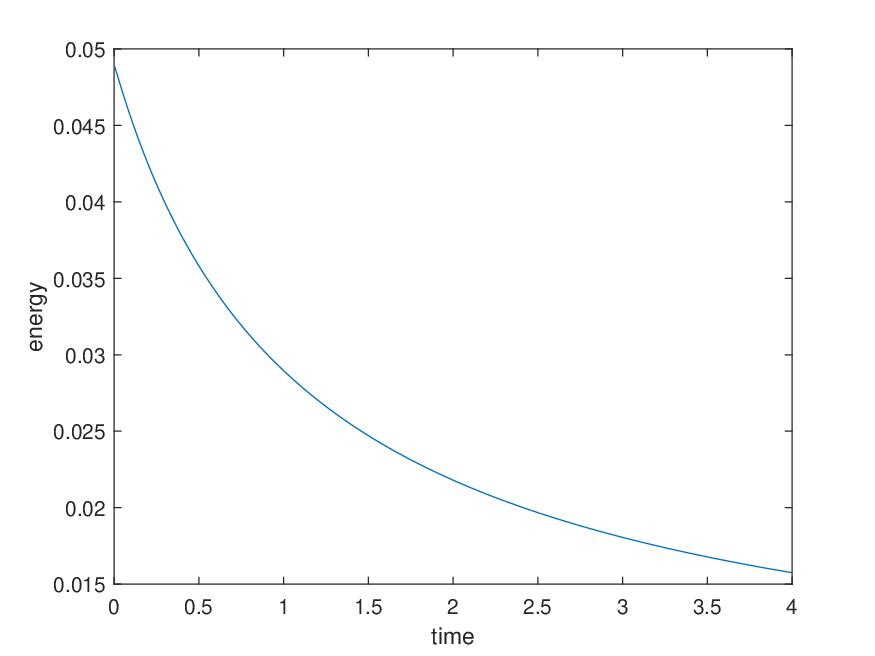}
	}
	\caption{Numerical solution for PME solved by \eqref{scheme:2d explicit1}-\eqref{scheme:2d explicit2} 
	with \eqref{barenblatt} using strategy \eqref{strategy2}, $m=5$, $M_x=M_y=64$, $T=4$, 
	$\varepsilon=40$, $\tau_{\min}=1e-4$, $\tau_{\max}=1e-2$, $r_{\text{user}}=1.25$,	
	$\beta=1e-2$. The blue dashed circle is the exact interface for \eqref{barenblatt} at $T=4$.}\label{fig:pme2d,m5}
\end{figure}

{\bf Non-radial case.} We now consider the following non-radial initial value:
 	\begin{equation}\label{ini:nonradial}
	\rho_0(x,y)=\begin{cases}
		25(0.25^2-(\sqrt{x^2+y^2}-0.75)^2)^{\frac{3}{2}}, &\sqrt{x^2+y^2}\in[0.5,1]\ \text{and}\ (x<0\ \text{or}\ y<0),\\
		25(0.25^2-x^2-(y-0.75)^2)^{\frac{3}{2}}, &x^2+(y-0.75)^2\le 0.25^2\ \text{and}\ x\ge 0,\\
		25(0.25^2-(x-0.75)^2-y^2)^{\frac{3}{2}}, &(x-0.75)^2+y^2\le 0.25^2\ \text{and}\ y\ge 0,\\
		0,\ &\text{otherwise},
	\end{cases}
\end{equation}
which has a partial donut-shaped support \cite{liu2020lagrangian}. We conduct numerical experiments using the scheme \eqref{scheme:2d explicit1}-\eqref{scheme:2d explicit2} with the fixed time step and the regularization term $\varepsilon\tau_{n+1}^2\Delta_{\bm X}{\bm x}^{n+1}$ where $\varepsilon=100$. The numerical results are shown in Figure~\ref{fig:pme2d,nonradial}, which displays the evolution of solution and trajectory movements. It can be observed that the proposed method performs well for solutions with complex support.
\begin{figure}[!htb]
	\centering
	\subfigure[$t=0.1$]{
		\includegraphics[width=0.31\textwidth]{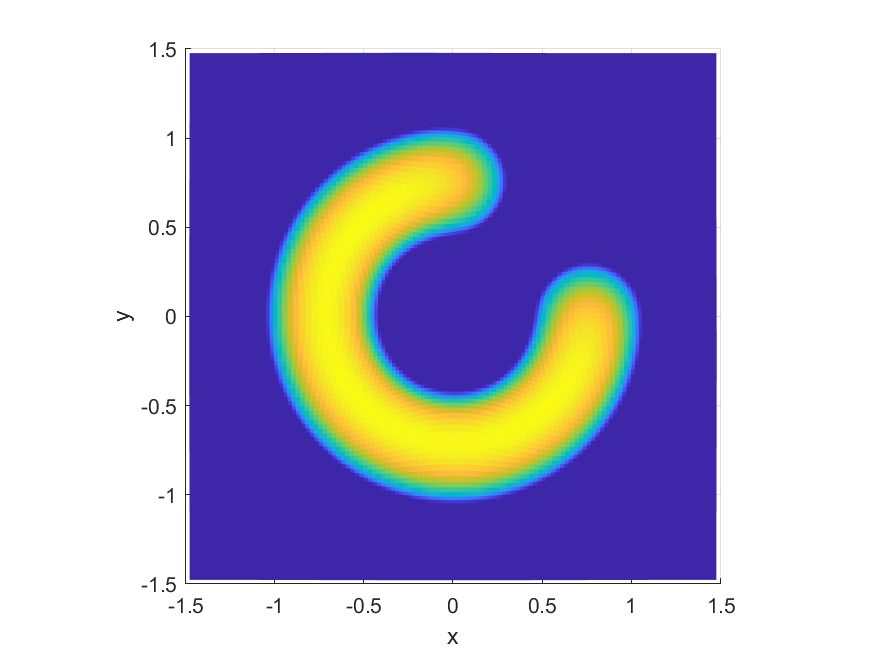}
	}
		\subfigure[$t=0.3$]{
		\includegraphics[width=0.31\textwidth]{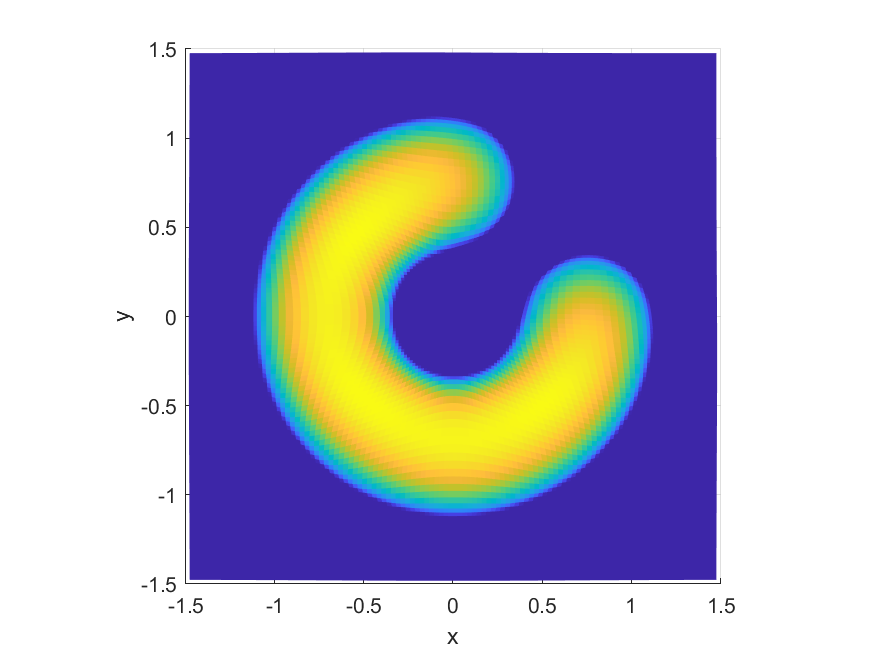}
	}
		\subfigure[$t=0.6$]{
		\includegraphics[width=0.31\textwidth]{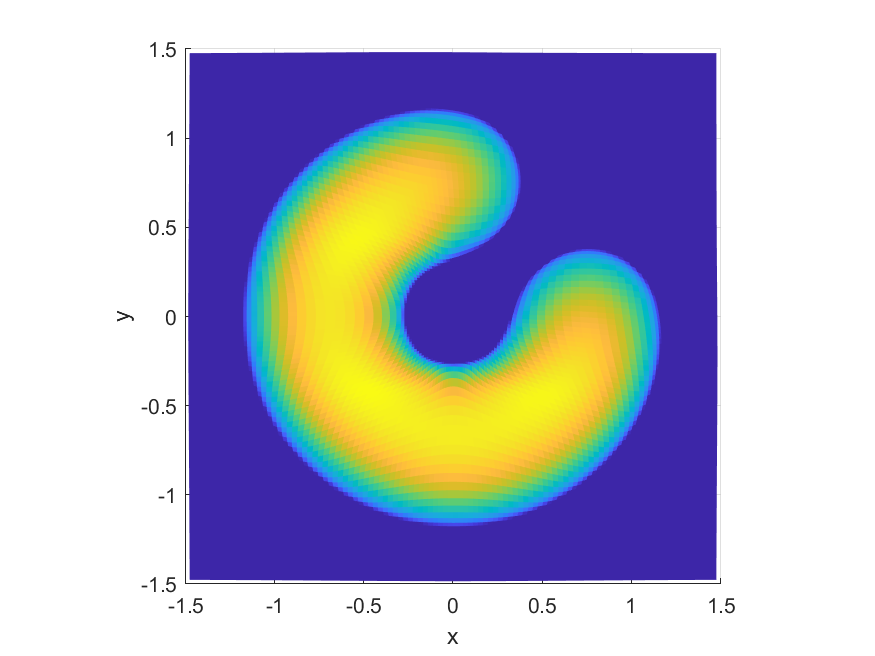}
	}
	\subfigure[ Trajectory at $t=0.1$]{
		\includegraphics[width=0.31\textwidth]{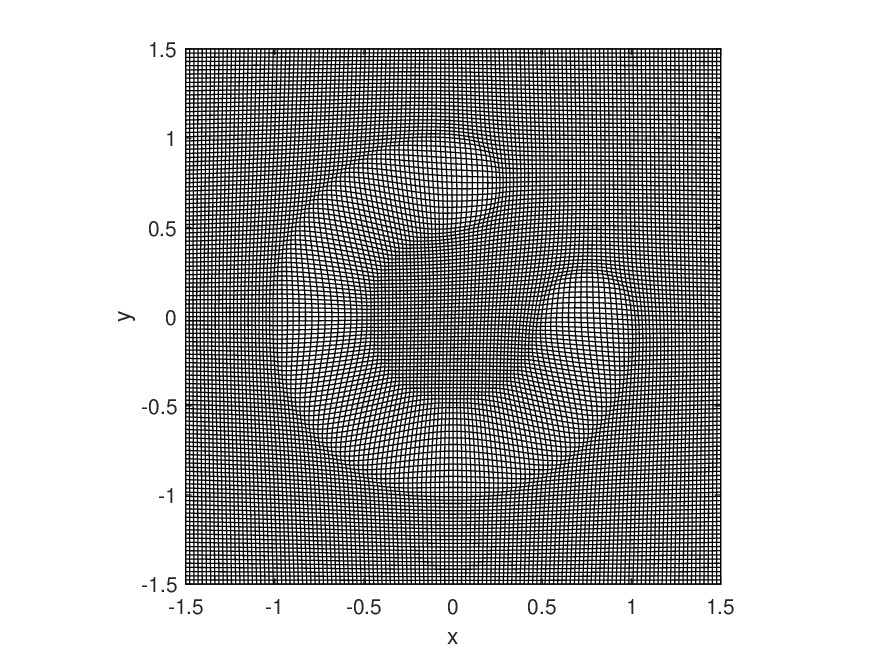}
	}
\subfigure[ Trajectory at $t=0.3$]{
	\includegraphics[width=0.31\textwidth]{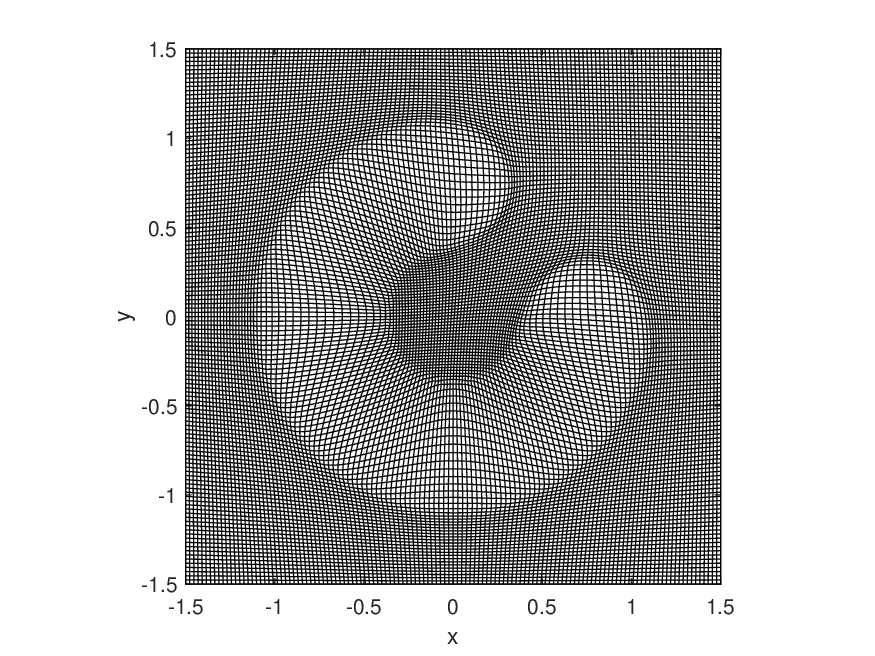}
}
\subfigure[ Trajectory at $t=0.6$]{
	\includegraphics[width=0.31\textwidth]{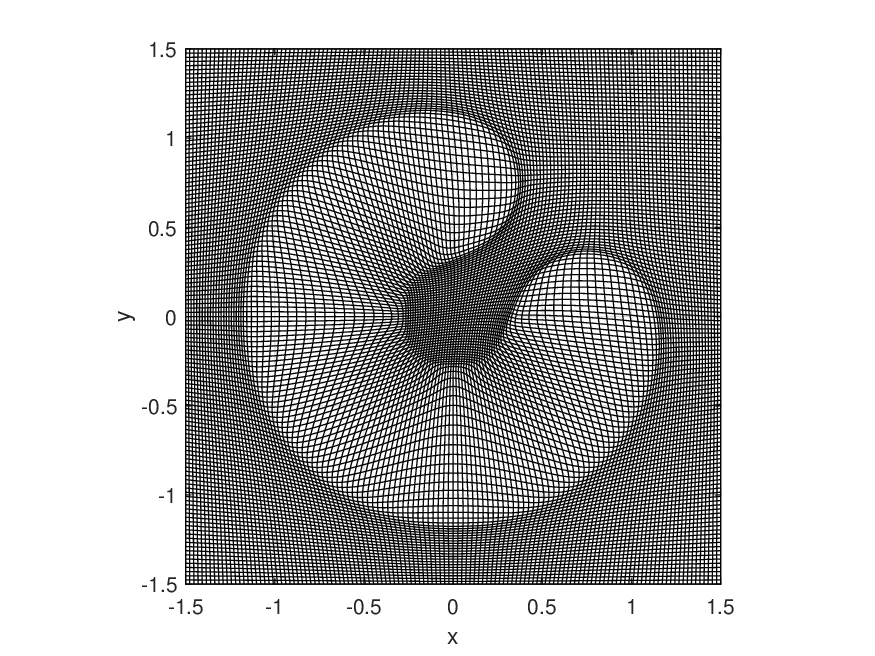}
}
	\caption{Numerical results for PME solved by \eqref{scheme:2d explicit1}-\eqref{scheme:2d explicit2} with fixed time step. Initial value \eqref{ini:nonradial}, $m=3$, $M_x=M_y=128$, $\tau=1e-3$, $\varepsilon=100$.}\label{fig:pme2d,nonradial}
\end{figure}
\subsubsection{Keller-Segel model}
Consider the Keller-Segel model given by
\begin{align*}
	\partial_t \rho = \nabla \cdot (\rho \nabla W * \rho) + \nu \Delta \rho^m, \quad m \ge 1,
\end{align*}
where \( W(x) = \frac{1}{2\pi} \ln |x| \) and \( \nu = 1 \). We choose \( m = 1 \) and \( m = 2 \) and use  the scheme  \eqref{scheme:2d explicit1}-\eqref{scheme:2d explicit2} with the regularization term \( \varepsilon \tau_{n+1}^2 \Delta_{\bm{X}} {\bm{x}}^{n+1} \) and adaptive time step strategy in \eqref{strategy2}. 

The following initial conditions, with \( C = 1 \) or \( C = 5\pi \), are considered:
\begin{align}\label{initial value}
	u_0(x,y) = C \exp^{-x^2 - y^2}, \quad (x,y) \in [-5,5] \times [-5,5].
\end{align}

The numerical results are presented in Figures \ref{fig:ks2d,C1,0.2} and \ref{fig:ks2d,C5pi,0.15}.  
The time steps (blue line) and the corresponding time-step ratios (red line) are plotted in diagram (c) and diagram (f) in Figures \ref{fig:ks2d,C1,0.2} and \ref{fig:ks2d,C5pi,0.15}. 
We observe in particular that the  trajectory movements are consistent with the interface positions.  
The results also  indicate that, when \( m = 1 \), a blow-up phenomenon occurs if the initial mass is large, while diffusion prevails when the initial mass is small. 
In contrast, for \( m = 2 \), the solution eventually converges to a bump with both initial conditions.


	\begin{figure}[!htb]
	\centering
	\subfigure[Solution, $t=0.2$, $m=1$]{
		\includegraphics[width=0.31\textwidth]{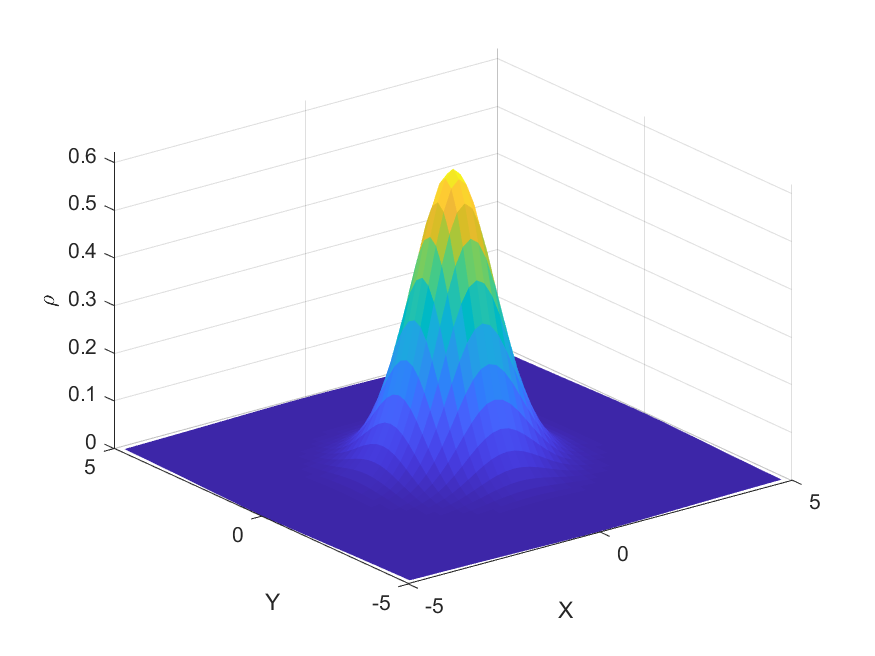}
}
\subfigure[Trajectory, $t=0.2$, $m=1$]{
	\includegraphics[width=0.31\textwidth]{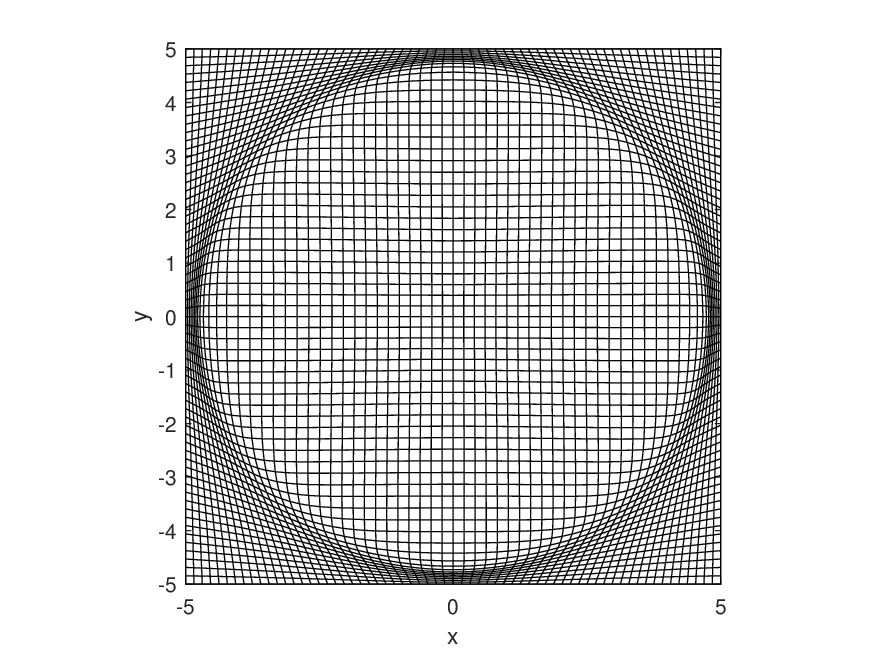}
	}
	\subfigure[Time step, $m=1$]{
\includegraphics[width=0.31\textwidth]{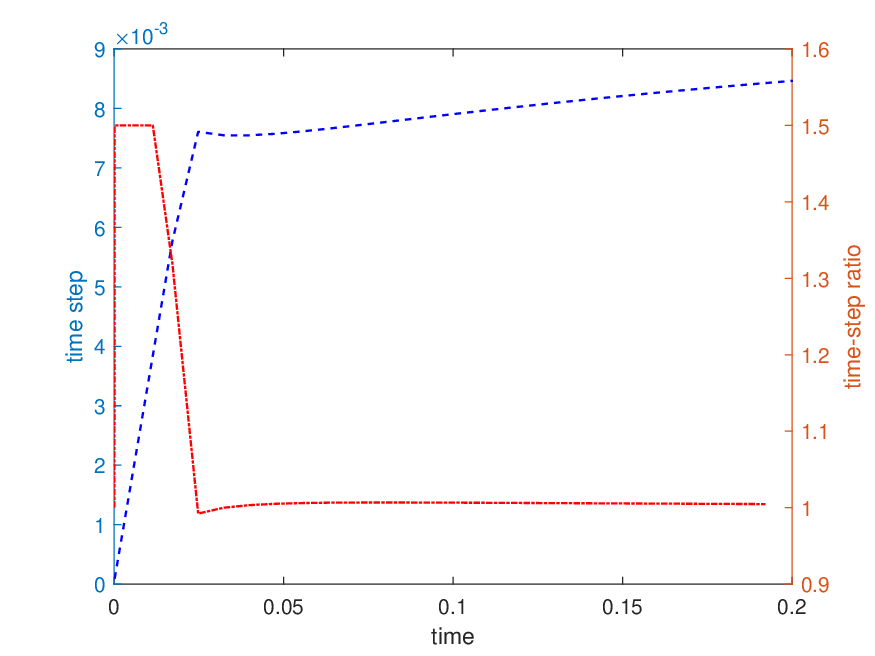}
}
\subfigure[Solution, $t=1.5$, $m=2$]{
\includegraphics[width=0.31\textwidth]{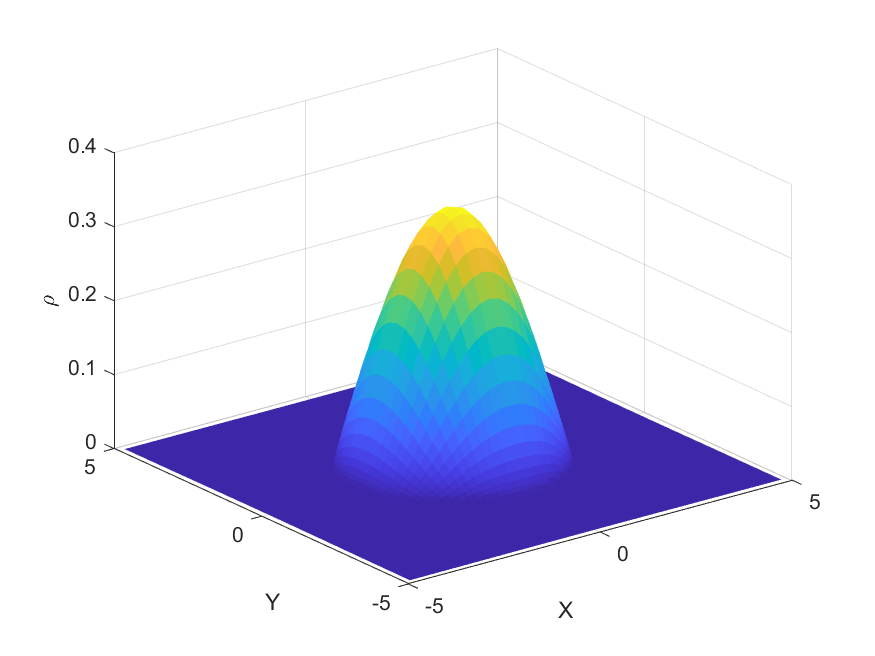}
}
\subfigure[Trajectory, $t=1.5$, $m=2$]{
\includegraphics[width=0.31\textwidth]{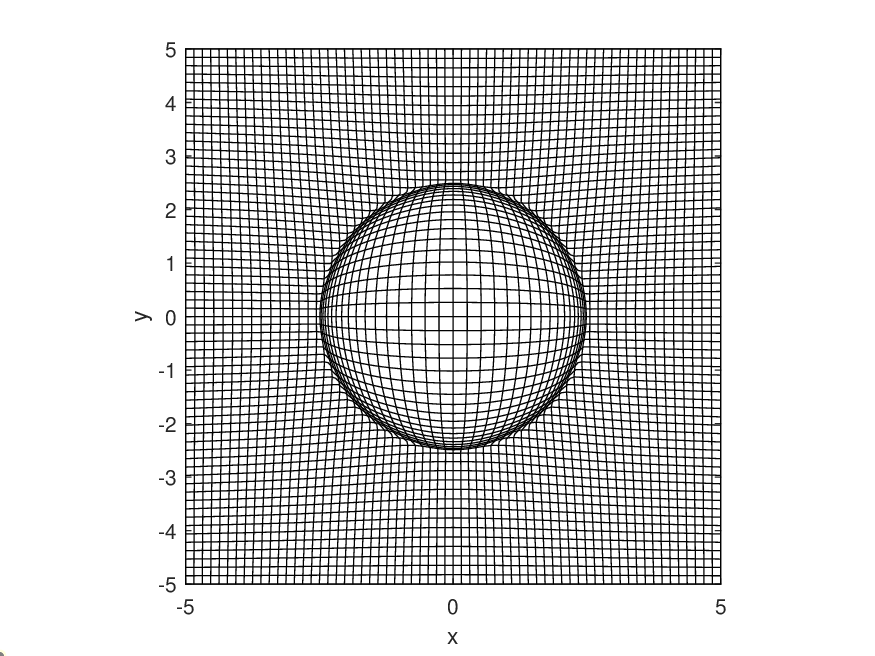}
}
\subfigure[Time step, $m=2$]{
\includegraphics[width=0.31\textwidth]{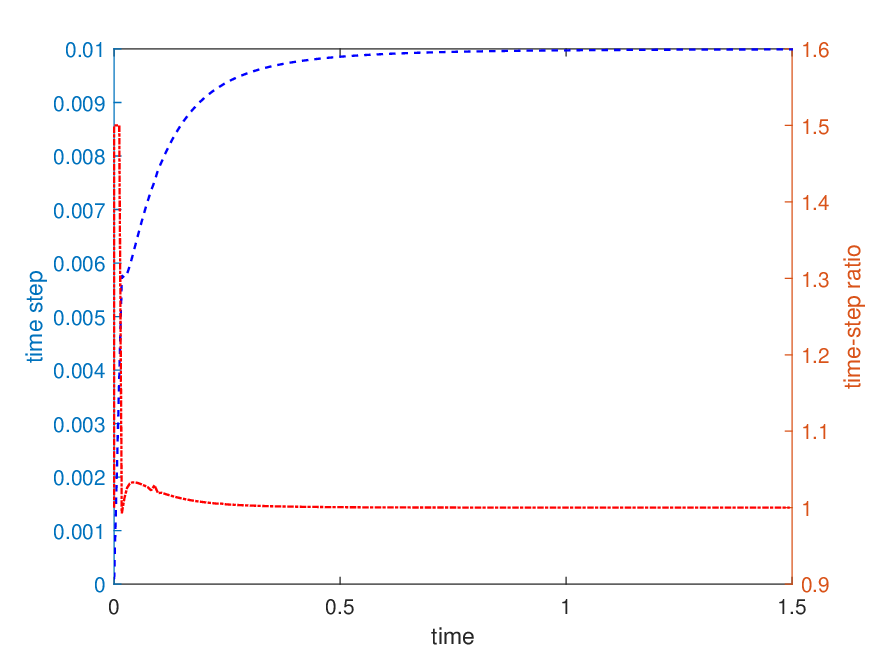}
}
\caption{Numerical results for Keller-Segel equation solved by \eqref{scheme:2d explicit1}-\eqref{scheme:2d explicit2} with \eqref{initial value} using  strategy \eqref{strategy2}, $C=1$, $M_x=M_y=64$, and $\beta=1e-2$, $\tau_{\min}=1e-4$, $\tau_{\max}= 1e-2$, $r_{\text{user}}=1.5$. $m=1$: $\varepsilon=0.1$; $m=2$: $\varepsilon=0.01$.}\label{fig:ks2d,C1,0.2}
\end{figure}
\begin{figure}[!htb]
\centering
\subfigure[Solution, $t=0.15$, $m=1$]{
\includegraphics[width=0.31\textwidth]{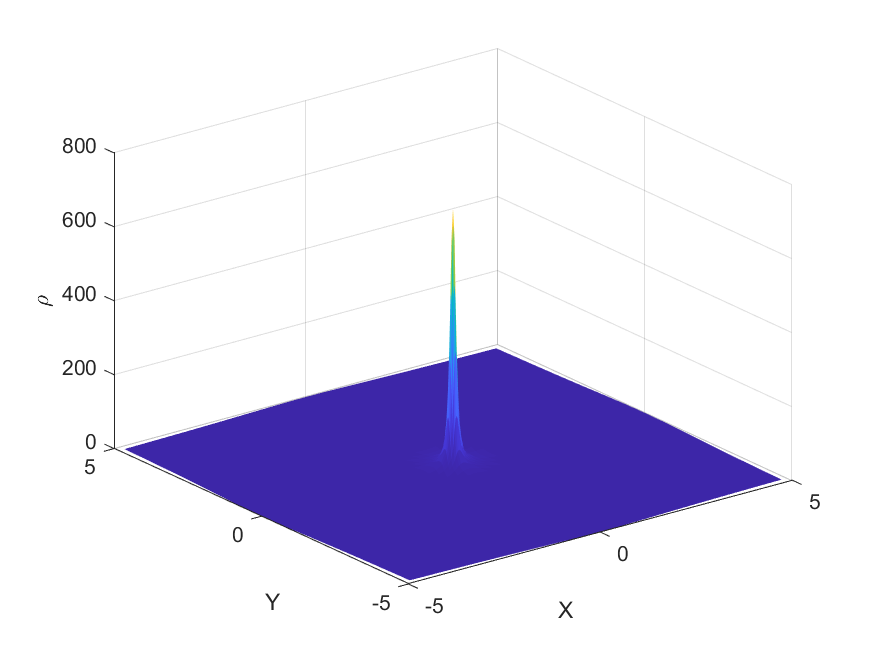}
}
\subfigure[Trajectory, $t=0.15$, $m=1$]{
\includegraphics[width=0.31\textwidth]{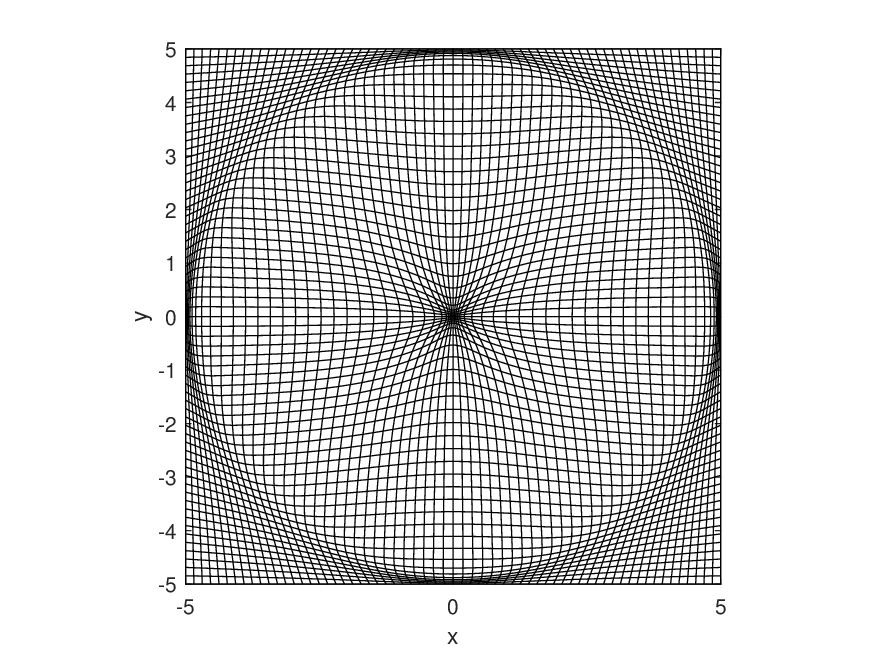}
}
\subfigure[Time step, $m=1$]{
\includegraphics[width=0.31\textwidth]{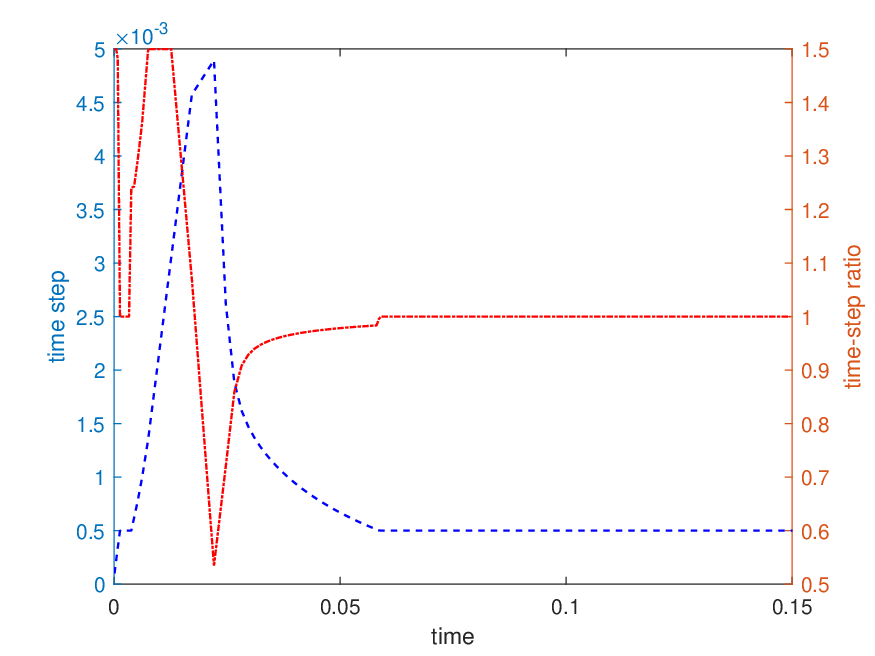}
}
\subfigure[Solution, $t=0.15$, $m=2$]{
\includegraphics[width=0.31\textwidth]{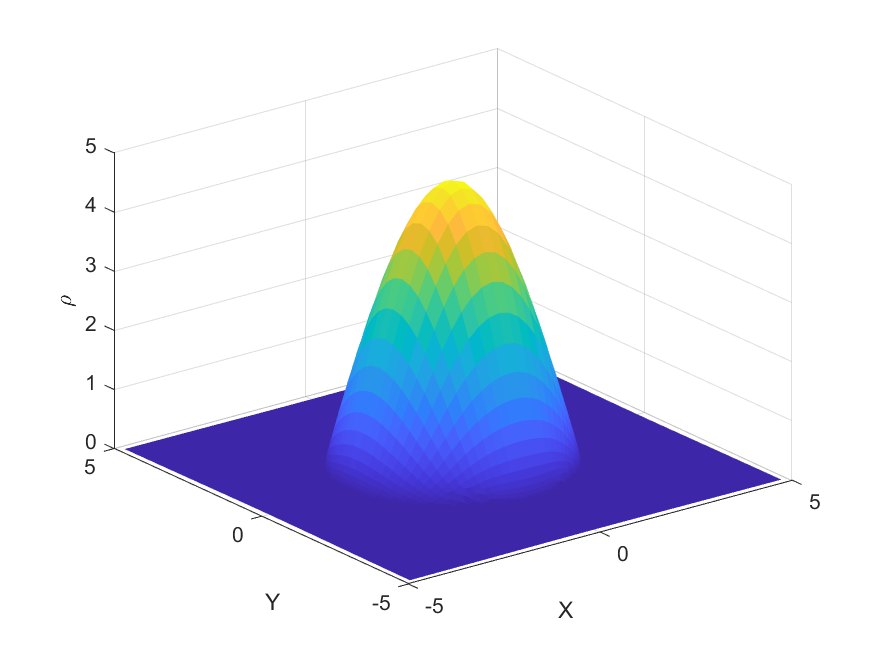}
}
\subfigure[Trajectory, $t=0.15$, $m=2$]{
\includegraphics[width=0.31\textwidth]{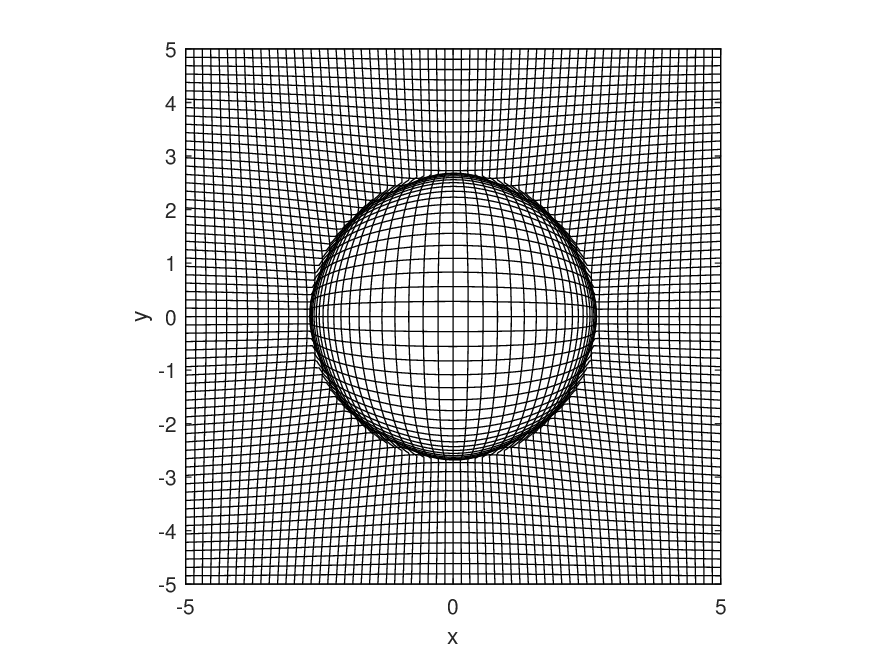}
}
\subfigure[Time step, $m=2$]{
\includegraphics[width=0.31\textwidth]{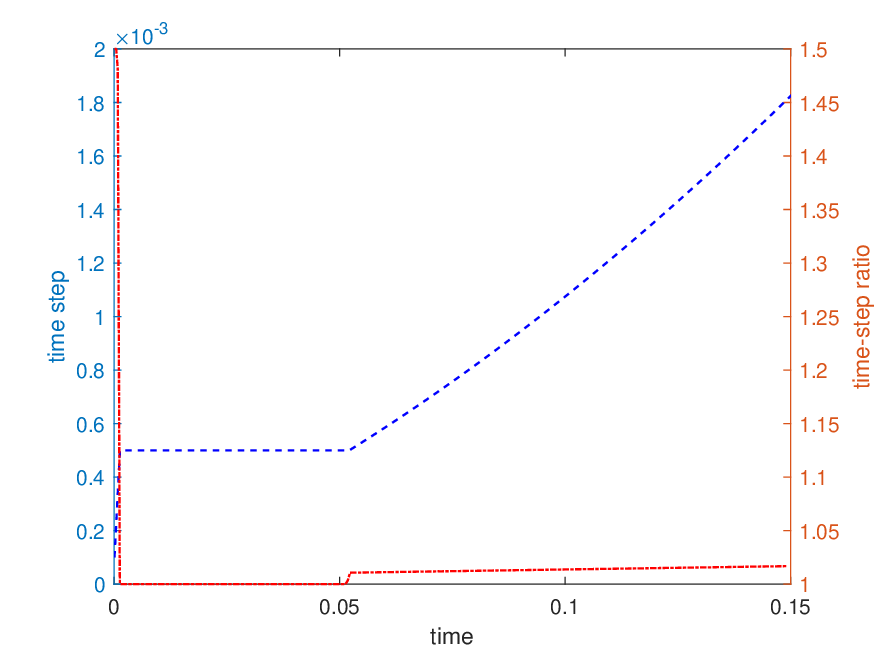}
}
\caption{Numerical results for Keller-Segel solved by \eqref{scheme:2d explicit1}-\eqref{scheme:2d explicit2} with \eqref{initial value} using  strategy \eqref{strategy2}, $C=5\pi$, $M_x=M_y=64$,  $\varepsilon=10$,  $\beta=1e-2$, $\tau_{\min}=5\times1e-4$, $\tau_{\max}=5\times 1e-2$, $r_{\text{user}}=1.5$.}\label{fig:ks2d,C5pi,0.15}
\end{figure}

\section{Concluding remarks}
We developed in this paper two adaptive time-stepping Lagrangian methods for gradient flows based on  different flow dynamic approaches. For the non-conservative models, we designed a modified second-order scheme which naturally preserves the MBP and ensures energy dissipation under a mild condition on the maximum time-step ratio. For the conservative models written as Wasserstein gradient flows, we constructed a mass-conservative method  which is proved to be energy dissipative and positivity-preserving under a mild condition on the maximum time-step ratio. The proposed Lagrangian methods offer several advantages, including the enhanced ability to capture sharp interfaces,  simulate trajectory movements, and address problems involving singularities and free boundaries. 

It should be noted that the condition on the time-step ratio is sufficient for the theoretical analysis, but not necessarily optimal in the practice. An interesting question for further exploration is whether   these schemes can maintain stability under a less restrictive condition on the time-step ratio.


\section*{Acknowledgments} 
Liu is supported by NSFC 123B2015, 
Cheng is supported by NSFC 12301522, 
Chen is supported by NSFC 12471369 and  NSFC 12241101, 
and Shen is supported by NSFC W2431008 and 12371409.

	\bibliographystyle{abbrv}
{\small
	\bibliography{ref.bib}
}

\end{document}